\documentclass[12pt]{article}
\usepackage{geometry}
\usepackage{titlesec}

\titleformat{\section}
  {\normalfont\fontsize{14}{15}\bfseries}{\thesection}{1em}{}
\titleformat{\subsection}
{\normalfont\fontsize{14}{15}\bfseries}{\thesubsection}{1em}{}

\newcommand{\footremember}[2]{%
   \footnote{#2}
    \newcounter{#1}
    \setcounter{#1}{\value{footnote}}%
}
\providecommand{\keywords}[1]
{
  \textbf{\textit{Keywords---}} #1
}

\usepackage{framed,multirow}

\usepackage{amssymb}
\usepackage{latexsym}

\usepackage{amsmath,amsthm}      % use Times fonts if available on your TeX system
\usepackage{enumerate}
\numberwithin{equation}{section}
\newtheorem{theorem}{Theorem}

\newtheorem{corl}{Corollary}
\newtheorem{defn}{Definition}
\usepackage{graphicx}
\usepackage{caption}
\usepackage{subcaption}%
\usepackage{color}
\usepackage{hyperref}

\usepackage{url}
\usepackage{xcolor}
\definecolor{newcolor}{rgb}{.8,.349,.1}

\begin{document}

%#\verso{S.Kivva}

%#\begin{frontmatter}

\title{\Large{Flux-corrected transport for scalar hyperbolic conservation laws and convection-diffusion equations by using linear programming }}%#\tnoteref{S.Kivva}}%
%#\tnotetext[S.Kivva]{This is an example for title footnote coding.}

\author{Sergii Kivva\footremember{alley}{Institute of Mathematical Machines and System Problems, National Academy of Sciences, Ukraine}, 
%e-mail: skivva@gmail.com
} 
\date{\Large}

%#\author{Sergii \snm{Kivva} \corref{cor1}}
%#\cortext[cor1]{Corresponding author: 
%  Tel.: +38-044-526-6285;  
%  fax: +38-044-526-6187;
%#  E-mail address: skivva@gmail.com}
%\author[1]{Given-name2 \snm{Surname2}\fnref{fn1}}
%\fntext[fn1]{This is author footnote for second author.}  
%\author[2]{Given-name3 \snm{Surname3}}
%% Third author's email
%\ead{author3@author.com}
%\author[2]{Given-name4 \snm{Surname4}}

%#\address{Institute of Mathematical Machines and System Problems, National Academy of Sciences, Ukraine}
%\address[2]{Affiliation 2, Address, City and Postal Code, Country}

%#\received{}
%#\finalform{}
%#\accepted{}
%#\availableonline{}

\maketitle

\begin{abstract}
%%%
Flux-corrected transport (FCT) is one of the flux limiter methods. Unlike the total variation diminishing methods, obtaining the known FCT formulas for computing flux limiters is not quite transparent, and their transformation is not obvious when the original differential operator changes. 
We propose a novel formal mathematical approach to design flux correction for weighted hybrid difference schemes by using linear programming. The hybrid scheme is a linear combination of a monotone scheme and a high order scheme. The determination of maximal antidiffusive fluxes is treated as an optimization problem with a linear objective function. To obtain constraints for the optimization problem, inequalities that are valid for the monotone difference scheme are applied to the hybrid difference scheme. The numerical solution of the nonlinear optimization problem is reduced to the iterative solution of linear programming problems. A nontrivial approximate solution of the corresponding linear programming problem can be treated as the required flux limiters. We present flux correction formulas for scalar hyperbolic conservation laws and convection-diffusion equations. The designed flux-corrected transport for scalar hyperbolic conservation laws yields entropy solutions. Numerical results are presented.
%%%%
\end{abstract} 

\keywords{flux corrected transport, linear programming, scalar hyperbolic conservation law, convection-diffusion equation, entropy solution, difference scheme}

%#\begin{keyword}
% MSC codes here, in the form: \MSC code \sep code
% or \MSC[2008] code \sep code (2000 is the default)
%#\MSC 65M06\sep 65M08
% Keywords
%#\KWD flux corrected transport\sep linear programming\sep scalar hyperbolic conservation laws
%#\sep convection-diffusion equation\sep entropy solution \sep difference scheme
%#\end{keyword}

%#\end{frontmatter}

%\linenumbers

  
\section{Introduction} \label{Sec1}
At present, schemes with flux-corrected transport (FCT) and schemes with total variation diminishing (TVD) are widely used in the simulation of various physical processes among difference schemes on a fixed stencil. Within these schemes, the flux at the cell interface is computed as a weighted combination of fluxes of a monotone low-order scheme and a high-order scheme. The basic idea is to switch between high-order scheme and positivity preserving low-order scheme to provide oscillation free good resolution in steep gradient areas, while at the same time preserve at least second-order accuracy in smooth regions. The two-step FCT scheme was firstly developed by Boris and Book \cite{b1} for solving the transient continuity equation. The procedure of two-step flux correction consists of computing the time advanced low order solution in the first step and to correct solution by adding antidiffusive fluxes in the second step to produce accurate and monotone results. Antidiffusive fluxes, which are defined as the difference between the high and low-order fluxes, are limited in such a way that neither new extrema are created nor existing extrema are increased. Zalesak~\cite{b2,b3} extended the FCT algorithm for multidimensional explicit difference schemes. Several implicit FEM-FCT schemes for unstructured grids were proposed by Kuzmin and his coworkers~\cite{b4,b5,b6,b7}. Many variations and generalizations of FCT and their applications are given in~\cite{b36}. It should be noted that known FCT algorithms could also produce entropy violating solutions.

TVD schemes are also flux-limiter schemes that were originally introduced by Harten~\cite{b8,b9}. Using the TVD conditions Sweby~\cite{b10} derived conditions that flux limiters should satisfy. A number of TVD schemes have been described in the literature \cite{b11,b12,b13,b14}.
Unlike the TVD methods, it is not obvious what transformation in the well-known FCT formulas for computing flux limiters should be done when changing the original differential equations.
 
To construct flux correction formulas with maximal antidiffusive fluxes, following the FCT approach we consider a hybrid difference scheme consisting of a linear combination of low and high-order schemes. We calculate the low-order flux by using a weighted first-order upwind flux. The algorithm of flux correction design for the hybrid difference scheme can be formulated as follows:

1. We determine constraints that are valid for the low-order difference scheme.

2. We treat the finding flux limiters as an optimization problem with the above constraints that are applied to the hybrid scheme.

3. A nontrivial approximate solution of the corresponding linear programming problem is used as the flux limiter for the hybrid scheme.

It is easy to see that in such a way defined optimization problem is always solvable. Moreover, this approach reduces the traditional two-step FCT method for explicit schemes to a one-step method. The finding of flux limiters for an implicit difference scheme is reduced to a nonlinear optimization problem or an iterative sequence of linear programming problems. Also, this approach allows designing of flux correction with desired properties. For example, it is known that a weak solution of hyperbolic conservation laws may not be unique. However, the physical solution is unique and should satisfy the entropy condition \cite{b15,b16}. Using so-called proper numerical entropy fluxes introduced by Merriam~\cite{b17}, Sonar~\cite{b18}, and Zhao and Wu~\cite{b19}, we design the flux correction to find an entropy solution for scalar hyperbolic conservation laws. 

 The notion of entropy solution is the cornerstone in obtaining a physically relevant solution in the theory of hyperbolic systems of nonlinear conservation laws. We mention here the pioneering studies of the entropy solutions by Oleinik~\cite{b15}, Lax~\cite{b37} and Kruzhkov~\cite{b16}. A difference scheme is called entropy stable if computed solutions satisfy the discrete cell entropy inequality.
Entropy stable schemes have been developed by several authors \cite{b8,b20,b42,b22,b23,b24,b39}. Harten et al.~\cite{b42} showed that for scalar conservation laws all explicit monotone schemes are entropy stable. A class of E-schemes which includes monotone schemes and is entropy stable was introduced by Osher~\cite{b22,b23}. 
Tadmor~\cite{b26, b30} introduced general families of entropy-conservative schemes and investigated the entropy stability of difference approximations to nonlinear hyperbolic conservation laws by comparing the entropy production of a given scheme against properly chosen entropy-conservative schemes.
Chalons and LeFloch~\cite{b40} studied limiting solutions of fully discrete finite-difference schemes for a diffusive-dispersive conservation law. They investigated the dependence of these solutions on balance between dissipative and disperse forces and producing non-classical shock solutions violating the standard entropy criterion. In~\cite{b21} LeFloch et al. proposed a general approach to construct second- and third-order accurate, fully discrete implicit entropy-conservative schemes. 
  Zhao and Wu~\cite{b19} proved that a three-point monotone semi-discrete schemes in conservative form satisfy corresponding semi-discrete entropy inequality with the proper numerical entropy flux.
Construction of other high-order accurate, fully discrete entropy stable schemes can be found in \cite{b38,b41,b35}.

Note that the numerical entropy flux consistent with the entropy flux function is not unique.  Tadmor~\cite{b25,b26} proposed another form of the numerical entropy flux that differs from the proper numerical entropy flux.  
Numerical examples show that not all numerical solutions which satisfy the discrete cell entropy inequality with the Tadmor's numerical entropy flux are physically correct solutions.

In this paper, the flux correction formulas for scalar hyperbolic conservation laws and convective-diffusion equations are derived. Some of the obtained formulas are similar to the Zalesak formulas for FCT methods and the Kuzmin formulas of algebraic flux corrections. 

The paper is organized as follows. Section~\ref{Sec2} describes the design of flux correction for conservative difference schemes for linear scalar hyperbolic conservation laws. Section~\ref{Sec21} recalls the basic properties for weighted monotone schemes. In Section~\ref{Sec22}, we give solvability conditions for a hybrid weighted scheme. We describe the iterations for flux limiter calculation by using linear programming and also prove the convergence of these iterations in Section~\ref{Sec23}. The approximate solution of the corresponding linear programming problem is obtained in Section~\ref{Sec24}.

We approximate the entropy inequality by the discrete cell entropy inequality with the proper numerical entropy fluxes in Section~\ref{Sec3}. Section~\ref{Sec31} describes the iterative algorithm using linear programming to calculate the flux limiters to obtain numerical entropy solutions. Flux correction formula that is based on the approximate solution of the linear programming problem is derived in Section~\ref{Sec32}.
Section~\ref{Sec4} extends the results of Section~\ref{Sec2} to schemes for 2D convection-diffusion equations. The results of numerical experiments are given in Section~\ref{Sec5}. Concluding remarks are drawn in Section~\ref{Sec6}.

\section{Linear Scalar Conservation Laws}  \label{Sec2}
On an interval $\left[ {a,b} \right]$, we consider  the initial boundary value problem (IBVP) for the advection equation
\begin{equation}
\label{eq:21} 
\frac{\partial \rho }{\partial t} + \frac{\partial }{\partial x} \left( {u\rho} \right) = 0,	\qquad	t > 0								
\end{equation}
with initial condition
\begin{equation}
\label{eq:22} 
\rho \left( {x,0} \right) = {\rho ^0}\left( x \right)
\end{equation}	
						 				
For a unique solution of the problem \eqref{eq:21}-\eqref{eq:22}, the number of conditions specified on each boundary should be equal to the number of characteristics entering the domain. Let ${\rho _0}$ and ${\rho _{N + 1}}$ denote the values of $\rho (x,t)$ at the left and right ends of the segment $\left[ {a,b} \right]$.

On $\left[ {a,b} \right]$, we introduce a nonuniform grid ${\Omega _h}$
\begin{equation}
{\Omega _h} = \left\{ {{x_i}: \; {\rm{    }}{x_{i + 1}} = {x_i} + 
{\Delta _{i + 1 / 2}}x, \quad {\rm{    }}i = \overline {1,N - 1} ; \quad {\rm{   }}{x_1} > a, \; {\rm{   }}{x_{\rm{N}}} < b} \right\}				\label{eq:23} 
\end{equation}

We denote the inner product and norm of grid functions $y({x_i})$ defined on ${\Omega _h}$ as 
\begin{displaymath}
\left\langle {\boldsymbol{y},\boldsymbol{v}} \right\rangle  = \sum\limits_i {{y_i}{v_i}},	 \qquad
\left\| \boldsymbol{y} \right\| = {\left\langle {\boldsymbol{y},\boldsymbol{y}} \right\rangle ^{1 / 2}}
\end{displaymath}

We will omit the limits of summation where this will not cause confusion. The vector norm and the corresponding consistent matrix norm are denoted as
\begin{displaymath}
{\left\| \boldsymbol{y} \right\|_1} = \sum\limits_i {\left| {{y_i}} \right|} ,	\qquad		{\left\| A \right\|_1} = \mathop {\max }\limits_j \sum\limits_i {\left| {{a_{ij}}} \right|} 
\end{displaymath}

We approximate \eqref{eq:21} by the weighted conservative difference scheme
\begin{equation}
\label{eq:24} 
y_i^{n + 1} - y_i^n + \frac{{\Delta t}}{{\Delta {x_i}}}\left[ {h_{i + {1 / 2}}^{(\sigma )} - h_{i - {1 \mathord{\left/ {\vphantom {1 2}} \right.
 \kern-\nulldelimiterspace} 2}}^{(\sigma )}} \right] = 0 	
\end{equation} 
where $y_i^n = y({x_i},{t^n})$, $\Delta t$ and $\Delta {x_i}$ are respectively the temporal and spatial grid size, \\ $\Delta {x_i} = \frac{1}{2}\left( {{x_{i + 1}} - {x_{i - 1}}} \right)$, 
$h_{i + {1 / 2}}^{(\sigma )} = \sigma h_{i + {1 \mathord{\left/ {\vphantom {1 2}} \right.
 \kern-\nulldelimiterspace} 2}}^{n + 1} + (1 - \sigma )h_{i + {1 / 2}}^n$,  $\sigma\in[0,1]$. The numerical flux $h_{i + {1 / 2}}^n = h(y_{i - l + 1}^n,...,y_{i + r}^n)$ is the Lipschitz continuous and consistent with the differential flux $f(\rho )$, that is
	$h(\rho ,...,\rho ) = f(\rho )$. 
Moreover, we consider the numerical flux $h_{i + {1 / 2}}^n$ as a linear combination of lower-order and high-order fluxes
\begin{displaymath}
h_{i + {1/2}}^n = h_{i + {1 / 2}}^{L,n} + \alpha _{i + {1 / 2}}^n\left( {h_{i + {1 / 2}}^{H,n} - h_{i + {1 / 2}}^{L,n}} \right) = h_{i + {1 / 2}}^{L,n} + \alpha _{i + {1 / 2}}^nh_{i + {1 / 2}}^{d,n}		\qquad \text {for} \;\; {0 \le \alpha _{i + {1 / 2}}^n \le 1} 
\end{displaymath}
where $h_{i + {1 / 2}}^{L,n}$ and $h_{i + {1 / 2}}^{H,n}$ are the low-order and high-order fluxes, respectively; $\alpha _{i + {1 / 2}}^n$ is a flux limiter and $h_{i + {1 / 2}}^{d,n} = h_{i + {1 / 2}}^{H,n} - h_{i + {1 / 2}}^{L,n}$.
Then the numerical flux $h_{i + {1 / 2}}^\sigma $ can be represented as
\begin{equation}
\begin{split}
\label{eq:25} 
h_{i + {1 / 2}}^\sigma  = \sigma \left( {h_{i + {1 / 2}}^{L,n + 1} + \alpha _{i + {1 / 2}}^{n + 1}h_{i + {1 / 2}}^{d,n + 1}} \right) + (1 - \sigma )\left( {h_{i + {1 / 2}}^{L,n} + \alpha _{i + {1 / 2}}^nh_{i + {1 / 2}}^{d,n}} \right) = \\
 h_{i + {1 / 2}}^{L,(\sigma )} + \sigma \alpha _{i + {1 / 2}}^{n + 1}h_{i + {1 / 2}}^{d,n + 1} + (1 - \sigma )\alpha _{i + {1 / 2}}^nh_{i + {1 / 2}}^{d,n} = h_{i + {1 / 2}}^{L,(\sigma )} + (\alpha {h^d})_{i + {1 / 2}}^{(\sigma )}
\end{split} 						
\end{equation}

Substituting \eqref{eq:25} into \eqref{eq:24}, we rewrite the scheme in the form
\begin{equation}
\label{eq:26} 
\frac{{\Delta {x_i}}}{{\Delta t}}\left( {y_i^{n + 1} - y_i^n} \right) + \left[ {h_{i + {1 / 2}}^{L,(\sigma )} - h_{i - {1 / 2}}^{L,(\sigma )} + (\alpha {h^d})_{i + {1 / 2}}^{(\sigma )} - (\alpha {h^d})_{i - {1 / 2}}^{(\sigma )}} \right] = 0
\end{equation}

We can consider the third and fourth terms in square brackets as antidiffusive fluxes or flux correction at the cell boundaries. Further, we will consider the flux correction only at the grid cell interfaces.

For example, if we choose 
\begin{equation}
\label{eq:27} 
h_{i + {1 / 2}}^L = u_{i + {1 / 2}}^ + {y_i} + u_{i + {1 / 2}}^ - {y_{i + 1}}
\end{equation}
\begin{equation}
\label{eq:28} 
h_{i + {1 / 2}}^H = \frac{1}{2}{u_{i + {1 / 2}}} \; ({y_i} + {y_{i + 1}})
\end{equation}
then the numerical flux $h_{i + {1 / 2}}^{(\sigma )}$ can be presented as
\begin{equation}
\begin{split}
\label{eq:29} 
h_{i + {1 / 2}}^{(\sigma )} = \sigma \left[ {u_{i + {1 / 2}}^{ + ,n + 1} \; y_i^{n + 1} + u_{i + {1 / 2}}^{ - ,n + 1} \; y_{i + 1}^{n + 1} + \frac{1}{2} \alpha _{i + {1 / 2}}^{n + 1} \; \left| {u_{i + {1 / 2}}^{n + 1}} \right| \; (y_{i + 1}^{n + 1} - y_i^{n + 1})} \right] + \\				
 + (1 - \sigma )\left[ {u_{i + {1 / 2}}^{ + ,n} \; y_i^n + u_{i + {1 / 2}}^{ - ,n} \; y_{i + 1}^n + \frac{1}{2}\alpha _{i + {1 / 2}}^n \; \left| {u_{i + {1 / 2}}^n} \right| \; (y_{i + 1}^n - y_i^n)} \right] 
\end{split} 
\end{equation}  
where ${u^ \pm } = 0.5(u \pm \left| u \right|)$. 
 
We rewrite the difference scheme \eqref{eq:24} with the numerical flux \eqref{eq:29} in matrix form
\begin{equation}
\label{eq:210} 
\begin{split}
\left[ {E + \Delta t\;\sigma \left( {{A^{n + 1}} - {B^{n + 1}}(\boldsymbol{\alpha} )} \right)} \right]{\boldsymbol{y}^{n + 1}} & = \left[ {E - \Delta t\;(1 - \sigma )\left( {{A^n} - {B^n}(\boldsymbol{\alpha} )} \right)} \right]{\boldsymbol{y}^n} \\
& + \Delta t\;\left[ {\sigma {\boldsymbol{g}^{n + 1}} + (1 - \sigma ){\boldsymbol{g}^n}} \right] 
\end{split} 
\end{equation}
where $A = \left\{ {{a_{ij}}} \right\}_i^j$ and $B = \left\{ {{b_{ij}}} \right\}_i^j$ are tridiagonal square matrices of order {\it N}, $\boldsymbol{g}$  is the vector of boundary conditions and $\boldsymbol{\alpha}  = {\left( {\alpha _{3/2} ,...,{\alpha _{N - {1 / 2}}}} \right)^T} \in {R^{N - 1}}$ is the numerical vector of flux limiters determining the order of spatial approximation of the difference scheme \eqref{eq:210}. The values of the velocity $u\left( {x,t} \right)$ at the ends of the segment $\left[ {a,b} \right]$  at the points $x _0 = a $  and   $x _{N+1} = b $ are denoted, respectively, as $u _{1/2}$  and  $u _{N + 1/2}$. The boundary condition vector {\it g}  is specified by $\boldsymbol{g} = {\left( {u_{1 / 2}^ + \; {y_0},0,...,0, - u_{N + {1 / 2}}^ - \; {y_{N + 1}}} \right)^T}$. Elements of the matrices $A$ and $B$ are calculated as
\begin{equation}
\begin{split}
\label{eq:211} 
& a_{ii} = {{\left( {u_{i + {1 / 2}}^ +  - u_{i - {1 / 2}}^ - } \right)} \mathord{\left/ {\vphantom {{\left( {u_{i + {1 / 2}}^ +  - u_{i - {1 / 2}}^ - } \right)} {\Delta {x_i}}}} \right. \kern-\nulldelimiterspace} {\Delta {x_i}}}; \qquad	\qquad \qquad {a_{ii \mp 1}} =  \mp {{u_{i \mp {1 / 2}}^ \pm } \mathord{\left/ {\vphantom {{u_{i \mp {1 / 2}}^ \pm } {\Delta {x_i}}}} \right. \kern-\nulldelimiterspace} {\Delta {x_i}}}  \\
& b_{ii \mp 1} \left( \boldsymbol{\alpha}  \right) = \, - 0.5\,{\alpha _{i \mp {1 / 2}}}\;\,{{\left| {{u_{i \mp {1 / 2}}}} \right|} \mathord{\left/ {\vphantom {{\left| {{u_{i \mp {1 / 2}}}} \right|} {\Delta {x_i}}}} \right. \kern-\nulldelimiterspace} {\Delta {x_i}}};	\qquad	{b_{ii}}\left( \boldsymbol{\alpha}  \right) =  - {b_{ii - 1}}(\boldsymbol{\alpha} ) - {b_{ii + 1}}(\boldsymbol{\alpha} )			
\end{split}
\end{equation}

Further, we consider difference schemes that can be represented in the form \eqref{eq:210} for which the matrix $A = \left\{ {{a_{ij}}} \right\}_i^j$ satisfies the conditions
\begin{equation}
\label{eq:212} 
a_{ii} \ge 0; \qquad a_{ij} \le 0 \qquad {\text {for}} \quad \forall \;i,j = \overline {1,N}  
\end{equation}

In general case, elements of the matrix $B = \left\{ {{b_{ij}}} \right\}_i^j$ are linear functions of coordinates of the vector  $\boldsymbol{\alpha} $ and the following equalities are hold
\begin{equation}
\label{eq:213} 
\sum\limits_j {{b_{ij}}}  = 0											\end{equation}

The number of nonzero columns in an {\it i}-th row of the matrix $B$ depends on the scheme stencil ${S_i}$ at the {\it i}-th grid node. On the other hand, $B(\boldsymbol{\alpha} )y$ can be represented in the form
\begin{equation}
\label{eq:214} 
\sum\limits_j {{b_{ij}}(\boldsymbol{\alpha} ) \; {y_j}}  = \sum\limits_k {{c_{ik}} \;  (\boldsymbol{{\Delta _i}y}){\alpha _{k + {1 / 2}}}} 								
\end{equation}
where ${c_{ik}}(\boldsymbol{{\Delta _i}y})$ is a linear combination of coordinates of the vector $\boldsymbol{{\Delta _i}y} = {({y_1} - {y_i}, \ldots ,{y_N} - {y_i})^T}$. The 
$c _{ik}(\boldsymbol{{\Delta _i}y})$ represents the antidiffusive term that is equal to the difference between high and low order fluxes at the interface of {\it i}-th grid node. By conservative discretization \eqref{eq:24}, each k-th column of the $N \times(N-1)$ matrix $C = \left\{ {{c_{ik}}} \right\}_i^k$   contains exactly two elements corresponding to ${\alpha _{k + 1 / 2}}$, which differ only in sign.

Further, we suppose that there exists a diagonal matrix $D = {\left\{ {{d_{ii}}} \right\}_i}$ such that
\begin{equation}
\label{eq:215} 
{d_{ii}} > 0, \qquad  \sum\limits_i {d_{ii} \; a_{ij}}  \ge 0 \qquad {\text {and} } \qquad \sum\limits_i {d_{ii} \; b_{ij} = 0} \qquad  {\text  {for all}} \quad i,j = \overline {1,N} 
\end{equation}

For example, for the difference scheme \eqref{eq:210}-\eqref{eq:211}, we can choose the matrix with diagonal elements ${d_{ii}} = \Delta {x_i}$ as the matrix {\it D}.

\subsection{Monotone Difference Scheme} \label{Sec21}
We consider the system of equations \eqref{eq:210} for ${\boldsymbol{\alpha} ^n},{\boldsymbol{\alpha} ^{n + 1}} = 0$

\begin{equation}
\label{eq:216} 
\left[ {E + \Delta t\;\sigma {A^{n + 1}}} \right]{\boldsymbol{y}^{n + 1}} = \left[ {E - \Delta t\;(1 - \sigma ){A^n}} \right]{\boldsymbol{y}^n} + \Delta t\;{\boldsymbol{g}^{(\sigma ),n}}	\end{equation}
where ${\boldsymbol{g}^{(\sigma ),n}} = \sigma {\boldsymbol{g}^{n + 1}} + (1 - \sigma ){\boldsymbol{g}^n}$. 

\begin{theorem}
\label{th:1}
Let conditions \eqref{eq:212} and \eqref{eq:215} be satisfied and $\sigma  \in [0,1]$. 
 If $\Delta t$ satisfies
\begin{equation}
\label{eq:217} 
\Delta t \; (1 - \sigma ) \; \mathop {\max }\limits_i a_{ii}^n \le 1
\end{equation}
then the difference scheme \eqref{eq:216} is monotone.
\end{theorem}
\begin{proof}
From conditions \eqref{eq:212} and \eqref{eq:215} it follows that the matrix $G = D\left[ {E + \Delta t\,\sigma {A^{n + 1}}} \right]$ is a M-matrix. Then the inverse matrix ${G^{ - 1}}$ is a matrix with nonnegative elements. Therefore, the matrix ${\left[ {E + \Delta t\;\sigma {A^{n + 1}}} \right]^{ - 1}} = {G^{ - 1}}D$ is also a matrix with nonnegative elements. 

Nonnegativity of the elements of ${\left[ {E + \Delta t\;\sigma {A^{n + 1}}} \right]^{ - 1}}\left[ {E - \Delta t(1 - \sigma ){A^n}} \right]$ and hence monotonicity of the scheme \eqref{eq:216} follows from the nonnegativity of the elements $\left[ {E - \Delta t(1 - \sigma ){A^n}} \right]$ for $\Delta t$ that satisfies \eqref{eq:217}.
\end{proof}

\begin{theorem}
\label{th:2}
 Let the matrices A and D satisfy the conditions \eqref{eq:212} and \eqref{eq:215}. Then the matrix $\left[ {E + \Delta t \;\sigma A} \right]$ is invertible for any $0 \le \sigma  \le 1$, $\Delta t\; > 0$ and
\begin{equation}
\label{eq:218}
\frac{{{{\left\| D \right\|}_1}}}{{\mathop {\max }\limits_j \left( {{d_{jj}} + \Delta t\;\sigma \sum\limits_i {{d_{ii}}{a_{ij}}} } \right)}} \le {\left\| {{{\left[ {E + \Delta t\;\sigma A} \right]}^{ - 1}}} \right\|_1} \le \frac{{{{\left\| D \right\|}_1}}}{{\mathop {\min }\limits_j \left( {{d_{jj}} + \Delta t\;\sigma \sum\limits_i {{d_{ii}}{a_{ij}}} } \right)}} 	\end{equation}  
\end{theorem}
 
\begin{proof}
 Under conditions \eqref{eq:212} and \eqref{eq:215}, the matrix $D\left[ {E +\Delta t\; \sigma A} \right]$ is a strongly column diagonally dominant matrix with positive diagonal and nonpositive off-diagonal elements. Consequently, the matrix $\left[ {E + \Delta t\;\sigma A} \right]$ is a nonsingular M-matrix and the elements of its inverse matrix are nonnegative.
  
To obtain the inequality \eqref{eq:218}, we use the approach given in \cite{b27}. For the inverse matrix 
 ${\left[ {E + \Delta t\; \sigma A} \right]^{ - 1}} = \left\{ {{{\bar a}_{ij}}} \right\}_i^j$ and matrix $\left[ {E +  \Delta t\;\sigma  A} \right] = \left\{ {{{\tilde a}_{ij}}} \right\}_i^j$, it is valid equality \\ $\left[ {E + \Delta t\; \sigma A} \right]{\left[ {E + \Delta t\;\sigma A} \right]^{ - 1}} = E$ or
\begin{equation}
\label{eq:219}
\sum\limits_k {{{\tilde a}_{ik}}{{\bar a}_{kj}} = {\delta _{ij}}}  			\end{equation}
where ${\delta _{ik}}$ is the Kronecker symbol. Multiplying \eqref{eq:218} by ${d_{ii}}$ and summing over {\it i}, we obtain
\begin{equation}
\label{eq:220}
\sum\limits_k {{{\bar a}_{kj}}} \sum\limits_i {{d_{ii}}} {\tilde a_{ik}} = {d_{jj}}
\end{equation}
										
Let ${\left\| {{{\left[ {E + \Delta t\;\sigma A} \right]}^{ - 1}}} \right\|_1} = \sum\limits_k {{{\bar a}_{kr}}} $ and ${\left\| D \right\|_1} = {d_{ss}}$. Since ${\bar a_{kj}} \ge 0$ and $\sum\limits_i {{d_{ii}}{a_{ik}}}  \ge 0$, it follows from \eqref{eq:220}
\[\mathop {\min }\limits_j \sum\limits_i {{d_{ii}}{{\tilde a}_{ij}}} \,{\left\| {{{\left[ {E + \Delta t\;\sigma A} \right]}^{ - 1}}} \right\|_1} \le \sum\limits_k {{{\bar a}_{kr}}} \sum\limits_i {{d_{ii}}} {\tilde a_{ik}} = {d_{rr}} \le {\left\| D \right\|_1}\] 
\[\mathop {\max }\limits_j \sum\limits_i {{d_{ii}}{{\tilde a}_{ij}}} \,{\left\| {{{\left[ {E + \Delta t\;\sigma A} \right]}^{ - 1}}} \right\|_1} \ge \sum\limits_k {{{\bar a}_{ks}}} \sum\limits_i {{d_{ii}}} {\tilde a_{ik}} = {d_{ss}} = {\left\| D \right\|_1}\]
which was to be proved.
\end{proof}

\begin{theorem}
\label{th:3}
 Let the matrices A and D satisfy the conditions \eqref{eq:212} and \eqref{eq:215}. 
 Then the difference scheme \eqref{eq:216} is unconditionally stable for $\sigma  \ge 0.5$ and the following inequality holds
\begin{equation}
\label{eq:221}
\left\| {\left( {{D^{{1 / 2}}} + \Delta t\;\sigma {D^{{1 / 2}}}{A^{n + 1}}} \right){\boldsymbol{y}^{n + 1}}} \right\| \le \left\| {\left( {{D^{{1 / 2}}} + \Delta t\;\sigma {D^{{1 / 2}}}{A^0}} \right){\boldsymbol{y}^0}} \right\| + \Delta t\;\sum\limits_{k = 0}^n {\left\| {{D^{{1 / 2}}}{\boldsymbol{g}^{(\sigma ),k}}} \right\|} 
\end{equation} 		
\end{theorem}
\begin{proof}
 Under the conditions \eqref{eq:212} and \eqref{eq:215}, the matrix $\it DA$ is a weakly column diagonally dominant matrix. Moreover, as follows from the Gershgorin disks, the eigenvalues of the matrix $\it DA$ are nonnegative.
 
We multiply both sides of the equation \eqref{eq:216} by the matrix $D^{1/2}$
\begin{equation}
\label{eq:222}
\left( {D^{1/2} + \Delta t \; \sigma D^{1/2} A^{n+1}} \right)\boldsymbol{y}^{n + 1} = \left[ {D^{1/2} - \Delta t \; (1 - \sigma ){D^{1/2}}{A^n}} \right]{\boldsymbol{y}^n} + \Delta t \, {D^{1/2}}{\boldsymbol{g}^{(\sigma ),n}}			
\end{equation} 		

From the triangle inequality we have
\begin{equation}
\label{eq:223}
\left\| {\left( {D^{1/2} + \Delta t\;\sigma {D^{1/2}}A^{n+1}} \right){\boldsymbol{y}^{n+1}}} \right\| \le \left\| {\left[ {D^{1/2} - \Delta t\;(1 - \sigma ){D^{1/2}}{A^n}} \right]{\boldsymbol{y}^n}} \right\| + \Delta t\;\left\| {D^{1/2}{\boldsymbol{g}^{(\sigma ),n}}} \right\|		
\end{equation} 		

We show that for any $\boldsymbol{y}$ defined on ${\Omega _h}$ and $\sigma \ge 0.5$ the following inequality holds
\begin{equation}
\label{eq:224}
\left\| {\left[ {D^{1/2} - \Delta t \;(1 - \sigma ){D^{1/2}}A} \right]\boldsymbol{y} }  \right\| \le \left\| {\left( D^{1/2} + \Delta t \; \sigma D^{1/2} A \right) \boldsymbol{y}}  \right\|					
\end{equation} 		

In fact, for $\sigma \ge 0.5$ we have
\begin{eqnarray*}
{\left\| {\left( {D^{1/2} + \Delta t \; \sigma {D^{1/2}}A} \right)\boldsymbol{y}} \right\|^2} - {\left\| {\left[ {D^{1/2} - \Delta t \; (1 - \sigma ){D^{1/2}}A} \right] \boldsymbol{y}} \right\|^2} = \\
 = 2\Delta t\;\left\langle {\boldsymbol{y},DA\boldsymbol{y}} \right\rangle  + \Delta {t^2}(2\sigma  - 1){\left\| {{D^{{1 / 2}}}A\boldsymbol{y}} \right\|^2} \ge 0,							
\end{eqnarray*} 		
so far as 
\begin{equation*}
\left\langle {\boldsymbol{y},DA\boldsymbol{y}} \right\rangle  \ge {\lambda _{\min }}\left\langle {\boldsymbol{y},\boldsymbol{y}} \right\rangle  \ge 0 									
\end{equation*} 		
where ${\lambda _{\min }}$ is the minimal eigenvalue of the matrix {\it DA}. 

Substituting \eqref{eq:224} into \eqref{eq:223} yields
\begin{equation}
\label{eq:225}
\left\| {\left( {{D^{{1 / 2}}} + \Delta t\;\sigma {D^{{1 / 2}}}{A^{n + 1}}} \right){\boldsymbol{y}^{n + 1}}} \right\| \le \left\| {\left( {{D^{{1 / 2}}} + \Delta t\;\sigma {D^{{1 / 2}}}{A^n}} \right){\boldsymbol{y}^n}} \right\| + \Delta t\;\left\| {{D^{{1 / 2}}}{\boldsymbol{g}^{(\sigma ),n}}} \right\|			
\end{equation}
 		
Summing \eqref{eq:225} over time {\it n}, we obtain the inequality \eqref{eq:221}.
\end{proof}

\begin{theorem}
\label{th:4}
 Let the matrix A satisfies the condition \eqref{eq:212}. Then for
\begin{equation}
\label{eq:226}
\Delta t\; \le \frac{-1}{{(1 - \sigma )\sum\limits_{j \ne i} a_{ij}^n }},	\end{equation}
the solution of the system of equations \eqref{eq:216} satisfies the following inequalities   
\begin{equation}
\label{eq:227}
\begin{split}
\mathop {\min }\limits_{j \in {S_i}} y_j^n & - \Delta t\,(1 - \sigma )\,y_i^n\sum\limits_j {a_{ij}^n}  + \Delta t\,{g_i^{(\sigma ),n}} \le y_i^{n + 1} + \Delta t\,\sigma \sum\limits_j {a_{ij}^{n + 1}} y_j^{n + 1} \le  \\ 
& \le \mathop {\max }\limits_{j \in {S_i}} y_j^n - \Delta t\,(1 - \sigma )\,y_i^n\sum\limits_j {a_{ij}^n}  + \Delta t\,{g_i^{(\sigma ),n}} \qquad	\forall \quad i = \overline {1,N} 					
\end{split} 
\end{equation} 		
where ${S_i}$ is the stencil of the difference scheme \eqref{eq:216} for i-th grid node. 
\end{theorem}
\begin{proof}
 We rewrite the {\it i}-th row of the system of equations \eqref{eq:216} in the form
\begin{equation}
\label{eq:228}
\begin{split}
y_i^{n + 1} + \Delta t\,\sigma \sum\limits_j {a_{ij}^{n + 1}y_j^{n + 1}} &  = \left[ {1 + \Delta t\,(1 - \sigma )\sum\limits_{j \ne i} {a_{ij}^n} } \right]y_i^n - \Delta t\,(1 - \sigma )\sum\limits_{j \ne i} {a_{ij}^ny_j^n} \\
& - \Delta t\,(1 - \sigma )\,y_i^n\sum\limits_j {a_{ij}^n}  + \Delta t\,{g_i^{(\sigma ),n}} 
\end{split} 
\end{equation} 		

Under the condition \eqref{eq:226}, the first two terms on the right-hand side of \eqref{eq:228} are a convex combination of the coordinates $y_j^n$, therefore

\begin{equation*}
y_i^{n + 1} + \Delta t\,\sigma \sum\limits_j {a_{ij}^{n + 1}y_j^{n + 1}}  \le \mathop {\max }\limits_{j \in {S_i}} y_j^n - \Delta t\,(1 - \sigma )\,y_i^n\sum\limits_j {a_{ij}^n}  + \Delta t\,{g_i^{(\sigma ),n}}
\end{equation*} 
		
Similarly, the lower bound \eqref{eq:227} is obtained, which proves the theorem.
\end{proof}

\subsection{Hybrid Scheme} \label{Sec22}
Let ${U_{ad}}$ denote the set of vectors ${\boldsymbol{\alpha} ^n},{\boldsymbol{\alpha} ^{n + 1}}$, which is defined as the Cartesian product of (N-1)-vectors
\begin{equation}
\label{eq:229}
  U^{ad}= \left\lbrace {\left( {\boldsymbol{\alpha}^n,\boldsymbol{\alpha}^{n+1} } \right) \in R^{N-1} \times R^{N-1}: \quad 0\leq \alpha_{i+1/2}^k\leq 1, \;\; i=\overline{1,N-1}; \;\; k=n,n+1  }\right\rbrace 
\end{equation}

Consider the weighted difference scheme \eqref{eq:210}.
\begin{theorem}
\label{th:5}
 Let $\boldsymbol{\alpha} \in R^{N-1}$  and $\boldsymbol{\alpha}, \sigma \geq 0$. Let the matrices A, B and D satisfy the conditions \eqref{eq:212}, \eqref{eq:213} and \eqref{eq:215}. Then, in order that the matrix  $\left[ {E + \Delta t\,\sigma \left( {A - B(\boldsymbol{\alpha} )} \right)} \right]$ to be invertible, it is sufficient to choose $\Delta t\,$ satisfying the inequality
\begin{equation}
\label{eq:230}
 \Delta t \, {\kern 1pt} \sigma \,{\left\| {DB(\boldsymbol{\alpha} )} \right\|_1} \,\, < \mathop {\min }\limits_j \left( {d_{jj} + \Delta t \,\sigma \sum\limits_i {d_{ii}{a_{ij}}} } \right)							
\end{equation}
\end{theorem} 
\begin{proof}
 Under conditions \eqref{eq:212} and \eqref{eq:215} the matrix $\left[ {D + \Delta t\,\sigma DA} \right]$ is a strongly column diagonally dominant matrix, therefore, we rewrite \eqref{eq:220} in the form 
\begin{equation}
\label{eq:231a}
\sum\limits_k {{{\bar a}_{kj}}} d_{jj}^{ - 1}\sum\limits_i {{d_{ii}}} {\tilde a_{ik}} = 1
\end{equation}
 
It follows from \eqref{eq:231a} that 
\begin{equation}
\label{eq:231}
{\left\| {{{\left[ {D + \Delta t\,\sigma DA} \right]}^{ - 1}}} \right\|_1} \le {\left( {\mathop {\min }\limits_j \left[ {{d_{jj}} + \Delta t\,\sigma \sum\limits_i {{d_{ii}}{a_{ij}}} } \right]} \right)^{ - 1}} 				\end{equation}

Then from \eqref{eq:230} and \eqref{eq:231} we have that
\[\Delta t\,\sigma {\left\| {{{\left[ {E + \Delta t\,\sigma A} \right]}^{ - 1}}B\left( \boldsymbol{\alpha}  \right)} \right\|_1} = \Delta t\,\sigma {\left\| {{{\left[ {E + \Delta t\,\sigma A} \right]}^{ - 1}}{D^{ - 1}}DB\left( \boldsymbol{\alpha}  \right)} \right\|_1} < 1\]	
				 
Therefore, $\left[ {E - \Delta t\,\sigma {{\left[ {E + \Delta t\,\sigma A} \right]}^{ - 1}}B} \right]$ is a nonsingular matrix \cite{b28}.

Thus, $\left[ {E + \Delta t\,\sigma \left( {A - B} \right)} \right]$ can be represented as the product of two nonsingular matrices
\[ \left[ {E + \Delta t\,\sigma \left( {A - B)} \right)} \right] = \left[ {E + \Delta t\,\sigma A} \right]\left[ {E - \Delta t\,\sigma {{\left[ {E + \Delta t\,\sigma A} \right]}^{ - 1}}B} \right] \]
which was to be proved.
\end{proof}

\begin{corl}
 Let the matrices A, B and D satisfy conditions \eqref{eq:212}-\eqref{eq:213} and \eqref{eq:215}. Let $\boldsymbol{\alpha}  \in {U_{ad}}$, $0 \le \sigma  \le 1$ and $\Delta t \,$ satisfies the inequality 
\begin{equation}
\label{eq:232}
\Delta t\,\sigma \,\,\,{\left\| {D{B^{n + 1}}(\boldsymbol{1})} \right\|_1}\, \le \gamma \mathop {\min }\limits_j \left( {{d_{jj}} + \Delta t\,\sigma \sum\limits_i {{d_{ii}}a_{ij}^{n + 1}} } \right) \qquad		0 < \gamma  < 1			\end{equation}

Then the set of solutions of the system of linear equations \eqref{eq:210} is bounded on ${U_{ad}}$ and
\begin{equation}
\label{eq:233}
{\left\| {{\boldsymbol{y}^{n + 1}}} \right\|_1} \le M\left( {{M_1}{{\left\| {{\boldsymbol{y}^n}} \right\|}_1} + \Delta t\,{{\left\| {{\boldsymbol{g}^{(\sigma ),n}}} \right\|}_1}} \right)							
\end{equation}
\[M = \frac{{{{\left\| D \right\|}_1}}}{{1 - \gamma }}\;\mathop {\max }\limits_j {\left( {{d_{jj}} + \Delta t\,\sigma \sum\limits_i {{d_{ii}}a_{ij}^{n + 1}} } \right)^{ - 1}}\] 
\[{M_1} = {\left\| {E - \Delta t\,(1 - \sigma ){A^n}} \right\|_1} + \Delta t\,(1 - \sigma ){\left\| {{B^n}(\boldsymbol{1})} \right\|_1}\] 
where $B(\boldsymbol{1})$ is the matrix $B(\boldsymbol{\alpha} )$ for $\alpha_i  = 1$.  
\end{corl} 
\begin{proof}
 The proof of the Theorem~\ref{th:5} shows that, under condition \eqref{eq:232}, we have
\[\sigma \Delta t\,{\left\| {{{\left[ {E + \sigma \Delta t\,A} \right]}^{ - 1}}B\left( \boldsymbol{\alpha}  \right)} \right\|_1} \le \gamma \qquad \text {for any} \;\; \boldsymbol{\alpha}  \in {U_{ad}}\].
				
Then
\[{\left\| {{{\left[ {E - \Delta t\,\sigma {{\left[ {E + \Delta t\,\sigma A} \right]}^{ - 1}}B} \right]}^{ - 1}}} \right\|_1} \le \frac{1}{{1 - \gamma }}\]
and taking into account \eqref{eq:218} 
\begin{equation*} 
\begin{split}  
{\left\| {{{\left[ {E + \sigma \,\Delta t\left( {A - B)} \right)} \right]}^{ - 1}}} \right\|_1} & \le {\left\| {{{\left[ {E + \sigma \,\Delta t\,A} \right]}^{ - 1}}} \right\|_1}{\left\| {{{\left[ {E - \Delta t\,\sigma {{\left[ {E + \Delta t\,\sigma A} \right]}^{ - 1}}B} \right]}^{ - 1}}} \right\|_1} \\
& \le \frac{1}{{1 - \gamma }}\frac{{{{\left\| D \right\|}_1}}}{{\mathop {\min }\limits_j \left( {{d_{jj}} + \Delta t\,\sigma \sum\limits_i {{d_{ii}}{a_{ij}}} } \right)}}
\end{split}
\end{equation*}

Therefore
\[{\left\| {{\boldsymbol{y}^{n + 1}}} \right\|_1} \le {\left\| {{{\left[ {E + \Delta t\,\sigma \left( {{A^{n + 1}} - {B^{n + 1}})} \right)} \right]}^{ - 1}}} \right\|_1}{\left\| \boldsymbol{\varphi}  \right\|_1} \le \frac{1}{{1 - \gamma }}\frac{{{{\left\| D \right\|}_1}}}{{\mathop {\min }\limits_j \left( {{d_{jj}} + \Delta t\,\sigma \sum\limits_i {{d_{ii}}a_{ij}^{n + 1}} } \right)}}{\left\| \boldsymbol{\varphi}  \right\|_1}\]		
where
$\boldsymbol{\varphi}  = \left[ {E - \Delta t\,(1 - \sigma )\left( {{A^n} - {B^n}(\boldsymbol{\alpha} )} \right)} \right]{\boldsymbol{y}^n} + \Delta t\,\left[ {\sigma {\boldsymbol{g}^{n + 1}} + (1 - \sigma ){\boldsymbol{g}^n}} \right]$

For ${\left\| \boldsymbol{\varphi}  \right\|_1}$ we have
\[{\left\| \boldsymbol{\varphi}  \right\|_1} \le \left\{ {{{\left\| {E - \Delta t\,(1 - \sigma ){A^n}} \right\|}_1} + \Delta t\,(1 - \sigma ){{\left\| {{B^n}(\boldsymbol{1})} \right\|}_1}} \right\}{\left\| {{\boldsymbol{y}^n}} \right\|_1} + \Delta t\,{\left\| {{\boldsymbol{g}^{(\sigma ),n}}} \right\|_1}\]
which proves the theorem.
\end{proof}

\subsection{Finding Flux Limiters} \label{Sec23}
Our goal is to find maximal values of the flux limiters for which solution of the difference scheme \eqref{eq:210} is similar to the solution of the monotone difference scheme \eqref{eq:216}. For this, we will require that the difference scheme \eqref{eq:210} satisfies inequalities \eqref{eq:227}. Then the problem of finding the flux limiters can be considered as the following optimization problem
\begin{equation}
\label{eq:234}
 \Im ({\boldsymbol{\alpha} ^n},{\boldsymbol{\alpha} ^{n + 1}}) = \left\langle {\boldsymbol{c_1},{\boldsymbol{\alpha} ^n}} \right\rangle  + \left\langle {\boldsymbol{c_2},{\boldsymbol{\alpha} ^{n + 1}}} \right\rangle  \to \mathop {\max }\limits_{{\boldsymbol{\alpha} ^n},{\boldsymbol{\alpha} ^{n + 1}} \in {U_{ad}}}  \qquad	\boldsymbol{c_1},\boldsymbol{c_2} > 0 				
\end{equation}
subject to
%\begin{equation}
\begin{eqnarray}
\label{eq:235} 
%\begin{split} 
 E{\underline {\boldsymbol{y}}^n} - \Delta t \,(1 - \sigma ) \, diag\left[ {\sum\limits_j {a_{ij}^n} }  \right]{\boldsymbol{y}^n} & \le &
\left[ {E - \Delta t \, (1 - \sigma )\left( {{A^n} - {B^n}({\boldsymbol{\alpha} ^n})} \right)} \right]{\boldsymbol{y}^n} + {B^{n + 1}}({\boldsymbol{\alpha} ^{n + 1}}){\boldsymbol{y}^{n + 1}} \nonumber \\
& \le & E{\bar {\boldsymbol{y}}^n} - \Delta t \,(1 - \sigma ) \, diag\left[ {\sum\limits_j {a_{ij}^n} }  \right]{\boldsymbol{y}^n}  
%\end{split}  								
\end{eqnarray}
%\end{equation}
\begin{equation}
%\begin{eqnarray}
\label{eq:236} 
\begin{split} 
 \left[ {E + \Delta t\, \sigma \,\left( {{A^{n + 1}} - {B^{n + 1}}({\boldsymbol{\alpha} ^{n + 1}})} \right)} \right] & {\boldsymbol{y}^{n + 1}}  \\
& = \left[ {E - \Delta t\, (1 - \sigma )\,\left( {{A^n} - {B^n}({\boldsymbol{\alpha} ^n})} \right)} \right]{\boldsymbol{y}^n} + \Delta t \,{\boldsymbol{g}^{(\sigma ),n}}	
%\end{eqnarray} 
\end{split} 
\end{equation}
where $diag\left[ {\sum\limits_j {a_{ij}^n} }  \right]$ is the diagonal matrix with elements $\sum\limits_j {a_{ij}^n}$, ${\underline y} {_i} = \mathop {\min }\limits_{j \in {S_i}} {y_j}$ , ${\bar y_i} = \mathop {\max }\limits_{j \in {S_i}} {y_j}$.
 
Note that for $\sigma  = 0$, this is a linear programming problem, and for $\sigma  > 0$ this is a non-linear programming problem.

To solve the non-linear optimization problem \eqref{eq:234}-\eqref{eq:236} at one time step, we will use the following iterative process: 
\begin{enumerate}[\bfseries   Step 1.]
\item \label {it:1}
 Initialize positive numbers $\delta ,{\varepsilon _1},{\varepsilon _2} > 0$ and vectors $\boldsymbol{c_1},\boldsymbol{c_2} > 0$. Set $p = 0$, ${\boldsymbol{y}^{n + 1,0}} = {\boldsymbol{y}^n}$, ${\boldsymbol{\alpha} ^{n,0}},{\boldsymbol{\alpha} ^{n + 1,0}} = 0$.
\item \label {it:2}
 Find the solution ${\boldsymbol{\alpha} ^{n,p + 1}},{\boldsymbol{\alpha} ^{n + 1,p + 1}}$ of the following linear programming problem
\begin{equation}
\label{eq:237}
\Im  = \left\langle {\boldsymbol{c_1},{\boldsymbol{\alpha} ^{n,p + 1}}} \right\rangle  + \left\langle {\boldsymbol{c_2},{\boldsymbol{\alpha} ^{n + 1,p + 1}}} \right\rangle  \to \mathop {\max }\limits_{{\boldsymbol{\alpha} ^{n,p + 1}},{\boldsymbol{\alpha} ^{n + 1,p + 1}} \in {U_{ad}}} 			\end{equation}
%\begin{equation}
\begin{eqnarray}
\label{eq:238}
%\begin{split} 
&  \mathop {\min }\limits_{j \in {S_i}} y_j^n - y_i^n + \Delta t\,(1 - \sigma )\sum\limits_{j \ne i} {a_{ij}^n} \left( {y_j^n - y_i^n} \right) \nonumber \\
&  \le \Delta t \, (1 - \sigma )\sum\limits_{j \ne i} {b_{ij}^n} ({\boldsymbol{\alpha} ^{n,p + 1}})\left( {y_j^n - y_i^n} \right) + \Delta t\,\sigma \sum\limits_{j \ne i} {b_{ij}^{n + 1}({\boldsymbol{\alpha} ^{n + 1,p + 1}})\left( {y_j^{n + 1,p} - y_i^{n + 1,p}} \right)}  	\nonumber	\\
&  \le \mathop {\max }\limits_{j \in {S_i}} y_j^n - y_i^n + \Delta t\,(1 - \sigma )\sum\limits_{j \ne i} {a_{ij}^n} \left( {y_j^n - y_i^n} \right) 
%\end{split}  
\end{eqnarray} 	
%\end{equation}
\item \label {it:3}
 For the ${\boldsymbol{\alpha} ^{n,p + 1}},{\boldsymbol{\alpha} ^{n + 1,p + 1}}$, find $y_i^{n + 1,p + 1}$ from the system of linear equations
\begin{equation}
\label{eq:239}
\begin{split} 
 \left[ {E + \Delta t\,\sigma {A^{n + 1}}} \right]{\boldsymbol{y}^{n + 1,p + 1}} & = \left[ {E - \Delta t\,(1 - \sigma )\left( {{A^n} - {B^n}({\boldsymbol{\alpha} ^{n,p + 1}})} \right)} \right]{\boldsymbol{y}^n} \\ 
& + \Delta t\,\sigma {B^{n + 1}}({\boldsymbol{\alpha} ^{n + 1,p + 1}}){\boldsymbol{y}^{n + 1,p}} + \Delta t\,{\boldsymbol{g}^{(\sigma ),n}}  
\end{split} 	
\end{equation}
\item \label {it:4}
 Algorithm stop criterion
\begin{equation}
\label{eq:240}
\frac{{\left| {y_i^{n + 1,p + 1} - y_i^{n + 1,p}} \right|}}{{\max \left( {\delta ,\left| {y_i^{n + 1,p + 1}} \right|} \right)}} < {\varepsilon _1},   \; \left| {\alpha _{i + {1 / 2}}^{n + 1,p + 1} - \alpha _{i + {1 / 2}}^{n + 1,p}} \right| < {\varepsilon _2} \;\; \text {and} \; \left| {\alpha _{i + {1 / 2}}^{n,p + 1} - \alpha _{i + {1 / 2}}^{n,p}} \right| < {\varepsilon _2} 		
\end{equation}
If conditions \eqref{eq:240} hold, then ${\boldsymbol{y}^{n + 1}} = {\boldsymbol{y}^{n + 1,p + 1}}$, otherwise $p = p + 1$ and go to {\bf Step \ref{it:2}}. 
\end{enumerate} 

\begin{theorem}
\label{th:6}
Let the square matrices A, B and D satisfy conditions \eqref{eq:212}-\eqref{eq:213} and \eqref{eq:215}. If $\Delta t \,$ satisfies \eqref{eq:226}, then the linear programming problem \eqref{eq:237}-\eqref{eq:238} is solvable.
\end{theorem} 
\begin{proof}
 To prove that the problem \eqref{eq:237}-\eqref{eq:238} is solvable, it is sufficient to show that the objective function $\Im ({\boldsymbol{\alpha} ^n},{\boldsymbol{\alpha} ^{n + 1}})$ is bounded and the feasible set is non-empty. The boundedness of the function \eqref{eq:237} follows from the boundedness of the vectors ${\boldsymbol{\alpha} ^n}$ and ${\boldsymbol{\alpha} ^{n + 1}}$ whose coordinates vary from zero to one. On another hand, if \eqref{eq:226} is valid, the zero vectors ${\boldsymbol{\alpha} ^n}$ and ${\boldsymbol{\alpha} ^{n + 1}}$ satisfy the system of inequalities \eqref{eq:238}.
 
 This completes the proof of the Theorem. 
\end{proof}
\begin{theorem}
\label{th:7}
 Let the matrices A, B and D satisfy conditions \eqref{eq:212}-\eqref{eq:213} and \eqref{eq:215}. Let $\left( {{\boldsymbol{\alpha} ^{n,p}},{\boldsymbol{\alpha} ^{n + 1,p}}} \right) \in {U_{ad}}$, $\sigma  \ge 0$ and $\Delta t\,$ satisfies the inequality 
\begin{equation}
\label{eq:241} 
\begin{split} 
 \Delta t\,\max \left\{ {(1 - \sigma )\,{{\left\| {D{B^n}(\boldsymbol{1})} \right\|}_1},\sigma {{\left\| {D{B^{n + 1}}(\boldsymbol{1})} \right\|}_1}\,\,} \right\} \\
 \le \gamma \mathop {\min }\limits_j \left( {{d_{jj}} + \Delta t\,\sigma \sum\limits_i {{d_{ii}}a_{ij}^{n + 1}} } \right),	 \qquad 0 < \gamma  < 1	  
\end{split}  
\end{equation}

Then the solution of the system of equations \eqref{eq:239} can be represented as the convergent series
\begin{equation}
\label{eq:242}
 {\boldsymbol{y}^{n + 1,p}} = \left[ {\sum\limits_{i = 1}^p {{T_i}} \prod\limits_{j = i + 1}^p {{Q_j}}  + \prod\limits_{i = 1}^p {{Q_i}} } \right]{\boldsymbol{y}^n} + \left[ {E + \sum\limits_{i = 2}^p {\prod\limits_{j = i}^p {{Q_i}} } } \right]\boldsymbol{r} 		\end{equation}
where $ {T_p} = \Delta t\,(1 - \sigma ){\left[ {E + \Delta t\,\sigma {A^{n + 1}}} \right]^{ - 1}}{B^n}({\boldsymbol{\alpha} ^{n,p}})$,  ${Q_p} = \Delta t\,\sigma {\left[ {E + \Delta t\,\sigma {A^{n + 1}}} \right]^{ - 1}}{B^{n + 1}}({\boldsymbol{\alpha} ^{n + 1,p}})$,
\begin{flushleft} 
$ \boldsymbol{r} = {\left[ {E + \Delta t\,\sigma {A^{n + 1}}} \right]^{ - 1}}\left\{ {\left[ {E - \Delta t\,(1 - \sigma ){A^n}} \right]{\boldsymbol{y}^n} + \Delta t\,{\boldsymbol{g}^{(\sigma ),n}}} \right\}$.  
\end{flushleft} 
\end{theorem} 
\begin{proof}
 In matrix form, the solution \eqref{eq:239} can be represented as the series \eqref{eq:242}. For the convergence of the series in square brackets on the right-hand side \eqref{eq:242}, it is sufficient to show the convergence of series of the norms of their members. Indeed, under conditions \eqref{eq:241}, we have that for any p and 
 $\left( {{\boldsymbol{\alpha} ^{n,p}},{\boldsymbol{\alpha} ^{n + 1,p}}} \right) \in {U_{ad}}$

\begin{eqnarray*} 
& {\left\| {{T_p}} \right\|_1} \le \Delta t\,(1 - \sigma )\dfrac{{{{\left\| {D{B^n}(\boldsymbol{1})} \right\|}_1}}}{{\mathop {\min }\limits_j \left( {{d_{jj}} + \Delta t\,\sigma \sum\limits_i {{d_{ii}}a_{ij}^{n + 1}} } \right)}} \le \gamma \quad \text  {and} \\ \quad  
& {\left\| {{Q_p}} \right\|_1} \le \Delta t\,\sigma \dfrac{{{{\left\| {D{B^{n + 1}}(\boldsymbol{1})} \right\|}_1}}}{{\mathop {\min }\limits_j \left( {{d_{jj}} + \Delta t\,\sigma \sum\limits_i {{d_{ii}}a_{ij}^{n + 1}} } \right)}} \le \gamma 
\end{eqnarray*}
Consequently, the series of the norms of members are majorized by the convergent numerical power series, as required.
\end{proof}

In the inequalities \eqref{eq:238}, ${b_{ij}}$ is multiplied by the differences of coordinates of the vectors ${\boldsymbol{y}^n}$ and ${\boldsymbol{y}^{n + 1,p}}$. These differences can be arbitrarily small and can lead to instability in the numerical solution of the linear programming problem \eqref{eq:237}-\eqref{eq:238}. Therefore, we reformulate the linear programming problem \eqref{eq:237}-\eqref{eq:238} in other variables.

Denote by ${\beta _k} = {\alpha _{k + {1 / 2}}}\left| {{c_{ik}}({\Delta _i}y)} \right|$ and ${s_{ik}} = {\mathop{\rm sgn}} ({c_{ik}})$, where ${c_{ik}}$ are elements of the matrix C from \eqref{eq:214}. Then we can rewrite \eqref{eq:237}-\eqref{eq:238} in the form
\begin{equation}
\label{eq:243}
\Im ({\boldsymbol{\beta} ^n},{\boldsymbol{\beta} ^{n + 1}}) = \sum\limits_k {\left( {\beta _k^n + \beta _k^{n + 1}} \right)}  \to \max  							
\end{equation}
\begin{equation}
\label{eq:244}
 - (1 - \sigma )\sum\limits_k {s_{ik}^n\beta _k^n}  - \sigma \sum\limits_k {s_{ik}^{n + 1}\beta _k^{n + 1}}  \ge \frac{1}{{\Delta t}}\left( {\mathop {\min }\limits_{j \in {S_i}} y_j^n - y_i^n} \right) + (1 - \sigma )\sum\limits_{j \ne i} {a_{ij}^n} \left( {y_j^n - y_i^n} \right) 		
\end{equation}
\begin{equation}
\label{eq:245}
 - (1 - \sigma )\sum\limits_k {s_{ik}^n\beta _k^n}  - \sigma \sum\limits_k {s_{ik}^{n + 1}\beta _k^{n + 1}}  \le \frac{1}{{\Delta t}}\left( {\mathop {\max }\limits_{j \in {S_i}} y_j^n - y_i^n} \right) + (1 - \sigma )\sum\limits_{j \ne i} {a_{ij}^n} \left( {y_j^n - y_i^n} \right)		
\end{equation}
\begin{equation}
\label{eq:246}
0 \le \beta _k^n \le \left| {c_{ik}^n} \right|, \quad 0 \le \beta _k^{n + 1} \le \left| {c_{ik}^{n + 1}} \right|								
\end{equation}

\subsection{Approximate Solution of the Optimization Problem} \label{Sec24}
We will look for a nonzero $\left( {\boldsymbol{\alpha} ^n},{\boldsymbol{\alpha} ^{n + 1}} \right) \in {U_{ad}}$ that satisfies the system of inequalities \eqref{eq:238}. To do this, we rewrite these inequalities in accordance with \eqref{eq:214} in the form 
\begin{equation}
\label{eq:247} 
\begin{split} 
(1 - \sigma )\sum\limits_k {\alpha _{k + {1 / 2}}^nc_{ik}^n(\boldsymbol{{\Delta _i}y}^n)} & + \sigma \sum\limits_k {\alpha _{k + {1 / 2}}^{n + 1}c_{ik}^{n + 1}(\boldsymbol{{\Delta _i}y}^{n + 1})}  \le \frac{1}{{\Delta t}}\left( {\mathop {\max }\limits_{j \in {S_i}} y_j^n - y_i^n} \right) \\
& + (1 - \sigma )\sum\limits_{j \ne i} {a_{ij}^n} \left( {y_j^n - y_i^n} \right) 
\end{split}
\end{equation}
\begin{equation}
\label{eq:248}
\begin{split} 
(1 - \sigma )\sum\limits_k {\alpha _{k + {1 / 2}}^nc_{ik}^n(\boldsymbol{{\Delta _i}y}^n)} & + \sigma \sum\limits_k {\alpha _{k + {1 / 2}}^{n + 1}c_{ik}^{n + 1}(\boldsymbol{{\Delta _i}y}^{n + 1})}  \ge \frac{1}{{\Delta t}}\left( {\mathop {\min }\limits_{j \in {S_i}} y_j^n - y_i^n} \right) \\
& + (1 - \sigma )\sum\limits_{j \ne i} {a_{ij}^n} \left( {y_j^n - y_i^n} \right) 
\end{split}
\end{equation}
\begin{equation}
\label{eq:249}
 0 \le \alpha _{k + {1 / 2}}^n \le 1, \quad  0 \le \alpha _{k + {1 / 2}}^{n + 1} \le 1		   						
\end{equation}

For the left-hand sides of the inequalities \eqref{eq:247} and \eqref{eq:248}, the following estimates are valid
\begin{equation}
\begin{split}
\label{eq:250}  
(1 - \sigma )\sum\limits_k {\alpha _{k + {1 / 2}}^nc_{ik}^n(\boldsymbol{{\Delta _i}y}^n)}  + \sigma \sum\limits_k {\alpha _{k + {1 / 2}}^{n + 1}c_{ik}^{n + 1}(\boldsymbol{{\Delta _i}y}^{n + 1})}  \le \\
 \le \mathop {\max }\limits_{k \in K(i)} {\left\{ {\alpha _{k + {1 / 2}}^n,\alpha _{k + {1 / 2}}^{n + 1}} \right\}^ + }\sum\limits_k {\left[ {(1 - \sigma )\max \left( {0,c_{ik}^n} \right) + \sigma \max \left( {0,c_{ik}^{n + 1}} \right)} \right]}  				
\end{split} 
\end{equation}
\begin{equation}
\begin{split}
\label{eq:251}
 (1 - \sigma )\sum\limits_k {\alpha _{k + {1 / 2}}^nc_{ik}^n(\boldsymbol{{\Delta _i}y}^n)}  + \sigma \sum\limits_k {\alpha _{k + {1 / 2}}^{n + 1}c_{ik}^{n + 1}(\boldsymbol{{\Delta _i}y}^{n + 1})}  \ge \\
  \ge \mathop {\max }\limits_{k \in K(i)} {\left\{ {\alpha _{k + {1 / 2}}^n,\alpha _{k + {1 / 2}}^{n + 1}} \right\}^ - }\sum\limits_k {\left[ {(1 - \sigma )\min \left( {0,c_{ik}^n} \right) + \sigma \min \left( {0,c_{ik}^{n + 1}} \right)} \right]} 				
\end{split} 
\end{equation}
where $\mathop {\max }\limits_{k \in K(i)} {\left\{ {\alpha _{k + {1 / 2}}^n,\alpha _{k + {1/2}}^{n + 1}} \right\}^ + }$ and $\mathop {\max }\limits_{k \in K(i)} {\left\{ {\alpha _{k + {1 / 2}}^n,\alpha _{k + {1 / 2}}^{n + 1}} \right\}^ - }$ are the maximal components of the vectors ${\boldsymbol{\alpha}^n}$ and $\boldsymbol{\alpha}^{n+1}$ for non-negative and non-positive terms on the left-hand sides of \eqref{eq:250} and \eqref{eq:251}, respectively; $K(i)$ is the set of subscript numbers {\it k} for the components of the vectors ${\boldsymbol{\alpha}^n}$ and $\boldsymbol{\alpha}^{n+1}$ in \eqref{eq:247} and \eqref{eq:248} that corresponds to the {\it i}-th node of the difference scheme. 
 
 Let ${\alpha _{k + {1 / 2}}}$ belongs to {\it k}-th column of the matrix {\it C} and corresponds to the elements ${c_{ik}}$ and ${c_{jk}}$. Substituting \eqref{eq:250} into \eqref{eq:247}, and \eqref{eq:251} into \eqref{eq:248} yields
\begin{equation}
\label{eq:252}
 {\alpha _{k + {1 / 2}}} = \left\{ {\begin{array}{*{20}{c}}{\min (R_i^ + ,R_j^ - ), \qquad {\rm{    }}{c_{ik}} > 0}\\{\min (R_i^ - ,R_j^ + ), \qquad {\rm{    }}{c_{ik}} < 0}\end{array}} \right. 								
\end{equation}
where 
\begin{equation}
\label{eq:253}
 R_i^ \pm  = \mathop {\max }\limits_{k \in K(i)} {\left\{ {\alpha _{k + {1 / 2}}^n,\alpha _{k + {1 / 2}}^{n + 1}} \right\}^ \pm } = \min (1,{{Q_i^ \pm } \mathord{\left/ {\vphantom {{Q_i^ \pm } {P_i^ \pm ),}}} \right. \kern-\nulldelimiterspace} {P_i^ \pm ),}} 					
\end{equation}
\begin{equation}
\label{eq:254}
 Q_i^ +  = \frac{1}{{\Delta t}}\left( {\mathop {\max }\limits_{j \in {S_i}} y_j^n - y_i^n} \right) + (1 - \sigma )\sum\limits_{j \ne i} {a_{ij}^n} \left( {y_j^n - y_i^n} \right) 						
\end{equation}
\begin{equation}
\label{eq:255}
 Q_i^ -  = \frac{1}{{\Delta t}}\left( {\mathop {\min }\limits_{j \in {S_i}} y_j^n - y_i^n} \right) + (1 - \sigma )\sum\limits_{j \ne i} {a_{ij}^n} \left( {y_j^n - y_i^n} \right)						
\end{equation}
\begin{equation}
\label{eq:256}
 P_i^ +  = \sum\limits_k {\left[ {(1 - \sigma )\max \left( {0,c_{ik}^n} \right) + \sigma \max \left( {0,c_{ik}^{n + 1}} \right)} \right]}  			\end{equation}
\begin{equation}
\label{eq:257}
 P_i^ -  = \sum\limits_k {\left[ {(1 - \sigma )\min \left( {0,c_{ik}^n} \right) + \sigma \min \left( {0,c_{ik}^{n + 1}} \right)} \right]} 			\end{equation}

Note that the formulas \eqref{eq:252}-\eqref{eq:257} are similar to the FCT formulas of Zalesak \cite{b2} for explicit schemes and the Kuzmin formulas \cite{b7,b5,b6,b4} for implicit schemes.

\section{Scalar Nonlinear Conservation Laws} \label{Sec3}
We consider the numerical solution of the initial value problem (IVP) for scalar conservation laws
\begin{equation}
\label{eq:31}
 \frac{{\partial \rho }}{{\partial t}} + \frac{\partial }{{\partial x}}f(\rho ) = 0 \qquad 	t > 0 	\quad  	 - \infty  < x < \infty  			\end{equation}
\begin{equation}
\label{eq:32}
 \rho (x,0) = {\rho ^0}(x) 										
\end{equation}
where {\it f} is a smooth flux function.

Solutions of nonlinear scalar conservation laws \eqref{eq:31}-\eqref{eq:32} may develop singularities in finite time, even for smooth initial data. Hence, such solutions of \eqref{eq:31}-\eqref{eq:32} are sought in the weak sense. Unfortunately, the weak solution of conservation laws is not unique. In order to select a physically relevant solution, the weak solution $\rho (x,t)$ of \eqref{eq:31} should satisfy the inequality
\begin{equation}
\label{eq:33}
 \frac{{\partial U(\rho )}}{{\partial t}} + \frac{{\partial F(\rho )}}{{\partial x}} \le 0 										
\end{equation}
in the sense of distribution for every entropy pair ({\it U},{\it F}), where $U(\rho )$ is a convex function with corresponding function $F(\rho )$ satisfying $F'(\rho ) = U'(\rho )f'(\rho )$. {\it U} is called an entropy function and {\it F} is called an entropy flux associated with the conservation law \eqref{eq:31}.

For the numerical solution of the scalar conservation laws \eqref{eq:31}, we apply the weighted conservative difference scheme \eqref{eq:26} for which the lower-order and high-order numerical fluxes $h_{i + {1/2}}^L = {h^L}(y_{i - l + 1}^{},...,y_{i + r}^{})$ and $h_{i + {1/2}}^H = {h^H}(y_{i - l + 1}^{},...,y_{i + r}^{})$ are the Lipschitz continuous and consistent with the differential flux $f(\rho )$, that is
${h^L}(\rho ,...,\rho ) = f(\rho )$ and	${h^H}(\rho ,...,\rho ) = f(\rho )$.

We consider the following three-point conservative scheme as the lower-order scheme
\begin{equation}
\label{eq:34a}
 y_i^{n + 1} - y_i^n + \frac{{\Delta t}}{{\Delta {x_i}}}\left[ {\sigma \left( {h_{i + {1 /2}}^{L,n + 1} - h_{i - {1 /2}}^{L,n + 1}} \right) + (1 - \sigma )\left( {h_{i + {1 /2}}^{L,n} - h_{i - {1 /2}}^{L,n}} \right)} \right] = 0		\end{equation}
where $h_{i + {1 /2}}^L = {h^L}(v,u)$ satisfies the inequalities 
\begin{equation}
\label{eq:35a}
 \frac{\partial {h^L}(v,u)}{\partial v} \ge 0	\qquad	\frac{\partial {h^L}(v,u)}{\partial u} \le 0 \qquad	\forall \; v,u \in R				\end{equation}

\begin{theorem}
\label{th8}
Assume that the numerical fluxes $h_{i + {1 /2}}^L$ of the three-point scheme \eqref{eq:34a} satisfy inequalities \eqref{eq:35a} for all ${y_i} \in R$. If for all i $\Delta t$ satisfies the conditions
 
\begin{equation}
\label{eq:36a}
\Delta t\;(1 - \sigma )\;\left[ {\frac{\partial }{{\partial {y_i}}}h_{i + {1 /2}}^L({y_i^n},{y_{i + 1}^n}) - \frac{\partial }{{\partial {y_i}}}h_{i - {1 /2}}^L({y_{i - 1}^n},{y_i^n})} \right] \le \Delta {x_i} 						
\end{equation}
then for the fully discrete scheme \eqref{eq:34a} the following inequalities hold

\begin{equation}
\label{eq:37a}
{\underline y}_i^n \le y_i^{n + 1} + \frac{{\Delta t}}{{\Delta {x_i}}}\sigma \left( {h_{i + {1 /2}}^L(y_i^{n + 1},y_{i + 1}^{n + 1}) - h_{i - {1 /2}}^L(y_{i - 1}^{n + 1},y_i^{n + 1})} \right) \le \bar y_i^n					\end{equation}

\begin{equation}
\label{eq:38a}
{\underline y}^n \le y_i^{n + 1} \le {\bar y^n} 								\end{equation}
where ${\underline y}_i = \min ({y_{i - 1}},{y_i},{y_{i + 1}})$, ${\bar y_i} = \max ({y_{i - 1}},{y_i},{y_{i + 1}})$, $\underline y = \mathop {\min }\limits_j {y_j}$ and $\bar y = \mathop {\max }\limits_j {y_j}$. 
\end{theorem}

\begin{proof}
 Consider the function 
\begin{equation*}
\Phi ({y_{i - 1}},{y_i},{y_{i + 1}}) = {y_i} - (1 - \sigma )\frac{{\Delta t}}{{\Delta {x_i}}}\left[ {h_{i + {1 /2}}^L({y_i},{y_{i + 1}}) - h_{i - {1 /2}}^L({y_{i - 1}},{y_i})} \right] 
\end{equation*}
which is non-decreasing for its arguments. Then we have

\begin{equation*}
 {\underline y}_i^n = \Phi ({\underline y}_i^n, {\underline y}_i^n, {\underline y}_i^n) \le \Phi (y_{i - 1}^n,y_i^n,y_{i + 1}^n) \le \Phi (\bar y_i^n,\bar y_i^n,\bar y_i^n) = \bar y_i^n 
\end{equation*}
 
Let $y_k^{n + 1} = {{\underline y}^{n + 1}}$. Then
\begin{equation*}
\begin{split}
 & {\underline y}^{n + 1} = {{\underline y}^{n + 1}} + \sigma \frac{{\Delta t}}{{\Delta {x_k}}}\left( {{h^L}({{\underline y}^{n + 1}},{{\underline y}^{n + 1}}) - {h^L}({{\underline y}^{n + 1}},{{\underline y}^{n + 1}})} \right) \\ 
 \ge {\underline y}^{n + 1} + \sigma \frac{\Delta t}{\Delta {x_k}} & \left( {{h^L}({{\underline y}^{n + 1}},y_{k + 1}^{n + 1}) - {h^L}(y_{k - 1}^{n + 1},{{\underline y}^{n + 1}})} \right) 
 = \Phi (y_{k - 1}^n,y_k^n,y_{k + 1}^n) \ge \Phi ({{\underline y}^n},{{\underline y}^n},{{\underline y}^n}) = {{\underline y}^n}
 \end{split}
\end{equation*}
Therefore, $y_i^{n + 1} \ge {\underline y}^n$ for all $i$. Similarly, we can show that $y_i^{n + 1} \le {\bar y^n}$ for all $i$.
\end{proof}

We discretize the entropy inequality \eqref{eq:33} by the following weighted conservative scheme 
\begin{equation}
\label{eq:39}
 U(y_i^{n + 1}) - U(y_i^n) + \frac{{\Delta t}}{{\Delta {x_i}}}\left[ {H_{i + {1 / 2}}^{(\sigma )} - H_{i - {1 / 2}}^{(\sigma )}} \right] \le 0			\end{equation}
where $H_{i + {1 / 2}}^{(\sigma )} = \sigma H_{i + {1 / 2}}^{n + 1} + (1 - \sigma )H_{i + {1 / 2}}^n$ and ${H_{i + {1 / 2}}} = H({y_{i - l + 1}},...,{y_{i + r}})$ is the numerical entropy flux consistent with the differential flux $H(u,...,u) = F(u)$. The numerical entropy flux $H({y_{i - l + 1}},...,{y_{i + r}})$ for {\it F} is not unique. We use the so-called proper numerical entropy flux introduced by Merriam \cite{b17}, Sonar \cite{b18}, and Zhao and Wu \cite{b19}. 

For the conservative scheme \eqref{eq:24}, we define the proper numerical entropy flux in a similar way to Zhao and Wu \cite{b19}.
\begin{defn}
 We call the numerical entropy flux $H({y_{i - l + 1}},...,{y_{i + r}})$ of the conservative scheme \eqref{eq:24} is proper, if
\begin{equation}
\label{eq:310}
 \frac{\partial }{{\partial {y_j}}}H({y_{i - l + 1}},...,{y_{i + r}}) = \frac{{dU({y_j})}}{{d{y_j}}}\frac{\partial }{{\partial {y_j}}}h({y_{i - l + 1}},...,{y_{i + r}})  \qquad	{\rm for} \quad j = i - l + 1,...,i + r 		
\end{equation}
\end{defn}

The proper numerical entropy flux for the difference scheme \eqref{eq:26} can be written in the form
\begin{equation}
\label{eq:311a}
H_{i + {1/2}}^{(\sigma )} = H_{i + {1/2}}^{L,(\sigma )} + \left( {\alpha {H^d}} \right)_{i + {1/2}}^{(\sigma )} 
\end{equation}
where $H_{i+{1/2}}^d = H_{i+{1/2}}^H - H_{i+{1/2}}^{L}$, $H_{i+{1/2}}^L$ and $H_{i+{1/2}}^H$ are the low-order and high-order proper numerical entropy fluxes corresponding to the low-order and high-order numerical fluxes  $h_{i+{1/2}}^L$ and $h_{i+{1/2}}^H$. 

\begin{theorem}
\label{th9}
 Let a convex function $U \in {C^2}(R)$, $f \in {C^1}(R)$ and for all i $\Delta t$ satisfies the conditions \eqref{eq:36a} and
\begin{equation}
\label{eq:312a} 
\begin{split} 
& (1 - \sigma ) \; \Delta t \; {\left( {U''({{\bar y}^n})} \right)^2}{\left( {h_{i + {1 /2}}^{L,n} - h_{i - {1 /2}}^{L,n}} \right)^2} \\
\le 2\Delta {x_i} \; & U''({\underline y}^n)\left[ {U'(y_i^n)\left( {h_{i + {1 /2}}^{L,n} - h_{i - {1 /2}}^{L,n}} \right) - H_{i + {1 /2}}^{L,n} + H_{i - {1 /2}}^{L,n}} \right]
\end{split} 
\end{equation}

Then the fully discrete scheme \eqref{eq:34a}-\eqref{eq:35a} satisfies the following cell entropy inequality

\begin{equation}
\label{eq:313a}
\begin{split}
U(y_i^{n + 1}) & - U(y_i^n) + \sigma \frac{{\Delta t}}{{\Delta {x_i}}}\left[ {H_{i + {1 /2}}^L(y_i^{n + 1},y_{i + 1}^{n + 1}) - H_{i + {1 /2}}^L(y_{i - 1}^{n + 1},y_i^{n + 1})} \right]  \\ 				
& + (1 - \sigma )\frac{{\Delta t}}{{\Delta {x_i}}}\left[ {H_{i + {1 /2}}^L(y_i^n,y_{i + 1}^n) - H_{i + {1 /2}}^L(y_{i - 1}^n,y_i^n)} \right] \le 0 
\end{split}
\end{equation}
where $H_{i + {1 /2}}^L$ is the proper numerical entropy flux corresponding to the numerical flux $h_{i + {1 /2}}^L$. 
\end{theorem}

\begin{proof}
 Multiplying the equation \eqref{eq:34a} by $U'(y_i^{n + 1})$ and subtracting it from the left-hand side of the inequality \eqref{eq:313a}, we have
\begin{equation}
\label{eq:314a}
\begin{split}
&  U(y_i^{n + 1}) - U(y_i^n) - U'(y_i^{n + 1})(y_i^{n + 1} - y_i^n) - (1 - \sigma )\frac{{\Delta t}}{{\Delta {x_i}}}\left( {U'(y_i^{n + 1}) - U'(y_i^n)} \right)\left( {h_{i + {1 /2}}^{L,n} - h_{i - {1 /2}}^{L,n}} \right)  \\
& \qquad \qquad \qquad + \sigma \frac{{\Delta t}}{{\Delta {x_i}}}\left[ {H_{i + {1 /2}}^{L,n + 1} - H_{i - {1 /2}}^{L,n + 1} - U'(y_i^{n + 1})\left( {h_{i + {1 /2}}^{L,n + 1} - h_{i - {1 /2}}^{L,n + 1}} \right)} \right] \\ 
& \qquad \qquad \qquad   + (1 - \sigma )\frac{{\Delta t}}{{\Delta {x_i}}}\left[ {H_{i + {1 /2}}^{L,n} - H_{i - {1 /2}}^{L,n} - U'(y_i^n)\left( {h_{i + {1 /2}}^{L,n} - h_{i - {1 /2}}^{L,n}} \right)} \right]  \\
&   = \left[ { - \frac{1}{2}U''(\xi ){{\left( {y_i^{n + 1} - y_i^n} \right)}^2} - (1 - \sigma )\frac{{\Delta t}}{{\Delta {x_i}}}U''(\eta )\left( {y_i^{n + 1} - y_i^n} \right)\left( {h_{i + {1 /2}}^{L,n} - h_{i - {1 /2}}^{L,n}} \right)} \right]   + \sigma \frac{\Delta t}{\Delta {x_i}} \\ 
& \times  \left[ {\int_{y_i^{n + 1}}^{y_{i + 1}^{n + 1}} {\left( {U'(v) - U'(y_i^{n + 1})} \right)\frac{\partial }{{\partial v}}h(y_i^{n + 1},v)dv + \int_{y_{i - 1}^{n + 1}}^{y_i^{n + 1}} {\left( {U'(v) - U'(y_i^{n + 1})} \right)\frac{\partial }{{\partial v}}h(v,y_i^{n + 1})dv} } } \right]  \\ 
& + (1 - \sigma )\frac{{\Delta t}}{{\Delta {x_i}}}\left[ {\int_{y_i^n}^{y_{i + 1}^n} {\left( {U'(v) - U'(y_i^n)} \right)\frac{\partial }{{\partial v}}h(y_i^n,v)dv + \int_{y_{i - 1}^n}^{y_i^n} {\left( {U'(v) - U'(y_i^n)} \right)\frac{\partial }{{\partial v}}h(v,y_i^n)dv} } } \right]
\end{split}
\end{equation}
where $\xi  = {\theta _1}y_i^{n + 1} + (1 - {\theta _1})y_i^n$, $\eta  = {\theta _2}y_i^{n + 1} + (1 - {\theta _2})y_i^n$ and $0 < {\theta _1},{\theta _2} < 1$. 
 
We show that the second and third terms in square brackets on the right-hand side of \eqref{eq:314a} are non-positive. Indeed, for a convex function $U(v) \in {C^1}(R)$, the inequality $\left[ {U'(v) - U'(w)} \right](v - w) \ge 0$ holds for any $v,w \in R$. Therefore, the sign of the expression $\left[ {U'(v) - U'(w)} \right]$ coincides with $sgn(v - w)$ and the integrands do not change the sign on the integration interval. Therefore,
\begin{equation*}
{\mathop{\rm sgn}} \left[ {\int_{{y_i}}^{{y_{i + 1}}} {\left( {U'(v) - U'({y_i})} \right)\frac{\partial }{{\partial v}}h({y_i},v)dv} } \right] =  - {{\mathop{\rm sgn}} ^2}({y_{i + 1}} - {y_i}) 
\end{equation*}

\begin{equation*}
{\mathop{\rm sgn}} \left[ {\int_{{y_{i - 1}}}^{{y_i}} {\left( {U'(v) - U'({y_i})} \right)\frac{\partial }{{\partial v}}h(v,{y_i})dv} } \right] =  - {{\mathop{\rm sgn}} ^2}({y_i} - {y_{i - 1}})
\end{equation*}

Substituting \eqref{eq:34a} into the first term in square brackets on the right-hand side of \eqref{eq:314a} and using the Young’s $\varepsilon$-inequality, we obtain that
\begin{equation}
\label{eq:315a}
\begin{split}
 - \frac{1}{2}U''(\xi ) & {\left( {y_i^{n + 1} - y_i^n} \right)^2} - (1 - \sigma )\frac{{\Delta t}}{{\Delta {x_i}}}U''(\eta )\left( {y_i^{n + 1} - y_i^n} \right)\left( {h_{i + {1 /2}}^{L,n} - h_{i - {1 /2}}^{L,n}} \right)  \\ 		& \qquad \le \frac{{{{(1 - \sigma )}^2}}}{2}{\left( {\frac{{\Delta t}}{{\Delta {x_i}}}} \right)^2}\frac{{{{\left( {U''(\eta )} \right)}^2}}}{{U''(\xi )}}{\left( {h_{i + {1 /2}}^{L,n} - h_{i - {1 /2}}^{L,n}} \right)^2} \\
& \qquad \le \frac{{{{(1 - \sigma )}^2}}}{2}{\left( {\frac{{\Delta t}}{{\Delta {x_i}}}} \right)^2}\frac{{{\left( {U''({\bar y}^n)} \right)}^2}}{U''({\underline y}^n)}{\left( {h_{i + {1/2}}^{L,n} - h_{i - {1/2}}^{L,n}} \right)^2} 
\end{split}
\end{equation}
This completes the proof of the theorem.
\end{proof}

Further, we consider the Rusanov scheme as low-order difference scheme \eqref{eq:34a} 
\begin{equation}
\label{eq:35}
 \frac{{\Delta {x_i}}}{{\Delta t}}\left( {y_i^{n + 1} - y_i^n} \right) + \sigma \left[ {h_{i + {1 / 2}}^{Rus,n + 1} - h_{i - {1 / 2}}^{Rus,n + 1}} \right] + (1 - \sigma )\left[ {h_{i + {1 / 2}}^{Rus,n} - h_{i - {1 / 2}}^{Rus,n}} \right] = 0 					 
\end{equation}
with the Rusanov numerical flux $h_{i + {1/2}}^{Rus}$
 
\begin{equation}
\label{eq:34}
 h_{i + {1 / 2}}^{Rus}({y_i},{y_{i + 1}}) = \frac{1}{2}\left[ {f({y_i}) + f({y_{i + 1}}) - \mathop {\max }\limits_{s \in [{y_i},{y_{i + 1}}]} \left| {f'(s)} \right|({y_{i + 1}} - {y_i})} \right] 				
\end{equation}
If {\it f} is a convex function than the expression $\mathop {\max }\limits_{s \in \left[ {{y_i},{y_{i + 1}}} \right]} \left| {f'(s)} \right|$ can be simplified to \[\mathop {\max }\limits_{s \in \left[ {{y_i},{y_{i + 1}}} \right]} \left| {f'(s)} \right|\; = \max \left( {\left| {f'({y_i})} \right|,\left| {f'({y_{i + 1}})} \right|} \right)\].

The proper numerical entropy flux for the Rusanov numerical flux can be written as
\begin{equation}
\label{eq:311}
  H_{i + {1/2}}^{Rus}({y_i},{y_{i + 1}}) = \frac{1}{2}\left[ {F({y_i}) + F({y_{i + 1}}) - \mathop {\max }\limits_{s \in [{y_i},{y_{i + 1}}]} \left| {f'(s)} \right|\left( {U({y_{i + 1}}) - U({y_i})} \right)} \right] 		 
\end{equation}
 
If $\Delta t$ satisfies the conditions \eqref{eq:36a} then we have
 
\begin{equation}
\label{eq:38}
 \mathop {\min }\limits_{j \in {S_i}} y_j^n \le y_i^{n + 1} + \sigma \frac{{\Delta t}}{{\Delta {x_i}}}\left[ {h_{i + {1 / 2}}^{Rus,n + 1} - h_{i - {1 / 2}}^{Rus,n + 1}} \right] \le \mathop {\max }\limits_{j \in {S_i}} y_j^n 	\end{equation}

Thus, to obtain a physically relevant solution of IVP \eqref{eq:31}-\eqref{eq:32} the flux limiters $\alpha _{i+{1/2}}^{n+1}$ and $\alpha _{i+{1/2}}^n$ should satisfy  
\begin{equation}
\label{eq:312}
  y_i^{n + 1} - y_i^n + \frac{{\Delta t}}{{\Delta {x_i}}}\left[ {h_{i + {1 / 2}}^{Rus,(\sigma )} + \left( {\alpha {h^d}} \right)_{i + {1 / 2}}^{(\sigma )} - h_{i - {1 / 2}}^{Rus,(\sigma )} - \left( {\alpha {h^d}} \right)_{i - {1 / 2}}^{(\sigma )}} \right] = 0				 
\end{equation}
\begin{equation}
\label{eq:313}
  U(y_i^{n + 1}) - U(y_i^n) + \frac{{\Delta t}}{{\Delta {x_i}}}\left[ {H_{i + {1/2}}^{Rus,(\sigma )} + \left( {\alpha {H^d}} \right)_{i + {1 / 2}}^{(\sigma )} - H_{i - {1 / 2}}^{Rus,(\sigma )} - \left( {\alpha {H^d}} \right)_{i - {1 / 2}}^{(\sigma )}} \right] \le 0			
\end{equation}
 
Let $\boldsymbol{w}$ is a grid function on ${\Omega _h}$. Multiplying the inequality \eqref{eq:312} by ${w_i}$ and then subtracting it from \eqref{eq:313}, we have
\begin{equation} 
\begin{split} 
\label{eq:314}
&  \left( {U(y_i^{n + 1})  - U(y_i^n)  - {w_i}(y_i^{n + 1} - y_i^n)} \right)   + \frac{{\Delta t}}{{\Delta {x_i}}}\left[ {H_{i+{1/2}}^{Rus,(\sigma )} - {w_i}h_{i+{1/2}}^{Rus,(\sigma )} - H_{i-{1/2}}^{Rus,(\sigma )} + {w_i}h_{i-{1/2}}^{Rus,(\sigma )}} \right]  \\
&  \le \frac{{\Delta t}}{{\Delta {x_i}}}\left[ { - \left( {\alpha {H^d}} \right)_{i+{1/2}}^{(\sigma )}  + {w_i}\left( {\alpha {h^d}} \right)_{i + {1/2}}^{(\sigma )} + \left( {\alpha {H^d}} \right)_{i-{1/2}}^{(\sigma )} - {w_i}\left( {\alpha {h^d}} \right)_{i-{1/2}}^{(\sigma )}} \right]	\end{split}
\end{equation}

Further, we will consider the inequality \eqref{eq:314} instead of \eqref{eq:313}.

\subsection{Finding Flux Limiters} \label{Sec31}
Usually, when numerically solving the IVP \eqref{eq:31}-\eqref{eq:32}, the modeling area is replaced by a large but limited area. Therefore, in further, we suggest that the flux limiters are finding in the limited area \[{U_{ad}} = \left\{ {\left. {({\boldsymbol{\alpha} ^n},{\boldsymbol{\alpha} ^{n + 1}})} \right| \quad {\rm{0}} \le \alpha _{i + {1/2}}^k \le {\rm{1, \quad - }}N \le i \le N; \quad k = n,n + 1} \right\}\]

We will consider the following optimization problem to find the flux limiters for the difference scheme \eqref{eq:312} at one time step.
\begin{equation}
\label{eq:315}
 \Im ({\boldsymbol{\alpha} ^n},{\boldsymbol{\alpha} ^{n + 1}}) = \left\langle {\boldsymbol{c_1},{\boldsymbol{\alpha} ^n}} \right\rangle  + \left\langle {\boldsymbol{c_2},{\boldsymbol{\alpha} ^{n + 1}}} \right\rangle  \to \mathop {\max }\limits_{{\boldsymbol{\alpha} ^n},{\boldsymbol{\alpha} ^{n + 1}} \in {U_{ad}}} \qquad	\boldsymbol{c_1},\boldsymbol{c_2} > 0 				
\end{equation}
subject to the inequalities \eqref{eq:314} and
\begin{equation}
\label{eq:316}
\begin{split}
 \frac{{\Delta {x_i}}}{{\Delta t}}\left( {\mathop {\min }\limits_{j \in {S_i}} y_j^n - y_i^n} \right) & + (1 - \sigma )\left( {h_{i + {1/2}}^{Rus,n} - h_{i - {1/2}}^{Rus,n}} \right) \le  
    \left( {\alpha {h^d}} \right)_{i - {1/2}}^{(\sigma )} - \left( {\alpha {h^d}} \right)_{i + {1/2}}^{(\sigma )} \\ 						
   & \le \frac{{\Delta {x_i}}}{{\Delta t}}\left( {\mathop {\max }\limits_{j \in {S_i}} y_j^n - y_i^n} \right) + (1 - \sigma )\left( {h_{i + {1/2}}^{Rus,n} - h_{i - {1/2}}^{Rus,n}} \right)  
\end{split}   	
\end{equation}

Note that due to \eqref{eq:314}, the optimization problem \eqref{eq:315}-\eqref{eq:316} is nonlinear for any $\sigma $. Then the iterative process to  find the flux limiters and to solve the difference scheme \eqref{eq:312}, similarly to the algorithm \eqref{eq:237}-\eqref{eq:240}, has the following form

\begin{enumerate}[\bfseries   Step 1.]
\item \label {it:21}
 Initialize positive numbers $\delta ,\; {\varepsilon _1},{\varepsilon _2} > 0$ and vectors $\boldsymbol{c_1},\boldsymbol{c_2} \ge 0$. Set $p = 0$, ${\boldsymbol{y}^{n + 1,0}} = {\boldsymbol{y}^n}$, ${\boldsymbol{\alpha} ^{n,0}},{\boldsymbol{\alpha} ^{n + 1,0}} = \boldsymbol{0}$.
\item \label {it:22}
 Find the solution ${\boldsymbol{\alpha} ^{n,p + 1}},{\boldsymbol{\alpha} ^{n + 1,p + 1}}$ of the following linear programming problem
\begin{equation}
\label{eq:317}
  \Im ({\boldsymbol{\alpha} ^{n,p + 1}},{\boldsymbol{\alpha} ^{n + 1,p + 1}}) = \sum\limits_k {\left( {{c_{1,k}}\;\alpha _{k + {1 / 2}}^{n,p + 1} + {c_{2,k}}\;\alpha _{k + {1 / 2}}^{n + 1,p + 1}} \right)}  \to \mathop {\max }\limits_{{\boldsymbol{\alpha} ^n},{\boldsymbol{\alpha} ^{n + 1}} \in {U_{ad}}}  				
\end{equation}
%\begin{equation} 
\begin{eqnarray} 
%\begin{split} 
& \dfrac{{\Delta {x_i}}}{{\Delta t}}\left( {\mathop {\min }\limits_{j \in {S_i}} y_j^n - y_i^n} \right) + (1 - \sigma )\left( {h_{i+{1/2}}^{Rus,n} - h_{i-{1/2}}^{Rus,n}} \right) \nonumber \\ 
 &  \le    (1 - \sigma )\left( { - \alpha _{i+{1/2}}^{n,p + 1} \; h_{i+{1/2}}^{d,n} + \alpha _{i-{1/2}}^{n,p + 1}\;h_{i-{1/2}}^{d,n}} \right) 
  + \sigma \left( { - \alpha _{i+{1/2}}^{n + 1,p + 1} \; h_{i+{1/2}}^{d,n+1,p} + \alpha _{i-{1/2}}^{n+1,p+1} \; h_{i-{1/2}}^{d,n+1,p}} \right) \nonumber \\	
&   \le \dfrac{{\Delta {x_i}}}{{\Delta t}}\left( {\mathop {\max }\limits_{j \in {S_i}} y_j^n - y_i^n} \right) 
   + (1 - \sigma )\left( {h_{i+{1/2}}^{Rus,n} - h_{i - {1 / 2}}^{Rus,n}} \right)   	
%\end{split}  
\label{eq:318}
\end{eqnarray}
%\end{equation}
%\begin{equation} 
%\begin{multiline} 
\begin{eqnarray} 
%\begin{split} 
\label{eq:319}
&  \dfrac{\Delta {x_i}}{\Delta t} \left( {U_i^{n + 1,p} - U_i^n - {w_i}(y_i^{n + 1,p} - y_i^n)} \right) + \sigma \left[ {H_{i + {1 / 2}}^{Rus,n + 1,p} - {w_i}h_{i + {1 / 2}}^{Rus,n + 1,p} }\right.  \\
& - \left. { H_{i - {1 / 2}}^{Rus,n + 1,p} + {w_i}h_{i - {1 / 2}}^{Rus,n + 1,p}} \right]   
   + (1 - \sigma )\left[ {H_{i+{1/2}}^{Rus,n} - {w_i}h_{i+{1/2}}^{Rus,n} - H_{i-{1/2}}^{Rus,n} + {w_i}h_{i-{1/2}}^{Rus,n}} \right] \nonumber\\    
&   \le \sigma \left( {\alpha _{i + {1 / 2}}^{n + 1,p + 1}\,\lambda _{ii + 1}^{n + 1,p}\; + \,\alpha _{i - {1 / 2}}^{n + 1,p + 1}\;\lambda _{ii - 1}^{n + 1,p}} \right) + (1 - \sigma )\left( {\alpha _{i + {1/2}}^{n,p + 1}\;\lambda _{ii + 1}^n + \alpha _{i - {1 / 2}}^{n,p + 1}\;\lambda _{ii - 1}^n} \right)  \nonumber 			   
%\end{split}  
\end{eqnarray} 
%\end{multiline} 
%\end{equation}
where ${\lambda_{ik}} = \left( {{w_i} h_{(i+k)/2}^d - H_{(i+k)/2}^d} \right) \;{\mathop{\rm sgn}} (k - i)$.
\item \label {it:31}
 For the ${\boldsymbol{\alpha} ^{n,p + 1}},{\boldsymbol{\alpha} ^{n + 1,p + 1}}$, we find $y_i^{n + 1,p + 1}$ from the system of equations
%\begin{equation} 
\begin{eqnarray} 
%\begin{split} 
\label{eq:320}
&  \dfrac{{\Delta {x_i}}}{{\Delta t}}y_i^{n + 1,p + 1} + \sigma \left( {h_{i+1/2}^{Rus,n+1,p+1} - h_{i-1/2}^{Rus,n+1,p+1}} \right) = \dfrac{{\Delta {x_i}}}{{\Delta t}}y_i^n - (1 - \sigma )\left( {h_{i+1/2}^{Rus,n} - h_{i-1/2}^{Rus,n}} \right)  \nonumber \\ &  \\ 
&  - \sigma \left( {\alpha _{i+{1/2}}^{n+1,p+1} \, h_{i+{1/2}}^{d,n + 1,p} - \alpha _{i-1/2}^{n+1,p+1} \, h_{i-1/2}^{d,n+1,p}} \right) -  
  (1 - \sigma ) \left( {\alpha _{i+1/2}^{n,p+1}\,h_{i+1/2}^{d,n} - \,\alpha _{i-1/2}^{n,p+1} \; h_{i+1/2}^{d,n}} \right) \nonumber 		
%\end{split}  
\end{eqnarray}
%\end{equation}
\item \label {it:41}
 Algorithm stop criterion
\begin{equation}
\label{eq:321}
\begin{split}
\frac{{\left| {y_i^{n + 1,p + 1} - y_i^{n + 1,p}} \right|}}{{\max \left( {\delta ,\left| {y_i^{n + 1,p + 1}} \right|} \right)}} < {\varepsilon _1}  \;
  \left| {\alpha _{i+1/2}^{n+1,p+1} - \alpha _{i+1/2}^{n+1,p}} \right| < {\varepsilon _2} \; \text {and} \;\; \left| {\alpha _{i + {1 / 2}}^{n,p + 1} - \alpha_{i+1/2}^{n,p}} \right| < {\varepsilon _2} 		
\end{split}   
\end{equation}
If the conditions \eqref{eq:321} hold, then ${\boldsymbol{y}^{n + 1}} = {\boldsymbol{y}^{n + 1,p + 1}}$, otherwise $p = p + 1$ and go to {\bf Step~\ref{it:2}}. 
\end{enumerate}

\begin{theorem}
\label{th:10}
  Let $U \in {C^2}(R)$, $f \in C^1(R)$. If $\Delta t$ satisfies 
 \eqref{eq:36a} and \eqref{eq:312a} then the linear programming problem \eqref{eq:317}-\eqref{eq:319} is solvable.
\end{theorem}
\begin{proof}
 To prove that the problem \eqref{eq:317}-\eqref{eq:319} is solvable, it is sufficient to show that the cost functional $\Im ({\boldsymbol{\alpha} ^n},{\boldsymbol{\alpha} ^{n + 1}})$ is bounded and the feasible set is non-empty. The boundedness of the functional \eqref{eq:317} follows from the boundedness of the vectors ${\boldsymbol{\alpha} ^n}$ and ${\boldsymbol{\alpha} ^{n + 1}}$ on ${\Omega _h}$, whose coordinates vary from zero to one. 

For ${\boldsymbol{\alpha} ^n},{\boldsymbol{\alpha}^{n+1}} = \boldsymbol{0}$, the inequalities \eqref{eq:318} and \eqref{eq:319} follow from the theorems~\ref{th8} and \ref{th9}. Therefore, the feasible set is non-empty that completes the proof.
\end{proof}

\subsection{Approximate Solution of the Optimization Problem} \label{Sec32}
A nontrivial admissible solution of the optimization problem \eqref{eq:317}-\eqref{eq:319} satisfies the inequalities \eqref{eq:318} and \eqref{eq:319}. From the inequality \eqref{eq:318}, similarly to Section~\ref{Sec24}, we have
\begin{equation}
\label{eq:323}
{\overline \alpha  _{i+1/2}} = \left\{ {\begin{array}{*{20}{c}}{\min (R_i^ + ,R_{i + 1}^ - ), \quad {\rm{ }} h_{i + 1/2}^d < 0}\\{\min (R_i^ - ,R_{i + 1}^ + ), \quad {\rm{ }} h_{i+1/2}^d > 0}\end{array}} \right.								
\end{equation}

Values $R_i^ \pm $ are calculated by the formulas
\begin{equation}
\label{eq:324}
 R_i^ \pm  = \min (1,{{Q_i^ \pm } \mathord{\left/ {\vphantom {{Q_i^ \pm } {P_i^ \pm ),}}} \right. \kern-\nulldelimiterspace} {P_i^ \pm )}} 			\end{equation}
where
\begin{eqnarray*}
& P_i^ -  = (1-\sigma) \left[ {\min \left( {0, - h_{i+1/2}^{d,n}} \right) + \min \left( {0,h_{i-1/2}^{d,n}} \right) } \right]  +\sigma \left[ {\min \left( {0, - h_{i+1/2}^{d,n+1}} \right) + \min \left( {0,h_{i-1/2}^{d,n+1}} \right) } \right] \\						
& P_i^ +  = (1-\sigma) \left[ {\max \left( {0, - h_{i+1/2}^{d,n}} \right) + \max \left( {0,h_{i-1/2}^{d,n}} \right) }\right] +\sigma \left[ {\max \left( {0, - h_{i+1/2}^{d,n+1}} \right) + \max \left( {0,h_{i-1/2}^{d,n+1}} \right) }\right] 					\end{eqnarray*}
	
\[  Q_i^ +  = \frac{\Delta {x_i}}{\Delta t}\left( {\mathop {\max }\limits_{j \in {S_i}} y_j^n - y_i^n} \right) + (1-\sigma) \left[ {h_{i+1/2}^{Rus,n} - h_{i-1/2}^{Rus,n}} \right] \]
\[  Q_i^ -  = \frac{\Delta {x_i}}{\Delta t}\left( {\mathop {\min }\limits_{j \in {S_i}} y_j^n - y_i^n} \right) +(1-\sigma) \left[ {h_{i+1/2}^{Rus,n} - h_{i-1/2}^{Rus,n}} \right] \]

On the other hand, choosing the largest value among the components of the vectors $\boldsymbol{\alpha} ^n$ and $\boldsymbol{\alpha}^{n+1}$ for negative terms, we can estimate the expression on the right-hand side of inequality \eqref{eq:319}. Then
\begin{equation} 
\begin{split} 
\label{eq:325}
 \mathop {\max }\limits_{\lambda_{ik}^n,\lambda_{ik}^{n + 1} < 0} \left\{ {\alpha _{k + {1 / 2}}^n,\alpha _{k + {1 / 2}}^{n + 1}} \right\}\sum\limits_{k = i - 1,i + 1} {\left[ {(1 - \sigma )\min (0,\lambda _{ik}^n) + \sigma \min (0,\lambda _{ik}^{n + 1})} \right]}  \\			
 \le \sum\limits_{k = i - 1,i + 1} {\left[ {\alpha _{(i+k)/2}^n \; (1 - \sigma )\, \lambda _{ik}^n + \alpha _{(i+k)/2}^{n + 1} \; \sigma \, \lambda _{ik}^{n + 1}} \right]} 
\end{split}  
\end{equation}

Substituting \eqref{eq:325} into \eqref{eq:319}, we obtain
\begin{equation}
\label{eq:326} 
\begin{split} 
 {\tilde \alpha _{i+1/2}} = \min \left\{ 1, \frac{-W_i}{Y_i}\min \left( {0,{\mathop{\rm sgn}} {\lambda _{ii+1}}} \right) + \max \left( {0,{\mathop{\rm sgn}} {\lambda _{ii+1}}} \right),  \right.  \\ 
 \left. 
 \frac{-W_{i+1}}{Y_{i+1}} \min \left( {0,{\mathop{\rm sgn}} {\lambda _{i+1i}}} \right) + \max \left( {0,{\mathop{\rm sgn}} {\lambda _{i+1i}}} \right) \right\}  
\end{split} 
\end{equation}
where
\begin{eqnarray*}
& {W_i} = \dfrac{{\Delta {x_i}}}{{\Delta t}}\left( {U_i^{n + 1} - U_i^n - {w_i}(y_i^{n + 1} - y_i^n)} \right) + \sigma \left[ {H_{i + {1 / 2}}^{Rus,n + 1} - {w_i}h_{i + {1 / 2}}^{Rus,n + 1} - H_{i - {1 / 2}}^{Rus,n + 1} + {w_i}h_{i-{1/2}}^{Rus,n+1}} \right]  \\
 &  + (1 - \sigma )\left[ {H_{i+{1/2}}^{Rus,n} - {w_i}h_{i+{1/2}}^{Rus,n} - H_{i - {1/2}}^{Rus,n} + {w_i}h_{i-{1/2}}^{Rus,n}} \right] \\ 
& {Y_i} = \sum\limits_{k = i-1,i+1} {\left[ {(1 - \sigma )\min (0,\lambda _{ik}^n) + \sigma \min (0,\lambda _{ik}^{n+1})} \right]}   
\end{eqnarray*}

Thus, from \eqref{eq:323} and \eqref{eq:326} follow that the admissible nontrivial solution can be calculated as
\begin{equation}
\label{eq:327}
  {\alpha _{i+{1/2}}} = \min ({\overline \alpha  _{i+{1/2}}},{\tilde \alpha _{i + {1/2}}}) 									
\end{equation}

\section{Two-Dimensional Convection-Diffusion Equation} \label{Sec4}
On an example of the two-dimensional convection-diffusion equation, we show how the results obtained in Sections~\ref{Sec2} and \ref{Sec3} can be extended to the multidimensional case.

We consider the two-dimensional linear convection-diffusion equation:
\begin{equation}
\label{eq:41}
  \frac{{\partial \rho }}{{\partial t}} + \sum\limits_{p = 1}^2 {\frac{\partial }{{\partial {x^{(p)}}}}\left( {{u^{(p)}}(x,t)\rho  - {k^{(p)}}(x,t)\frac{{\partial \rho }}{{\partial {x^{(p)}}}}} \right)}  = 0, \qquad   x = ({x_1},{x_2}) \in \left[ {{a_1},{b_1}} \right] \times \left[ {{a_2},{b_2}} \right]  \quad	t > 0	
\end{equation}
with initial condition
\begin{equation}
\label{eq:42}
\rho \left( {x,0} \right) = {\rho ^0}\left( x \right)						\end{equation}

For a unique solution of \eqref{eq:41}-\eqref{eq:42}, the boundary conditions are set at the boundary of the modeling area. Usually, the boundary conditions are specified either by the values of the function $\rho $, or by its fluxes on the boundary surfaces. The coefficients ${k^{(p)}}$ are bounded below and above, i.e.
\[0 < {c_1} \le {k^{(p)}}(x,t) \le {c_2}  \qquad	{c_1},{c_2} = const > 0\]
 
We introduce a non-uniform grid ${\Omega _h}$ on $\left[ {{a_1},{b_1}} \right] \times \left[ {{a_2},{b_2}} \right]$ 
\[{\Omega _h} = \left\{ {(x_i^{(1)},x_j^{(2)}): \quad {\rm{  }}x_{j+1}^{(p)} = x_j^{(p)} + h_j^{(p)}, \quad {\rm{  }}j = \overline {1,{N_p} - 1} ; \quad {\rm{ }}x_1^{(p)} > {a_p}, \, {\rm{ }}x_N^{(p)} < {b_p},{\rm{  }} \quad p = 1,2} \right\}\]

We approximate the problem \eqref{eq:41}-\eqref{eq:42} by the following conservative difference scheme with weights 
\begin{equation}
\label{eq:43} 
\begin{split} 
y_{ij}^{n + 1} - y_{ij}^n + \frac{{\Delta t}}{\Delta x_i^{(1)}}\left[ {f_{i + {1 / 2}j}^{(\sigma )} + F_{i + {1 / 2}j}^{(\sigma )} - f_{i - {1 / 2}j}^{(\sigma )} - F_{i - {1 / 2}j}^{(\sigma )}} \right] \\ 
+ \frac{{\Delta t}}{{\Delta x_j^{(2)}}}\left[ {h_{ij + {1 / 2}}^{(\sigma )} + H_{ij + {1 / 2}}^{(\sigma )} - h_{ij - {1 / 2}}^{(\sigma )} - H_{ij - {1 / 2}}^{(\sigma )}} \right] = 0	 
\end{split} 
\end{equation}
\begin{equation}
\label{eq:44}
\boldsymbol{y} = {\boldsymbol{y}^0}											
\end{equation}

The numerical fluxes are defined as follows
\[f_{i + {1 / 2}j}^{(\sigma )} = f_{i + {1 / 2}j}^{L,(\sigma )} + \left( {\alpha {f^d}} \right)_{i + {1 / 2}j}^{(\sigma )} \qquad 
 f_{i + {1 / 2}j}^d = \max \left( {0,f_{i + {1 / 2}j}^H - f_{i + {1 / 2}j}^L - f_{i + {1 / 2}j}^s} \right)\] 
\[F_{i + {1 / 2}j}^{(\sigma )} = \sigma \min \left( {0,f_{i + {1 / 2}j}^{H,n + 1} - f_{i + {1 / 2}j}^{L,n + 1} - f_{i + {1 / 2}j}^{s,n + 1}} \right) + (1 - \sigma )\min \left( {0,f_{i + {1 / 2}j}^{H,n} - f_{i + {1 / 2}j}^{L,n} - f_{i + {1 / 2}j}^{s,n}} \right)\]
\[h_{ij + {1 / 2}}^{(\sigma )} = h_{ij + {1 / 2}}^{L,(\sigma )} + \left( {\alpha {h^d}} \right)_{ij + {1 / 2}}^{(\sigma )} \qquad 
 h_{ij+1/2}^d = \max \left( {0,h_{ij+1/2}^H - h_{ij+1/2}^L - h_{ij+1/2}^s} \right) \] 
\[H_{ij +1/2}^{(\sigma )} = \sigma \min \left( {0,h_{ij+1/2}^{H,n + 1} - h_{ij + {1 / 2}}^{L,n + 1} - h_{ij +1/2}^{s,n + 1}} \right) + (1 - \sigma )\min \left( {0,h_{ij +1/2}^{H,n} - h_{ij +1/2}^{L,n} - h_{ij +1/2}^{s,n}} \right)\]
where ${f^H}$,${h^H}$ and ${f^L}$,\,${h^L}$ are the consistent numerical fluxes of high-order and low-order accuracy for convective differential fluxes, and ${f^s}$,\,${h^s}$ are the consistent numerical fluxes with the diffusive fluxes. 

Then, for example, setting \; $f_{i+1/2 j}^L = u_{i+1/2 j}^ + {y_{ij}} + u_{i+1/2 j}^ - {y_{i+1 j}}$, \; 
$f_{i+1/2 j}^H = 0.5{u_{i+1/2 j}}\left( {{y_{ij}} + {y_{i+1 j}}} \right)$, \;
$f_{i+1/2 j}^s = k_{i+1/2 j}^{(1)} \; \left( {y_{i+1 j} - y_{ij}} \right)  \mathord{\left/ {\vphantom {{\left( {y_{i + 1j} - y_{ij}} \right)} {{\Delta _{i+1/2}}{x^{(1)}}}}} \right. \kern-\nulldelimiterspace} {\Delta _{i+1/2}{x^{(1)}} } $ and similarly $h_{i j+1/2}^L$, \, $h_{i j+1/2}^H$, \, $h_{i j+1/2 }^f$ we rewrite \eqref{eq:43}-\eqref{eq:44} in the matrix form
\begin{equation}
\label{eq:45} 
\begin{split} 
 \left[ {E + \sigma \,\Delta t\left( {{A^{n + 1}} - {B^{n + 1}}(\boldsymbol{\alpha} )} \right)} \right]{\boldsymbol{y}^{n + 1}} = \left[ {E - (1 - \sigma )\,\Delta t\left( {{A^n} - {B^n}(\boldsymbol{\alpha} )} \right)} \right]{\boldsymbol{y}^n}  \\
 + \Delta t\left[ {\sigma 
   {\boldsymbol{g}^{n + 1}} + (1 - \sigma ){\boldsymbol{g}^n}} \right]   
\end{split}   
\end{equation}
where $\boldsymbol{ g}$ is the vector of boundary conditions. The matrices $ A = \left\{ {a_{rq}} \right\}_r^q $ and $B = \left\{ {{b_{rq}}} \right\}_r^q$ are five-diagonal matrices whose elements are calculated by the relations
%\begin{equation}
\begin{eqnarray}
%\begin{displaymath}
\label{eq:46}
%\begin{split}
 {a_{rr}} =  & \nonumber \\
\dfrac{1}{{\Delta x_i^{(1)}}} &  \left[ {u_{i+1/2 j}^{(1) + } - u_{i-1/2 j}^{(1) - } + \max \left( {0,\dfrac{{k_{i+1/2 j}^{(1)}}}{{{\Delta _{i +1/2}}{x^{(1)}}}} - \dfrac{{\left| {u_{i+1/2 j}^{(1)}} \right|}}{2}} \right) + \max \left( {0,\dfrac{{k_{i - {1 / 2}j}^{(1)}}}{{{\Delta _{i -1/2}}{x^{(1)}}}} - \dfrac{{\left| {u_{i -1/2 j}^{(1)}} \right|}}{2}} \right)} \right]  \nonumber\\ & \\
 +  \dfrac{1}{{\Delta x_j^{(2)}}} & \left[ {u_{ij+1/2}^{(2)+ } - u_{ij - {1 / 2}}^{(2) - } + \max \left( {0,\dfrac{{k_{ij+1/2}^{(2)}}}{{{\Delta _{j +1/2}}{x^{(2)}}}} - \dfrac{{\left| {u_{ij +1/2}^{(2)}} \right|}}{2}} \right) + \max \left( {0,\dfrac{{k_{ij -1/2}^{(2)}}}{{{\Delta _{j - {1 / 2}}}{x^{(2)}}}} - \dfrac{{\left| {u_{ij - {1 / 2}}^{(2)}} \right|}}{2}} \right)} \right]	\nonumber	
%\end{split}  
%\end{equation}
\end{eqnarray}
%\end{displaymath}
\begin{equation}
\label{eq:47}
  {a_{rr \mp 1}} =  \mp \frac{1}{{\Delta x_i^{(1)}}}\left[ {u_{i \mp {1 \mathord{\left/ {\vphantom {1 {2j}}} \right. \kern-\nulldelimiterspace} {2j}}}^{(1) \pm } \pm \max \left( {0,\frac{{k_{i \mp {1 / 2}j}^{(1)}}}{{{\Delta _{i \mp {1 / 2}}}{x^{(1)}}}} - \frac{{\left| {u_{i \mp {1 / 2}j}^{(1)}} \right|}}{2}} \right)} \right]  						
\end{equation}
\begin{equation}
\label{eq:48}
  {a_{rr \mp {N_1}}} =  \mp \frac{1}{{\Delta x_j^{(2)}}}\left[ {u_{ij \mp {1 / 2}}^{(2) \pm } \pm \max \left( {0,\frac{{k_{ij \mp {1 / 2}}^{(2)}}}{{{\Delta _{j \mp {1 / 2}}}{x^{(2)}}}} - \frac{{\left| {u_{ij \mp {1 / 2}}^{(2)}} \right|}}{2}} \right)} \right]  					\end{equation}
\begin{eqnarray}
%\begin{equation}
\label{eq:49}
%\begin{split}
 {b_{rr}}\left( \boldsymbol{\alpha}  \right) = &   \nonumber \\
\dfrac{{ - 1}}{{\Delta x_i^{(1)}}} & \left[ {\alpha _{i - {1 / 2}j}^{(1)}\min \left( {0,\dfrac{{k_{i -1/2 j}^{(1)}}}{{{\Delta _{i - {1 / 2}}}{x^{(1)}}}} - \dfrac{{\left| {u_{i -1/2 j}^{(1)}} \right|}}{2}} \right) + \alpha _{i +1/2 j}^{(1)}\min \left( {0,\dfrac{{k_{i + {1 / 2}j}^{(1)}}}{{{\Delta _{i + {1 / 2}}}{x^{(1)}}}} - \dfrac{{\left| {u_{i + {1 / 2}j}^{(1)}} \right|}}{2}} \right)} \right]  \nonumber \\ & \\
  - \dfrac{1}{{\Delta x_j^{(2)}}} & \left[ {\alpha _{ij - {1 / 2}}^{(2)}\min \left( {0,\dfrac{{k_{ij - {1 / 2}}^{(2)}}}{{{\Delta _{j - {1 / 2}}}{x^{(2)}}}} - \dfrac{{\left| {u_{ij - {1 / 2}}^{(2)}} \right|}}{2}} \right) + \alpha _{ij + {1 / 2}}^{(2)}\min \left( {0,\dfrac{{k_{ij + {1 / 2}}^{(2)}}}{{{\Delta _{j + {1 / 2}}}{x^{(2)}}}} - \dfrac{{\left| {u_{ij + {1 / 2}}^{(2)}} \right|}}{2}} \right)} \right]  \nonumber 		
%\end{split}  
%\end{equation}
\end{eqnarray}
\begin{equation}
\label{eq:410}
  {b_{rr \mp 1}}\left( \boldsymbol{\alpha}  \right) = \frac{1}{{\Delta x_i^{(1)}}}\alpha _{i \mp {1 / 2}j}^{(1)}\min \left( {0,\frac{{k_{i \mp {1 / 2}j}^{(1)}}}{{{\Delta _{i \mp {1 / 2}}}{x^{(1)}}}} - \frac{{\left| {u_{i \mp {1 / 2}j}^{(1)}} \right|}}{2}} \right)  					
\end{equation}
\begin{equation}
\label{eq:411}
  {b_{rr \mp {N_1}}}\left( \boldsymbol{\alpha}  \right) = \frac{1}{{\Delta x_j^{(2)}}}\alpha _{ij \mp {1 / 2}}^{(2)}\min \left( {0,\frac{{k_{ij \mp {1 / 2}}^{(2)}}}{{{\Delta _{j \mp {1 / 2}}}{x^{(2)}}}} - \frac{{\left| {u_{ij \mp {1 / 2}}^{(2)}} \right|}}{2}} \right) , \quad		r = i + (j - 1){N_1} 	
\end{equation}
where \, ${u^{(p) \pm }} = 0.5\left( {{u^{(p)}} \pm \left| {{u^{(p)}}} \right|} \right)$, \, $u_{i+1/2 j}^{(1)} = {u^{(1)}}\left( {x_{i+1/2}^{(1)},x_j^{(2)},t} \right)$, \, $x_{i + 1/2 j}^{(1)} = 0.5\left( {x_{ij}^{(1)} + x_{i+1 j}^{(1)}} \right)$.

It is easy to verify that the matrices $A$ and $B$ satisfy the conditions \eqref{eq:212}-\eqref{eq:215}. As the matrix $D$, we can take the diagonal matrix with elements ${d_{rr}} = \Delta x_i^{(1)}\Delta x_j^{(2)}$. The theorems \ref{th:1} – \ref{th:4} are valid for the system of equations \eqref{eq:45}, and the iterative procedure \eqref{eq:237}-\eqref{eq:240} can be used to solve it. Therefore, the flux limiters can be calculated by using the formulas \eqref{eq:252}-\eqref{eq:257}.

\section{Numerical Results} \label{Sec5}
In our calculations, we apply GLPK (GNU Linear Programming Kit) v.4.65 set of routines for solving linear programming, mixed integer programming, and other related problem. GLPK is available at $https://www.gnu.org/software/glpk/$. 

For numerical solving of the IVP \eqref{eq:31}-\eqref{eq:32}, we use the weighted conservative scheme \eqref{eq:312}, in which the centered space flux was chosen as the high-order flux. In this case, the relations \eqref{eq:312}-\eqref{eq:313} can be written in the form
\begin{equation}
\label{eq:51}
\begin{split}
  y_i^{n + 1} - y_i^n + \frac{\Delta t}{\Delta {x_i}}\left[ {h_{i+1/2}^{Rus,(\sigma )} + {{\left( {\frac{\alpha _{i+1/2}}{2}\mathop {\max }\limits_{s \in [{y_i},{y_{i + 1}}]} \left| {f'(s)} \right|{\Delta _{i + {1 /2}}}y} \right)}^{(\sigma )}} } \right. \\ \left. {
  - h_{i -1/2}^{Rus,(\sigma )} - {{\left( {\frac{\alpha _{i - 1/2}}{2}\mathop {\max }\limits_{s \in [{y_{i-1}},{y_i}]} \left| {f'(s)} \right|{\Delta _{i -1/2}}y} \right)}^{(\sigma )}}} \right] = 0	
\end{split}  		
\end{equation}

\begin{equation}
\begin{split}
\label{eq:52}
U(y_i^{n + 1}) - U(y_i^n) + \frac{{\Delta t}}{{\Delta {x_i}}}\left[ {H_{i+1/2}^{Rus,(\sigma )} + \left( {\frac{{{\alpha _{i+1/2}}}}{2}\mathop {\max }\limits_{s \in [{y_i},{y_{i + 1}}]} \left| {f'(s)} \right|{\Delta _{i+1/2}}U} \right)_{i+1/2}^{(\sigma )}} \right. \\
  - \left. {H_{i-1/2}^{Rus,(\sigma )} - \left( {\frac{{{\alpha _{i-1/2}}}}{2}\mathop {\max }\limits_{s \in [{y_{i-1}},{y_i}]} \left| {f'(s)} \right|{\Delta _{i-1/2}}U} \right)_{i-1/2}^{(\sigma )}} \right] \le 0
\end{split}   			
\end{equation}

Then, from \eqref{eq:51}-\eqref{eq:52} we have the following inequalities for the finding of flux limiters
\begin{equation}
\label{eq:53} 
\begin{split} 
&  \frac{{\Delta {x_i}}}{{\Delta t}}\left( {\mathop {\min }\limits_{j \in {S_i}} y_j^n - y_i^n} \right) + (1 - \sigma )\left( {h_{i + {1 /2}}^{Rus,n} - h_{i - {1 /2}}^{Rus,n}} \right)  \\
& \le (1 - \sigma )\left( { - \frac{{\alpha _{i + {1 /2}}^n}}{2}\mathop {\max }\limits_{s \in \left[ {y_i^n,y_{i + 1}^n} \right]} \left| {f'(s)} \right|\;{\Delta _{i + {1 /2}}}{y^n} + \frac{{\alpha _{i - {1 /2}}^n}}{2}\mathop {\max }\limits_{s \in \left[ {y_{i - 1}^n,y_i^n} \right]} \left| {f'(s)} \right|\;{\Delta _{i - {1 /2}}}{y^n}} \right)    		\\	 
& + \sigma \left( { - \frac{{\alpha _{i + {1 /2}}^{n + 1}}}{2}\mathop {\max }\limits_{s \in \left[ {y_i^{n + 1},y_{i + 1}^{n + 1}} \right]} \left| {f'(s)} \right|\;{\Delta _{i + {1 /2}}}{y^{n + 1}} + \frac{{\alpha _{i - {1 /2}}^{n + 1}}}{2}\mathop {\max }\limits_{s \in \left[ {y_{i - 1}^{n + 1},y_i^{n + 1}} \right]} \left| {f'(s)} \right|\;{\Delta _{i - {1 /2}}}{y^{n + 1}}} \right)   \\
&   \le \frac{{\Delta {x_i}}}{{\Delta t}}\left( {\mathop {\max }\limits_{j \in {S_i}} y_j^n - y_i^n} \right) + (1 - \sigma )\left( {h_{i + {1 /2}}^{Rus,n} - h_{i - {1 /2}}^{Rus,n}} \right)  	
\end{split} 
\end{equation}
\begin{equation}
\begin{split} 
&  \frac{{\Delta {x_i}}}{{\Delta t}}\left( {U_i^{n + 1} - U_i^n - {w_i}(y_i^{n + 1} - y_i^n)} \right) + \sigma \left[ {H_{i+1/2}^{Rus,n + 1} - {w_i}h_{i +1/2}^{Rus,n + 1} - H_{i-1/2}^{Rus,n+1} + {w_i}h_{i-1/2}^{Rus,n+1}} \right]  \\
&  + (1 - \sigma )\left[ {H_{i + 1/2}^{Rus,n} } \right.  
   - {w_i}h_{i +1/2}^{Rus,n} 
- \left. {H_{i - 1/2}^{Rus,n} + {w_i}h_{i - 1/2}^{Rus,n}} \right] 
 \le \sigma \left( {\alpha _{i+1/2}^{n + 1} \, \lambda _{ii+1}^{n+1}\; - \alpha _{i-1/2}^{n + 1} \, \lambda _{ii-1}^{n+1} \; } \right) \\ 
& + (1 - \sigma )\left( {\alpha _{i+1/2}^n \; \lambda _{ii + 1}^n \; - \alpha _{i-1/2}^n \; \lambda _{ii - 1}^n \,} \right)   
\label{eq:54} 
\end{split}				
\end{equation}
where ${\lambda _{ik}} = \frac{1}{2}\mathop {\max }\limits_{s \in \left[ {{y_i},{y_k}} \right]} \left| {f'(s)} \right| \;\; \left( {w_i \; {\Delta _{(i + k)/2}} y - {\Delta _{(i + k)/2}}U} \right)\;$.
  
Also, for the numerical solving of the IVP \eqref{eq:31}-\eqref{eq:32} we apply the Godunov scheme with numerical flux 
\begin{equation}
\label{eq:55}
  f_{i + {1 /2}}^G({y_i},{y_{i + 1}}) = \left\{ {\begin{array}{*{20}{c}}{\mathop {\min \,}\limits_{{y_i} \le y \le {y_{i + 1}}} f(y) \quad {\rm{if }} \; {y_i} \le {y_{i + 1}}}\\{\mathop {\max \,}\limits_{{y_{i + 1}} \le y \le {y_i}} f(y)  \quad {\rm{if }} \; {y_i} > {y_{i + 1}}}\end{array}} \right.  						
\end{equation}

The Rusanov scheme and the Godunov scheme are $E$-schemes \cite{b22,b24}. Osher~\cite{b22,b23} showed that E schemes satisfy the entropy inequality and have no less numerical viscosity than the Godunov scheme. Therefore, we will consider Godunov numerical solution as a reference for solutions obtained by using the weighted scheme for which flux limiters are calculated by linear programming.

Below, we will mark the numerical results by the scheme name with the one of  endings LP, AP, LE, AE and LET that indicates how the flux limiters were calculated. The endings LP and AP mean that the flux limiters were calculated by using exact and approximate solutions of the linear programming problem \eqref{eq:244}-\eqref{eq:247} or \eqref{eq:317}-\eqref{eq:318} without the entropy condition \eqref{eq:52}. LE and AE denote the numerical results which in addition to previous are obtained with the entropy condition \eqref{eq:52}.  AP and AE mean that the flux limiters were calculated by approximate relations. The letter T indicates that the Tadmor’s numerical entropy flux \cite{b26,b30} was used to discretize the entropy inequality.
 
We used the square entropy function $U = 0.5{\rho ^2}$ in numerical experiments.

%=====================================================%
%figure 1
\begin{figure*}[!b]
\centering
\begin{tabular}[t]{ccc} 
  \includegraphics[height=3.8cm]{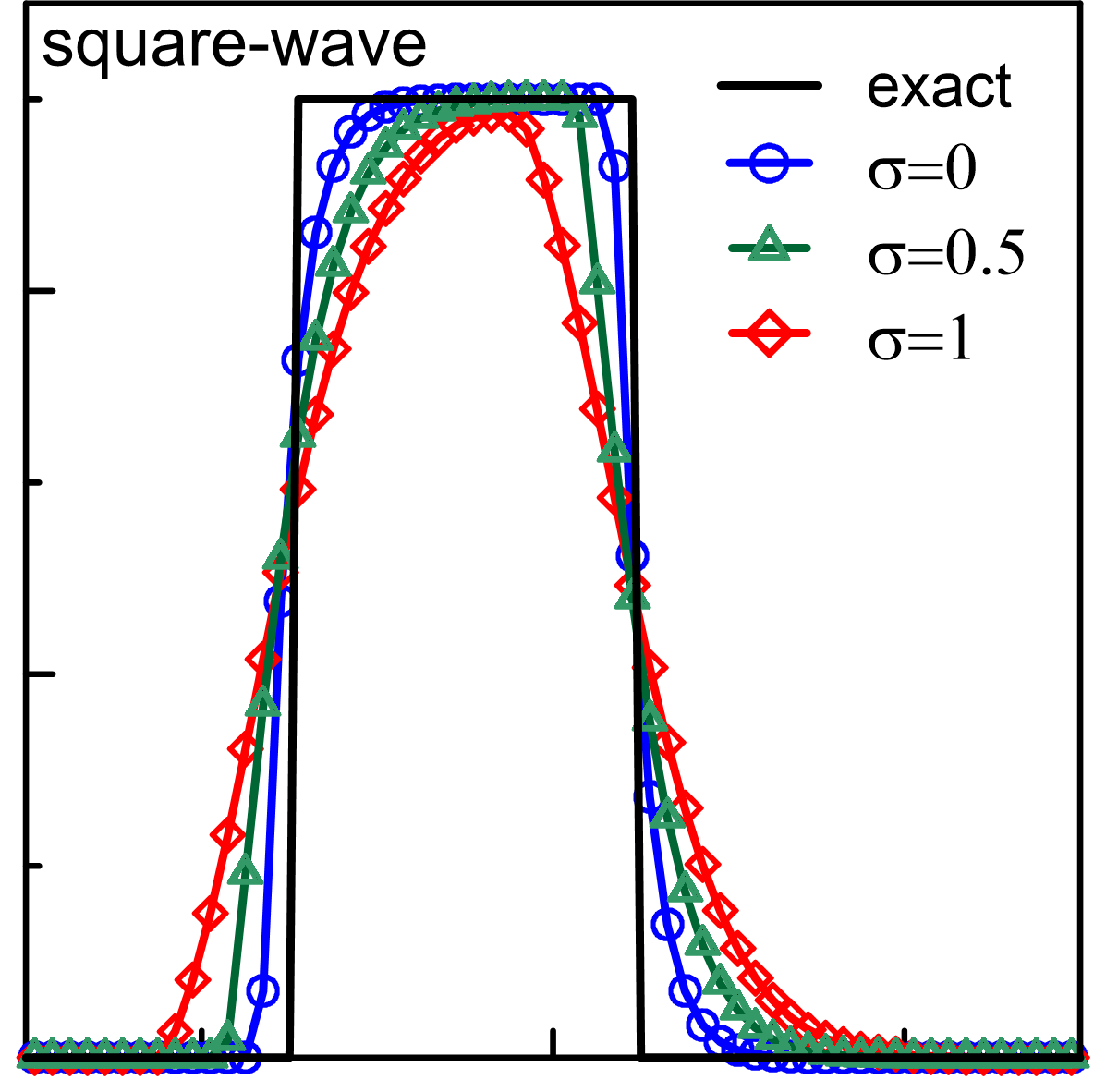}  &
  \includegraphics[height=3.8cm]{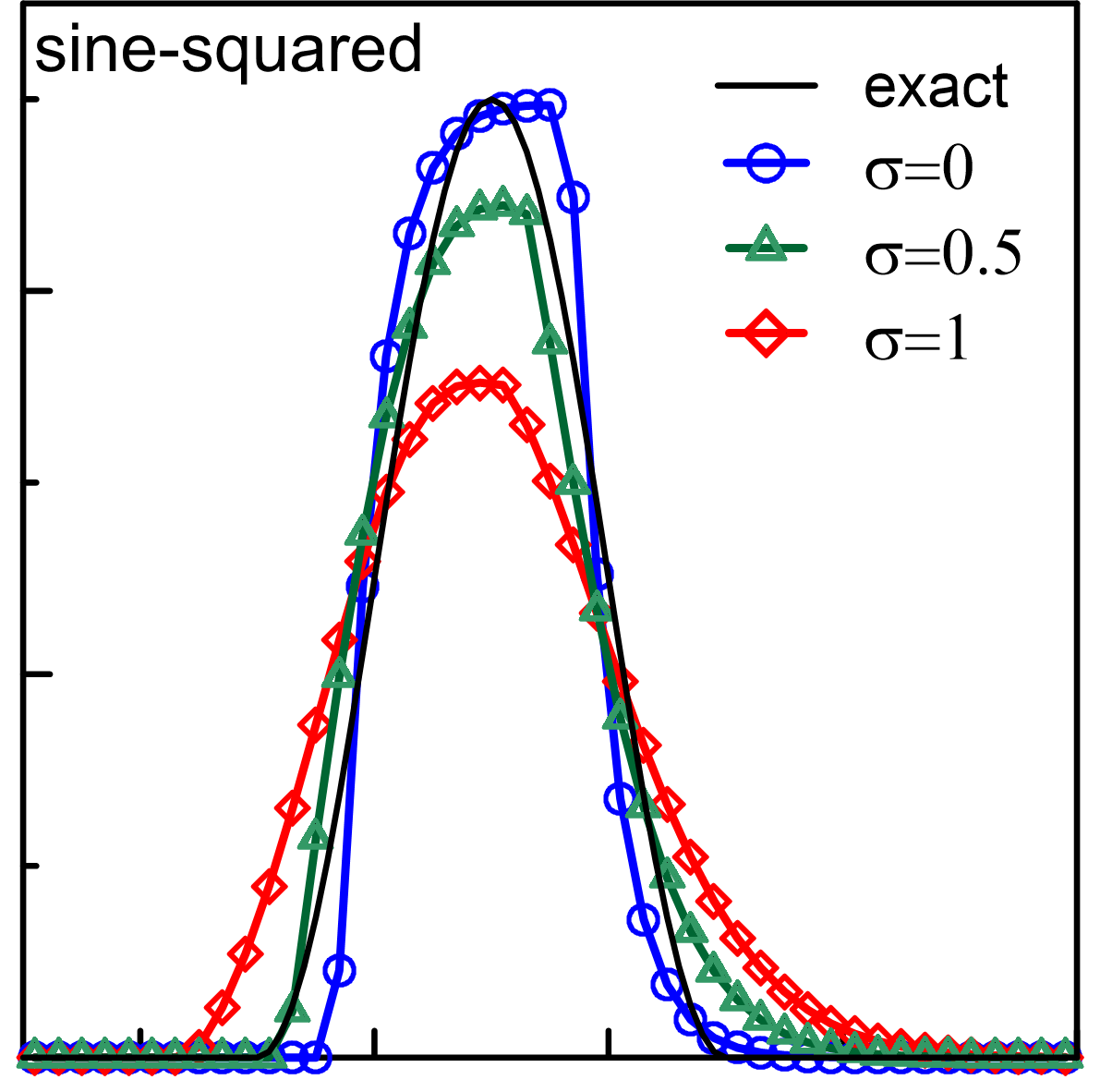}  &
  \includegraphics[height=3.8cm]{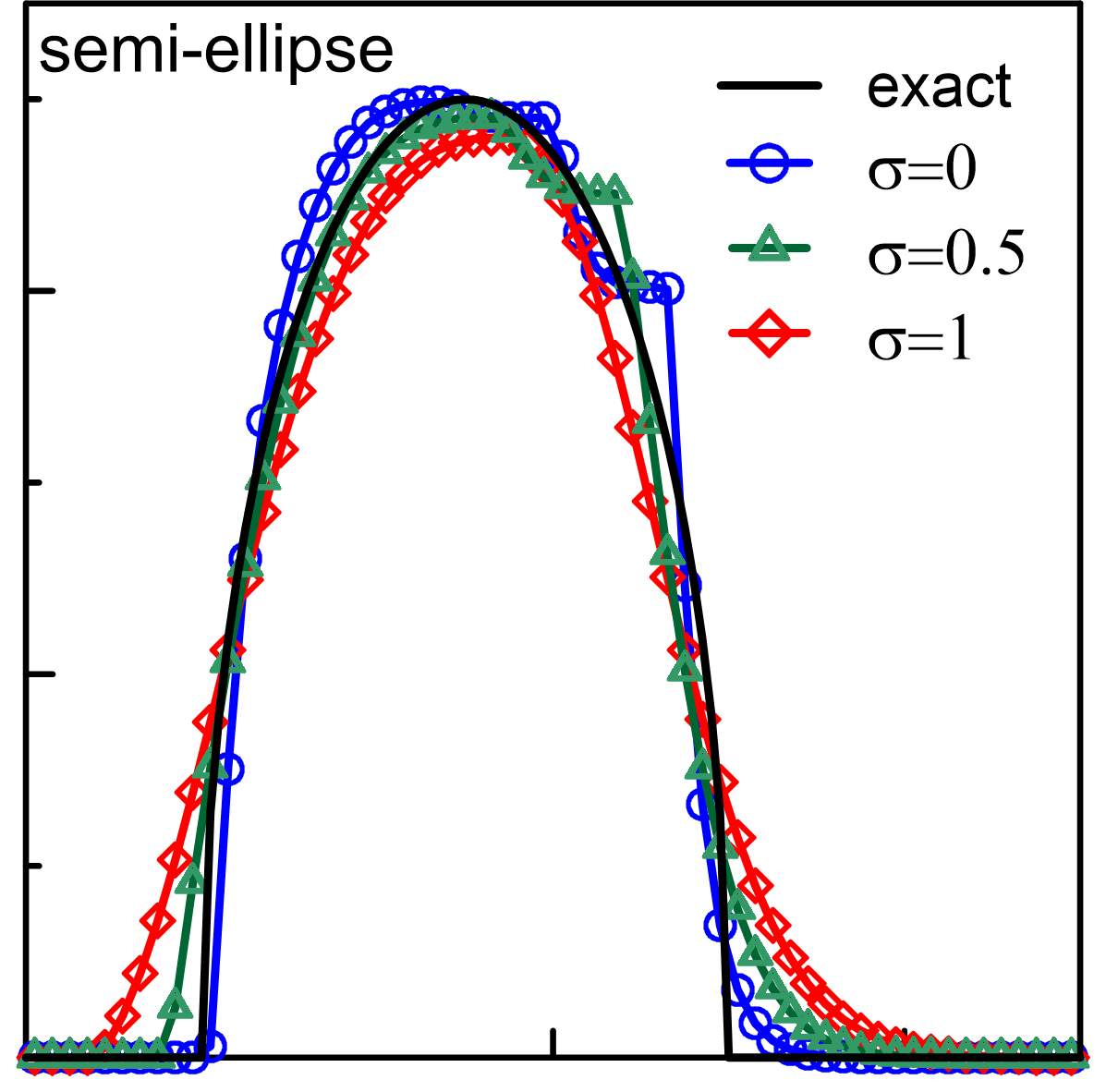} 
\end{tabular}   
\begin{tabular}[t]{cc} 
  \includegraphics[height=3.8cm]{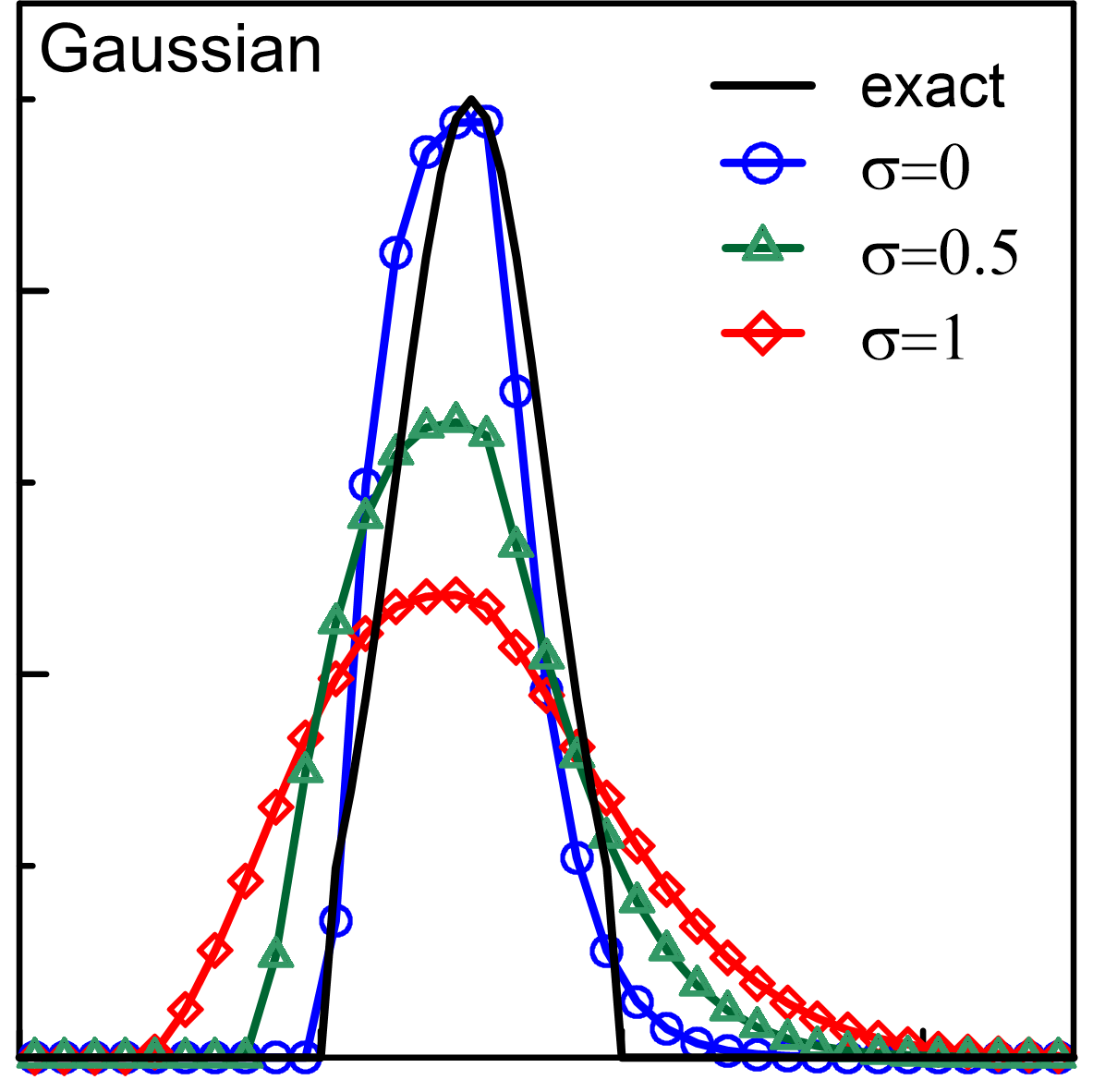}  &
  \includegraphics[height=3.8cm]{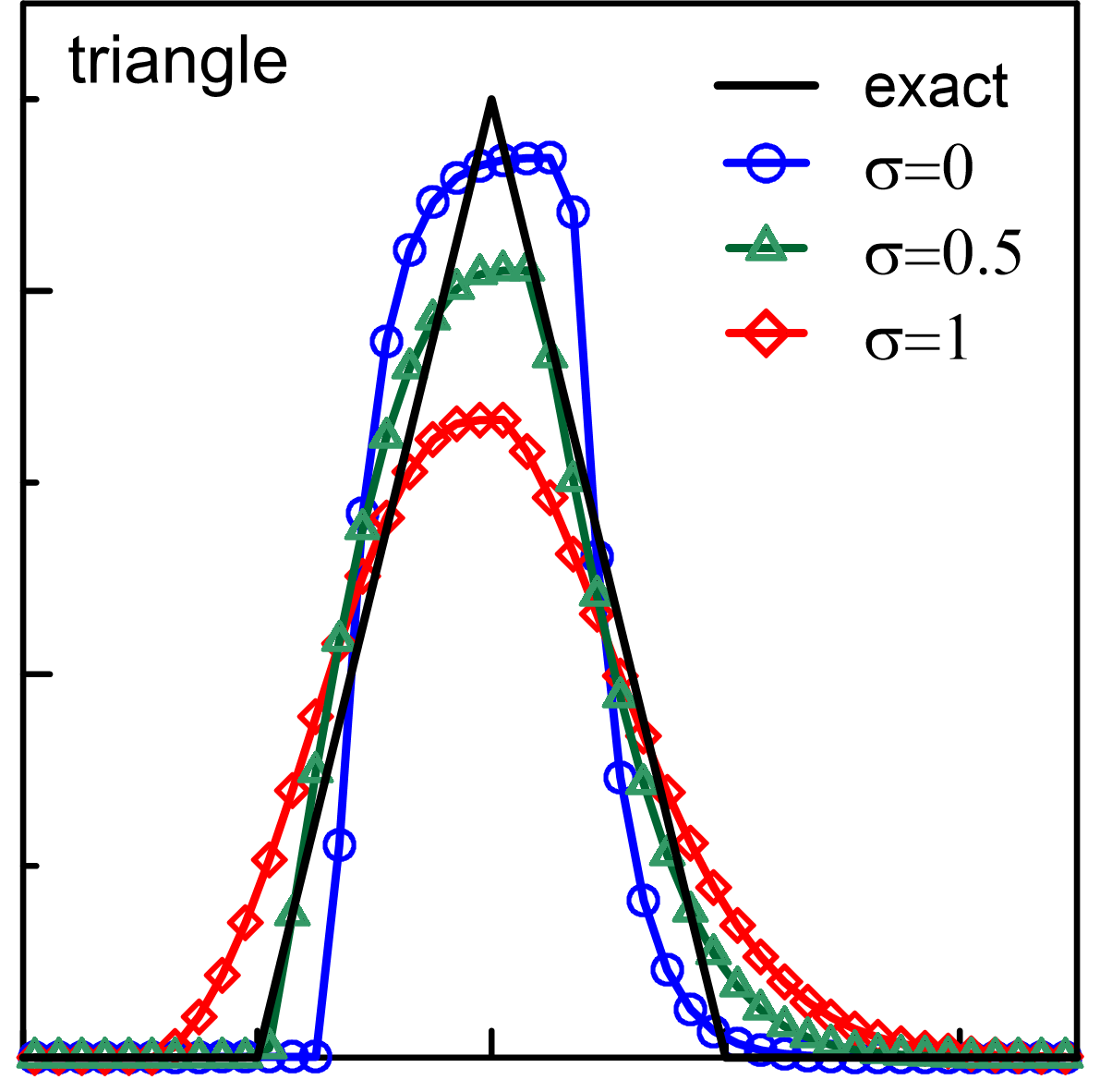}
\end{tabular}   
\caption{Numerical results of the advection test \eqref{eq:56} with the weighted scheme \eqref{eq:26}-\eqref{eq:29} for various weights $\sigma $  .  Flux limiters are calculated by using the linear programming problem  \eqref{eq:243}-\eqref{eq:246}}
\label{fig:1}       % Give a unique label
\end{figure*}
%=====================================================%

%=====================================================%
%figure 2
\begin{figure*}[!t]
\centering
\begin{tabular}[t]{ccc} 
  \includegraphics[height=3.8cm]{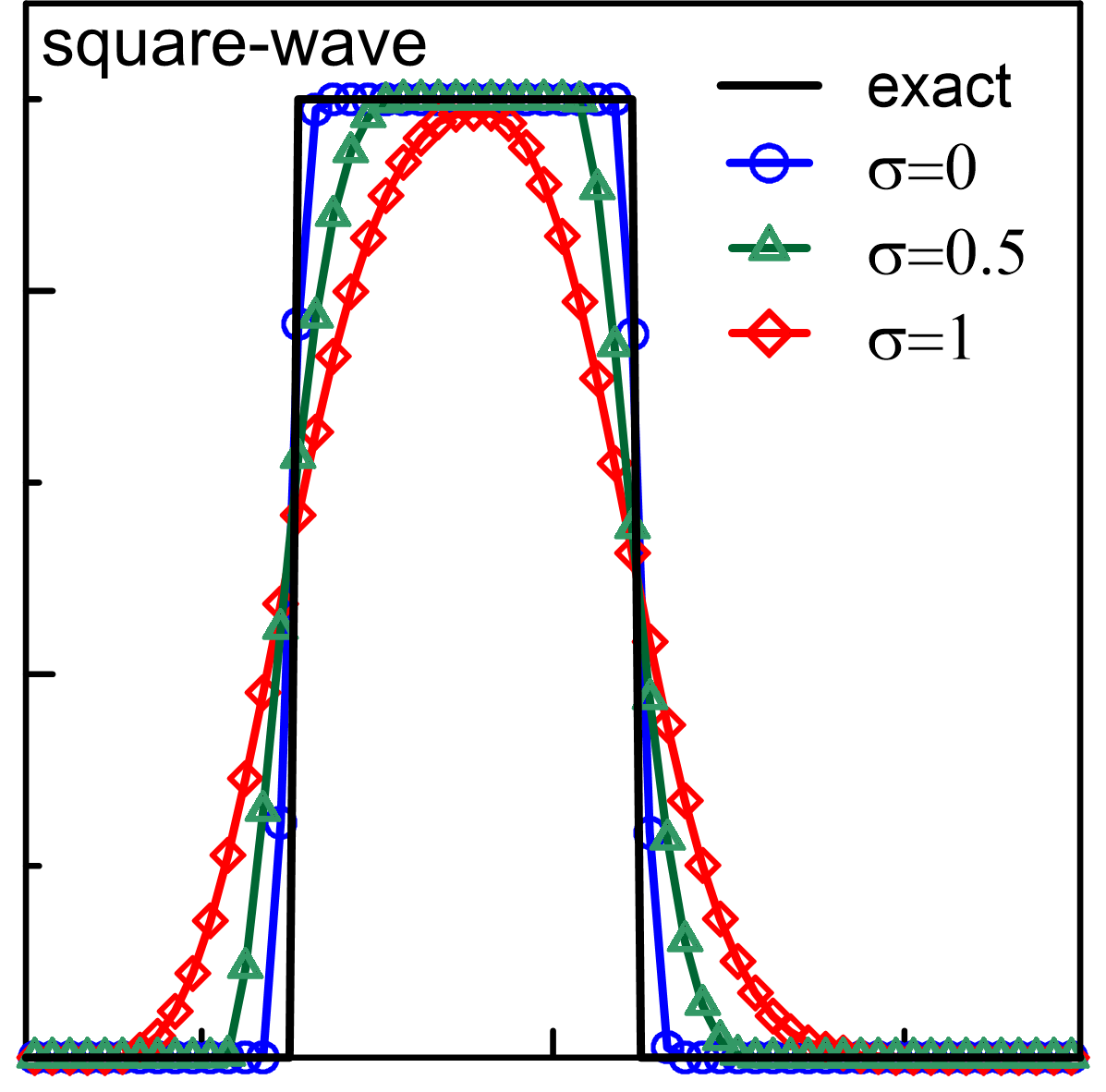} &
  \includegraphics[height=3.8cm]{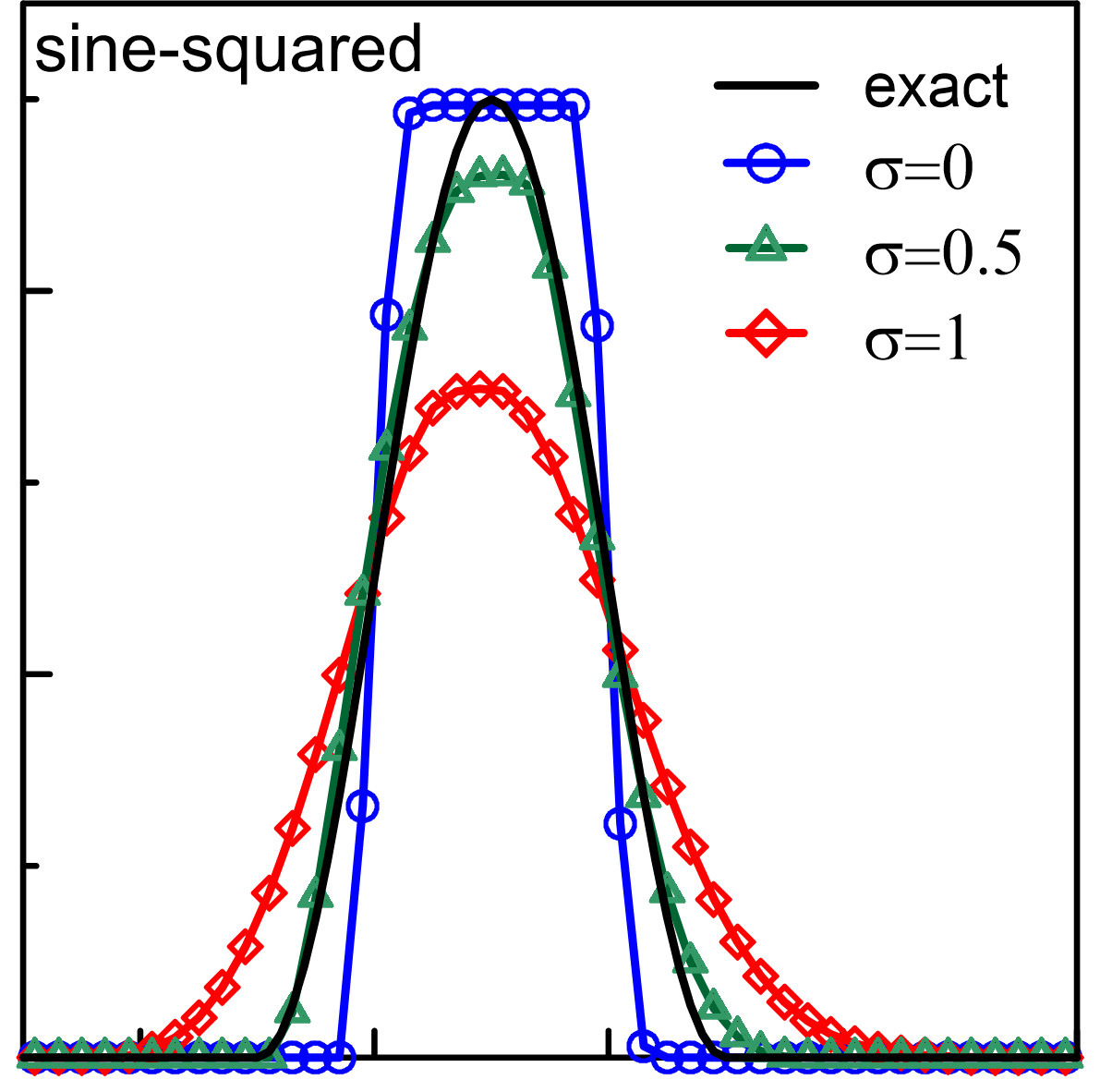} &
  \includegraphics[height=3.8cm]{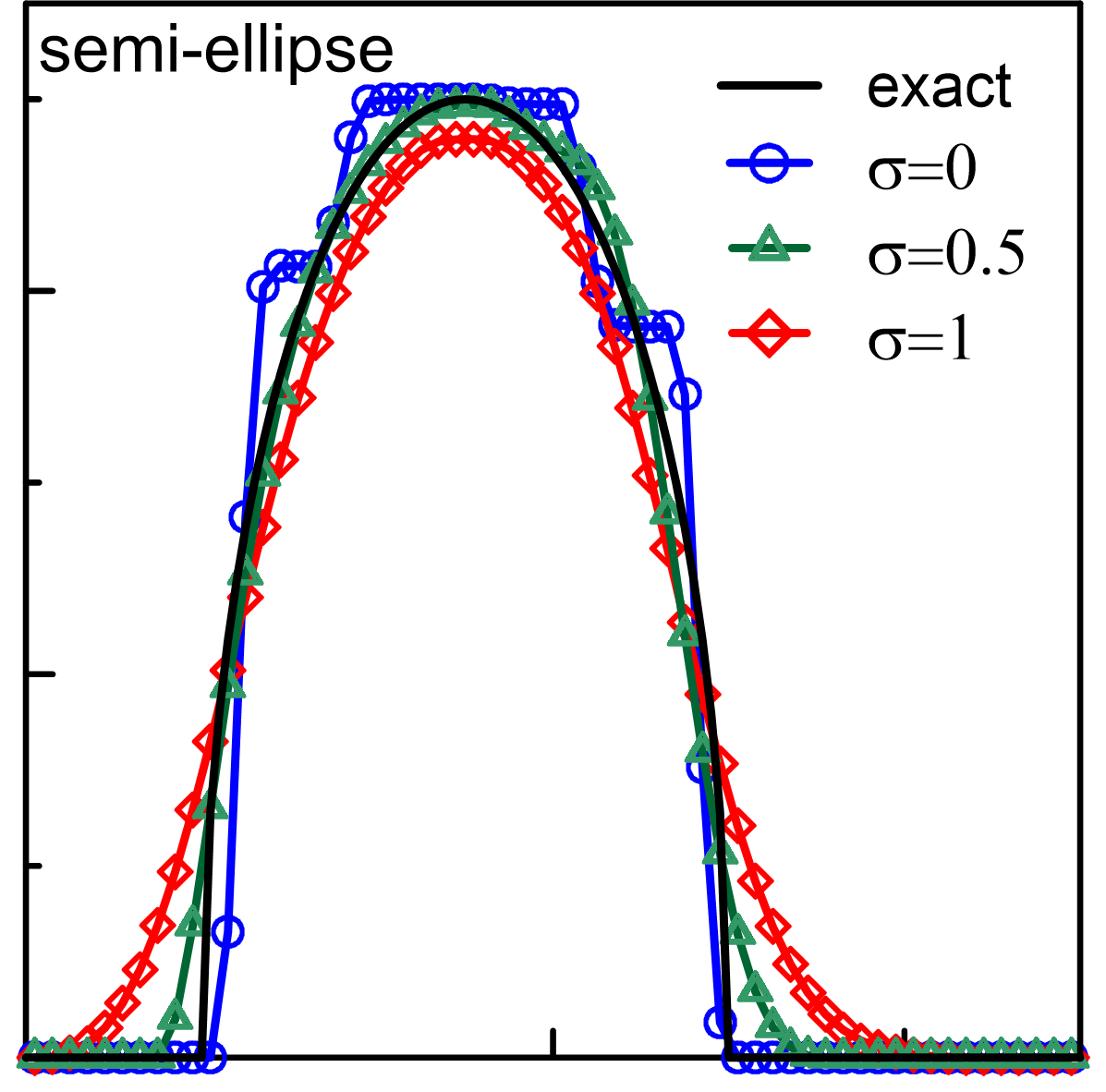}
\end{tabular}   
\begin{tabular}[t]{cc} 
  \includegraphics[height=3.8cm]{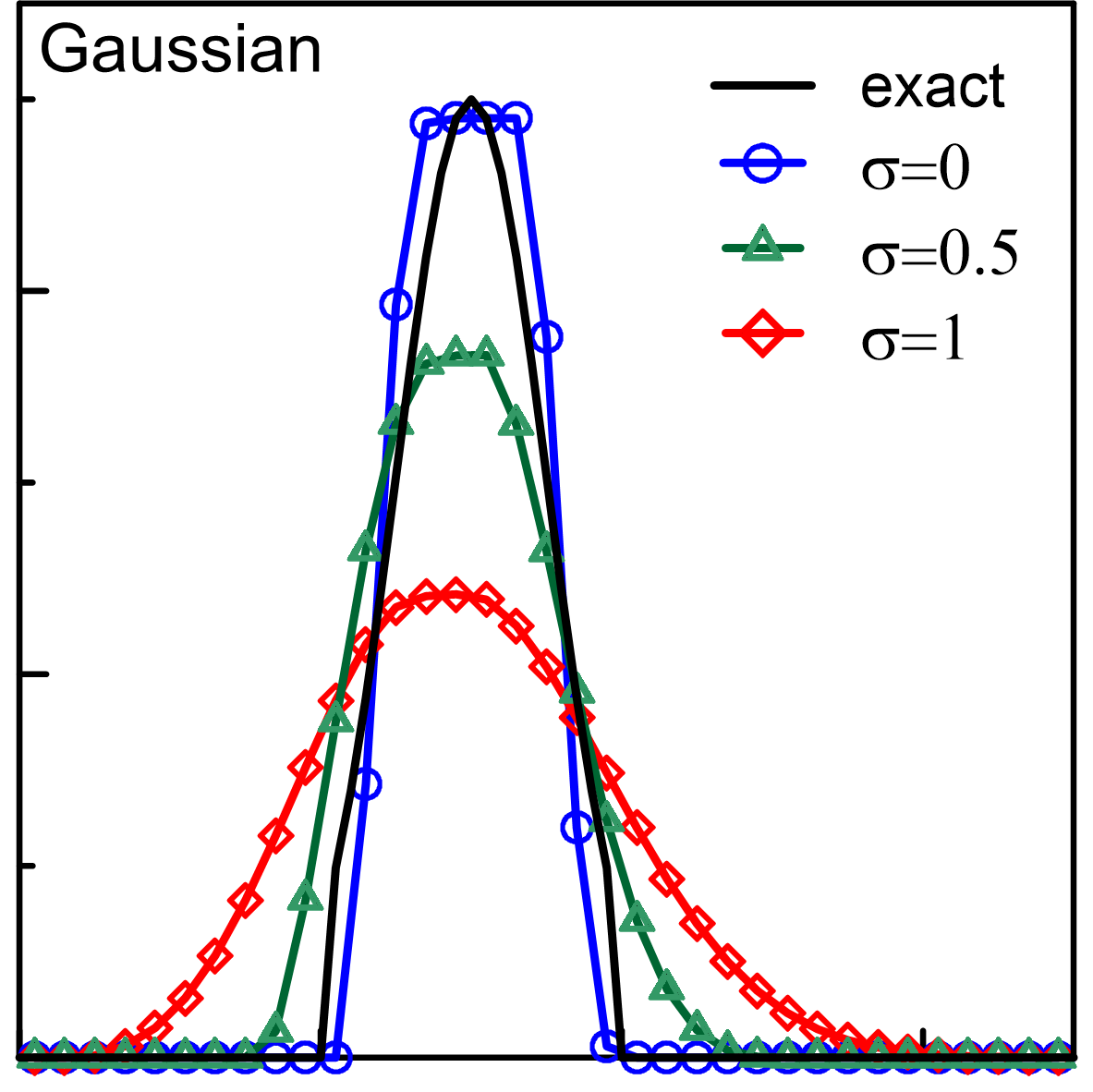} &
  \includegraphics[height=3.8cm]{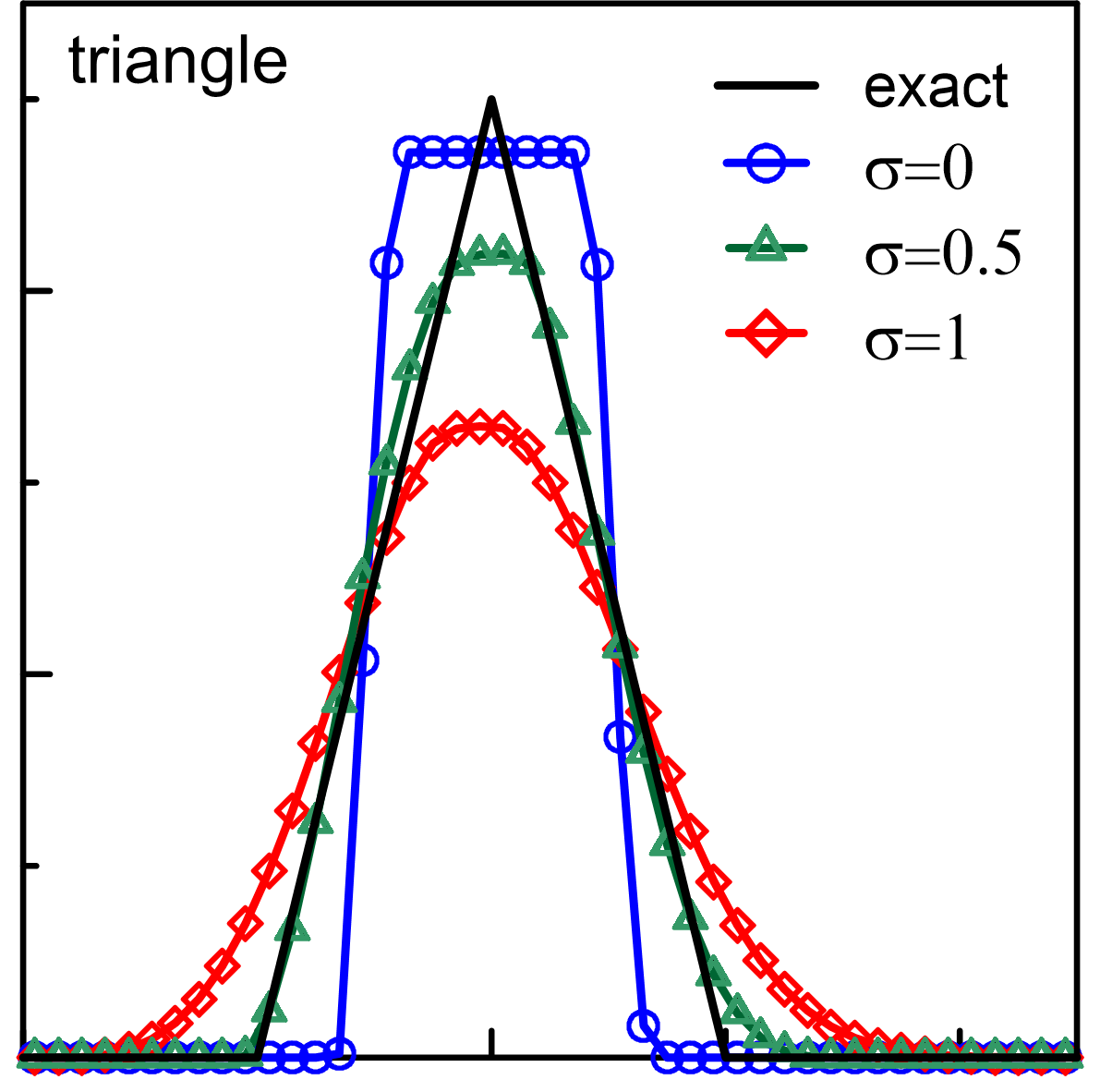} 
\end{tabular}   
\caption{Numerical results of the advection test \eqref{eq:56} with the QUICK scheme for various weights $\sigma$. Flux limiters are calculated by using linear programming}
\label{fig:2}       % Give a unique label
\end{figure*}
%=====================================================%

\subsection{One-Dimensional Advection}  \label{Sec51} 

We consider the one-dimensional advection test of Leonard et al. \cite{b31} on the uniform grid with $\Delta x = 0.01$ and constant velocity. The initial scalar profile consists of five different shapes: square wave, sine-squared, semi-ellipse, Gaussian, and triangle. The initial profile is specified as
\begin{equation}
\label{eq:56}
 y({x_i}) =  \left\{ {\begin{array}{*{20}{l}}
1&{{\rm if} \;\; 0.05 \le {x_i} \le 0.25}&{{\rm{(square}}\;{\rm{wave)}}} \\
{{{\sin }^2}\left[ {\dfrac{\pi }{{0.2}}\left( {{x_i} - 0.85} \right)} \right]}&{{\rm if} \;\; 0.85 \le {x_i} \le 1.05}&{{\rm{(sine - squared)}}} \\
{\sqrt {1 - {{\left[ {\dfrac{1}{{15\Delta x}}\left( {{x_i} - 1.75} \right)} \right]}^2}} }& {{\rm if} \;\; 1.6 \le {x_i} \le 1.9}&{{\rm{(semi - ellipse)}}} \\
{\exp \left[ { - \dfrac{1}{{2{\sigma ^2}}}{{\left( {{x_i} - 2.65} \right)}^2}} \right]}& {{\rm if} \;\; 2.6 \le {x_i} \le 2.7}&{{\rm{(Gaussian)}}} \\
{10\left( {{x_i} - 3.3} \right)}& {{\rm if} \;\; 3.3 \le {x_i} \le 3.4}&{{\rm{(triangle)}}} \\
{1.0 - 10\left( {{x_i} - 3.4} \right)}& {{\rm if} \;\; 3.4 \le {x_i} \le 3.5}&{} \\
0&{{\rm{otherwise}}}&{}
\end{array}} \right.   				
\end{equation} 
The standard deviation for the Gaussian profile is specified as $\sigma  = 2.5$. Periodic boundary conditions are used.

Numerical results for the difference scheme \eqref{eq:26}-\eqref{eq:29} after 400 time steps at Courant number of 0.2 are shown in Fig.~\ref{fig:1}. The flux limiters are calculated by using the linear programming problem \eqref{eq:243}-\eqref{eq:246}. Numerical results for which the flux limiters are calculated by using exact and approximate solutions of the linear programming problem \eqref{eq:243}-\eqref{eq:246} are slightly different. Their $L^1$-norm of errors and maximum values are presented in Table~\ref{tab1}.

%\begin{tabular}{|p{2.25cm}|p{2cm}|l|p{4cm}|p{3cm}|p{2cm}|}

\begin{table*}[!t]
\caption{\label{tab1}Advection test \eqref{eq:56} with the weighted scheme \eqref{eq:26}-\eqref{eq:29}. $L^1 $–norm of errors and the maximum values of the numerical results for which flux limiters have been calculated by using exact (LP) and approximate (AP) solutions of linear programming problem}
\centering
%\begin{tabular}{@{}p{2.45cm} p{1.2cm} *{2}{ p{2.55cm} p{1.5cm}@{}}}
\begin{tabular}{@{}p{2.45cm} p{1.2cm} p{2.55cm} p{1.5cm} p{2.55cm} p{1.5cm} @{}}
\hline
 &  &  \multicolumn{2}{c}{LP} & \multicolumn{2}{c}{AP}  \\
 \cline{3-4} \cline{5-6}
  & $\sigma$  &  $L^1$ error & $y_{max}$ & $L^1$ error & $y_{max}$ \\[3pt] \hline
\multirow{3}{*}{Square wave}  
	& 0.0 &	2.1811$\times 10^{-2}$ & 1.0000 & 2.1811$\times 10^{-2}$ & 1.0000 \\
	& 0.5 &	4.3933$\times 10^{-2}$ & 0.9997 & 4.3934$\times 10^{-2}$ & 0.9997 \\
	& 1.0 &	6.9477$\times 10^{-2}$ & 0.9843 & 6.9490$\times 10^{-2}$ & 0.9843 \\
\hline	
\multirow{3}{*}{Sine-squared}
	& 0.0 &	1.6883$\times 10^{-2}$ & 0.9938 & 1.6776$\times 10^{-2}$ & 0.9909 \\
	& 0.5 &	1.6423$\times 10^{-2}$ & 0.8895 & 1.6391$\times 10^{-2}$ & 0.8850 \\
	& 1.0 &	3.9029$\times 10^{-2}$ & 0.7043 & 3.9043$\times 10^{-2}$ & 0.7046 \\
\hline	
\multirow{3}{*}{Semi-ellipse}
	& 0.0 &	1.7926$\times 10^{-2}$ & 0.9973 & 1.7886$\times 10^{-2}$ & 0.9965 \\
	& 0.5 &	1.7913$\times 10^{-2}$ & 0.9810 & 1.7908$\times 10^{-2}$ & 0.9810 \\
	& 1.0 &	3.6078$\times 10^{-2}$ & 0.9601 & 3.6079$\times 10^{-2}$ & 0.9600 \\
\hline	
\multirow{3}{*}{Gaussian}
	& 0.0 &	1.3639$\times 10^{-2}$ & 0.9764 & 1.2237$\times 10^{-2}$ & 0.9512 \\
	& 0.5 &	2.7592$\times 10^{-2}$ & 0.6629 & 2.7154$\times 10^{-2}$ & 0.6601 \\ 
	& 1.0 &	4.3681$\times 10^{-2}$ & 0.4828 & 4.3710$\times 10^{-2}$ & 0.4832 \\
\hline	
\multirow{3}{*}{Triangle}
	& 0.0 &	2.5205$\times 10^{-2}$ & 0.9389 & 2.4952$\times 10^{-2}$ & 0.9365 \\
	& 0.5 &	1.3843$\times 10^{-2}$ & 0.8216 & 1.3672$\times 10^{-2}$ & 0.8197 \\
	& 1.0 &	3.1245$\times 10^{-2}$ & 0.6655 & 3.1260$\times 10^{-2}$ & 0.6653 \\
\hline
\end{tabular}
\end{table*}

Along with the second-order flux \eqref{eq:28}, we apply the QUICK \cite{b32} numerical flux
\begin{eqnarray}
\label{eq:57} 
%\begin{split} 
&  h_{i + {1 /2}}^Q  =  u_{i + {1 /2}}^ + \left( {\dfrac{3}{8}{y_{i + 1}} + \dfrac{3}{4}{y_i} - \dfrac{1}{8}{y_{i - 1}}} \right) + u_{i + {1 /2}}^ - \left( {\dfrac{3}{8}{y_i} + \dfrac{3}{4}{y_{i + 1}} - \dfrac{1}{8}{y_{i + 2}}} \right)  \\ 				
&  =  u_{i + {1 /2}}^ + \;{y_i} + u_{i + {1 /2}}^ - \;{y_{i + 1}} + \left\{ {\dfrac{{3\left| {{u_{i + {1 /2}}}} \right|}}{8}\left( {{y_{i + 1}} - {y_i}} \right) + \dfrac{{u_{i + {1 /2}}^ + }}{8}\left( {{y_i} - {y_{i - 1}}} \right) + \dfrac{{u_{i + {1 /2}}^ - }}{8}\left( {{y_{i + 1}} - {y_{i + 2}}} \right)} \right\}  \nonumber
%\end{split} 
\end{eqnarray}
which is third-order flux on a uniform grid. We consider the term in braces on the right-hand side of \eqref{eq:57} as antidiffusive. 

\begin{table*}[!t]
\caption{\label{tab2}Advection test \eqref{eq:56} for the weighted scheme \eqref{eq:26}-\eqref{eq:27} with the QUICK numerical flux \eqref{eq:57}. $L^1 $–norm of errors and the maximum values of the numerical results for which flux limiters have been calculated by using exact (LP) and approximate (AP) solutions of linear programming problem}
\centering
%\begin{tabular}{@{}p{2.45cm} p{1.2cm} *{2}{ p{2.55cm} p{1.5cm}@{}}}
\begin{tabular}{@{}p{2.45cm} p{1.2cm} p{2.55cm} p{1.5cm} p{2.55cm} p{1.5cm} @{}}
\hline
 &  &  \multicolumn{2}{c}{LP} & \multicolumn{2}{c}{AP}  \\
 \cline{3-4} \cline{5-6}
  & $\sigma$  &  $L^1$ error & $y_{max}$ & $L^1$ error & $y_{max}$ \\[3pt] \hline
\multirow{3}{*}{Square wave}  
	& 0.0 &	9.8128$\times 10^{-2}$ & 1.0000 & 9.8128$\times 10^{-2}$ & 1.0000 \\
	& 0.5 &	3.2015$\times 10^{-2}$ & 1.0000 & 3.2046$\times 10^{-2}$ & 1.0000 \\
	& 1.0 &	6.6670$\times 10^{-2}$ & 0.9855 & 6.6719$\times 10^{-2}$ & 0.9853 \\
\hline	
\multirow{3}{*}{Sine-squared}
	& 0.0 &	2.5703$\times 10^{-2}$ & 0.9938 & 2.5705$\times 10^{-2}$ & 0.9938 \\
	& 0.5 &	7.3080$\times 10^{-3}$ & 0.9213 & 7.3354$\times 10^{-3}$ & 0.9205 \\
	& 1.0 &	3.7450$\times 10^{-2}$ & 0.6980 & 3.7472$\times 10^{-2}$ & 0.6976 \\
\hline	
\multirow{3}{*}{Semi-ellipse}
	& 0.0 &	2.0788$\times 10^{-2}$ & 0.9995 & 2.0810$\times 10^{-2}$ & 0.9994 \\
	& 0.5 &	1.1353$\times 10^{-2}$ & 0.9937 & 1.1364$\times 10^{-2}$ & 0.9937 \\
	& 1.0 &	3.5186$\times 10^{-2}$ & 0.9585 & 3.5194$\times 10^{-2}$ & 0.9584 \\
\hline	
\multirow{3}{*}{Gaussian}
	& 0.0 &	1.2183$\times 10^{-2}$ & 0.9802 & 1.2201$\times 10^{-2}$ & 0.9801 \\
	& 0.5 &	1.7177$\times 10^{-2}$ & 0.7331 & 1.7147$\times 10^{-2}$ & 0.7319 \\ 
	& 1.0 &	4.1866$\times 10^{-2}$ & 0.4835 & 4.1935$\times 10^{-2}$ & 0.4835 \\
\hline	
\multirow{3}{*}{Triangle}
	& 0.0 &	3.4487$\times 10^{-2}$ & 0.9447 & 3.4505$\times 10^{-2}$ & 0.9445 \\
	& 0.5 &	7.6615$\times 10^{-3}$ & 0.8389 & 7.6448$\times 10^{-3}$ & 0.8384 \\
	& 1.0 &	2.9795$\times 10^{-2}$ & 0.6588 & 2.9813$\times 10^{-2}$ & 0.6585 \\
\hline
\end{tabular}
\end{table*}

Numerical results of the advection test with the QUICK scheme are given in Fig.~\ref{fig:2}. Their $L^1$-norm of errors and maximum values are presented in Table~\ref{tab2}.

%=====================================================%
%figure 3
\begin{figure}[!b]
  \centering 
  \includegraphics[scale=0.8]{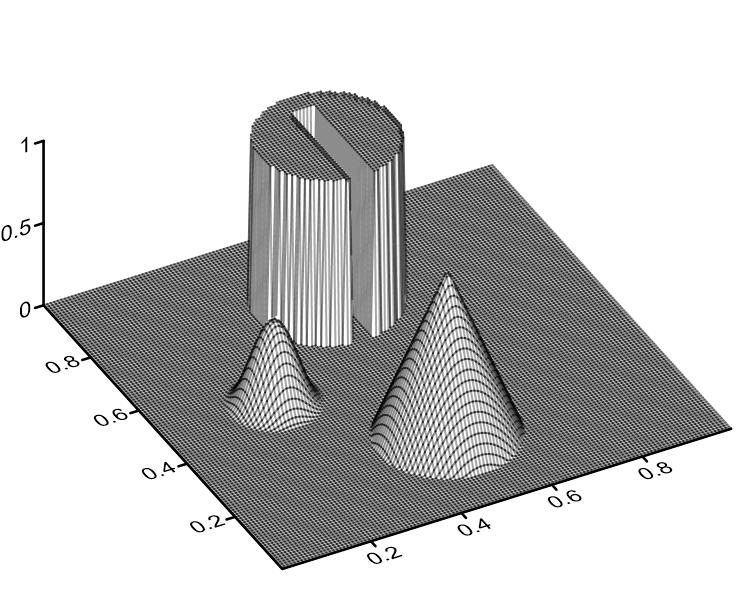}
\caption{Initial data and exact solution at the final time for solid body rotation test}
\label{fig:3}       % Give a unique label
\end{figure}
%=====================================================%

\subsection{Solid Body Rotations}  \label{Sec52}
In this section, we consider the rotation of solid bodies \cite{b33,b7,b2} under an incompressible flow that is described by the linear equation
\begin{equation}
\label{eq:58} 
 \frac{{\partial \rho }}{{\partial t}} + \nabla  \cdot \left( {u\rho } \right) = 0 \qquad 	\text{in} \quad \Omega  = \left( {0,1} \right) \times \left( {0,1} \right)
\end{equation} 							
with zero boundary conditions. The initial condition includes a slotted cylinder, a cone and a smooth hump (Fig.~\ref{fig:3}).  The slotted cylinder of radius 0.15 and height 1 is centered at the point (0.5,0.75) and
\[
  \rho (x,y,0) =\begin{cases}
  1 \qquad  {\rm if} \;\; \left| {x - 0.5} \right| \ge 0.025 \;\; {\rm or} \;\; y \ge 0.85\\
  0 \qquad  \rm{otherwise}
\end{cases}   
\]

%=====================================================%
%figure 4
\begin{figure*}[!t]
\centering
  \includegraphics[width=0.3\textwidth]{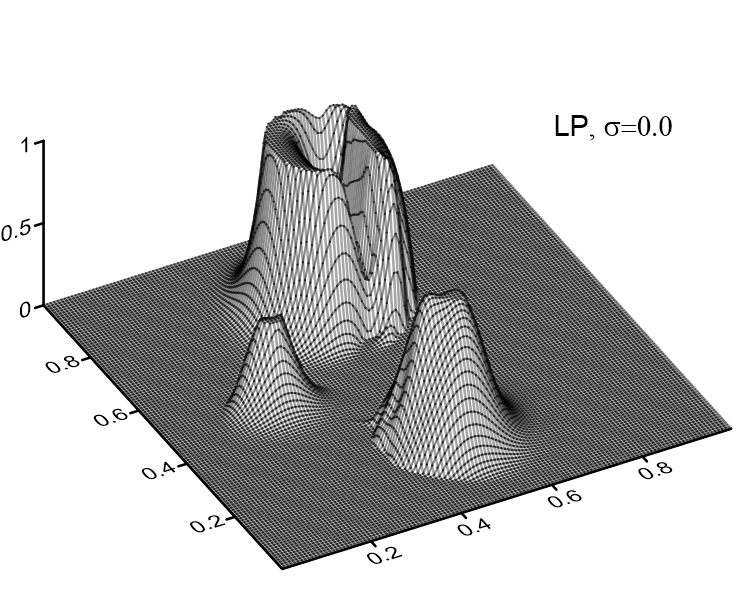}
  \includegraphics[width=0.3\textwidth]{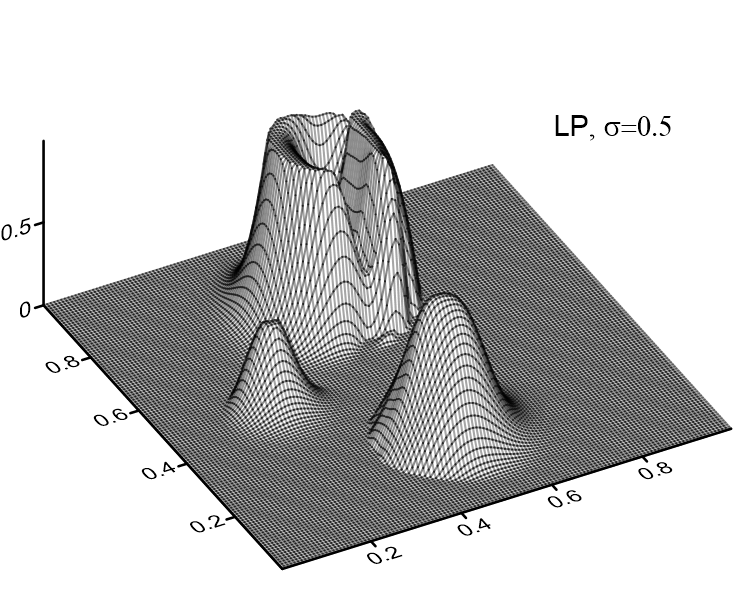}
  \includegraphics[width=0.3\textwidth]{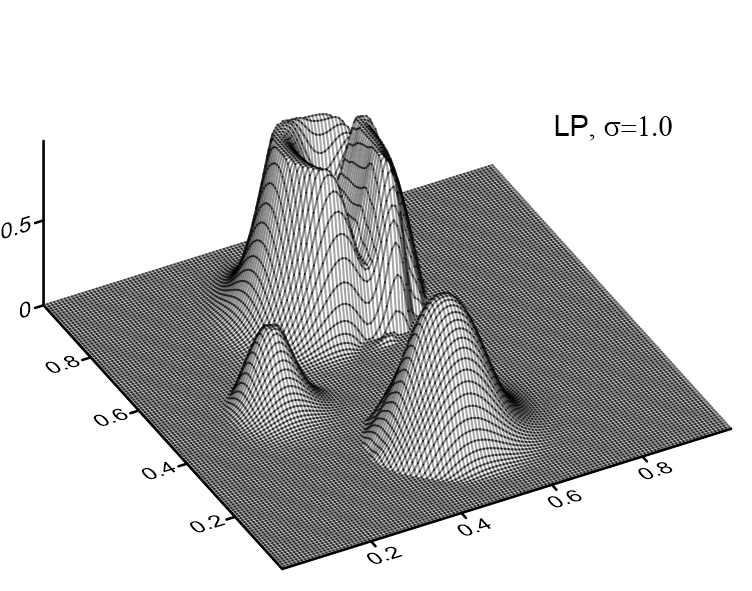}
  \includegraphics[width=0.3\textwidth]{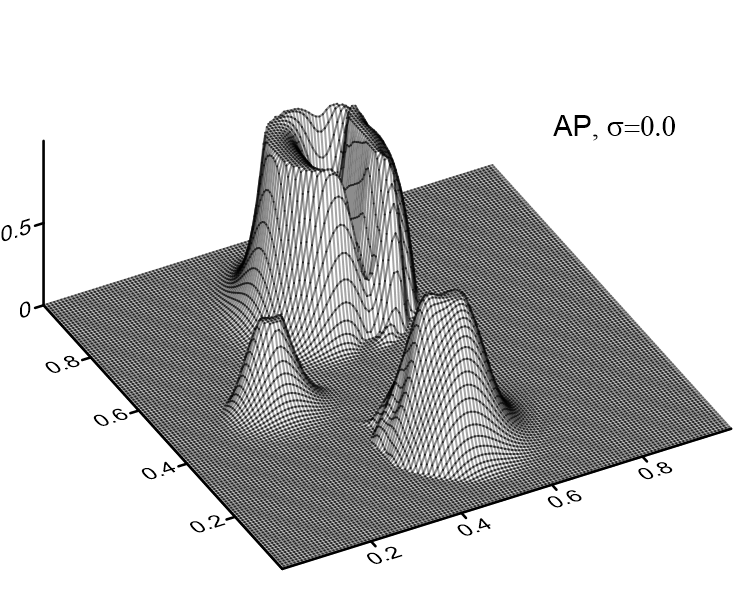}
  \includegraphics[width=0.3\textwidth]{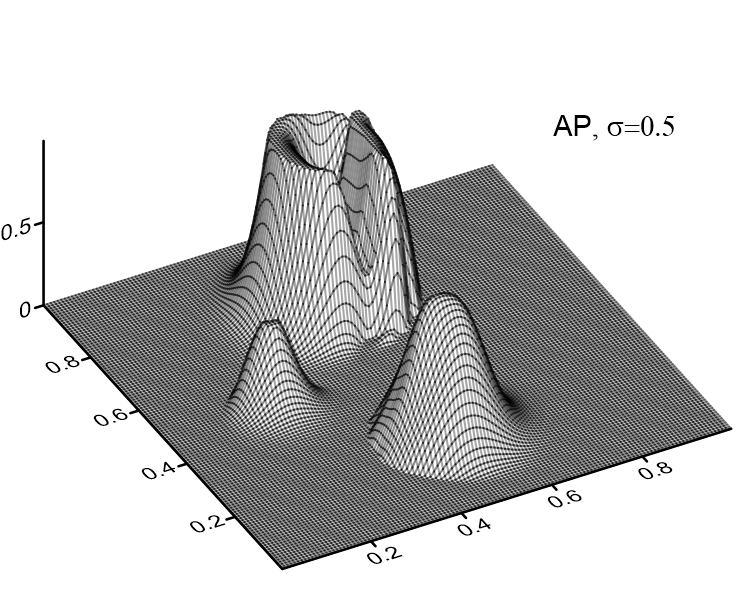}
  \includegraphics[width=0.3\textwidth]{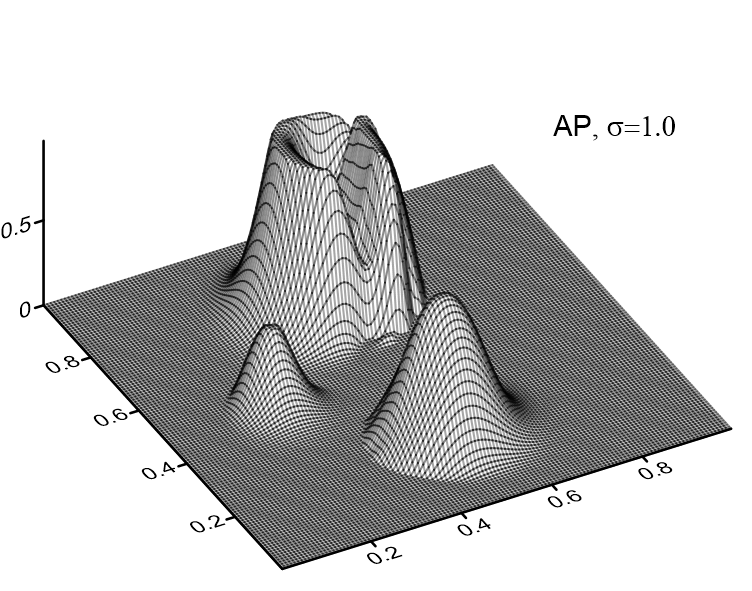}
%  \includegraphics[height=4.5cm]{figs/fig_4a_bl.png}
%  \includegraphics[height=4.5cm]{figs/fig_4b_bl.png}
%  \includegraphics[height=4.5cm]{figs/fig_4c_bl.png}
%  \includegraphics[height=4.5cm]{figs/fig_4d_bl.png}
%  \includegraphics[height=4.5cm]{figs/fig_4e_bl.png}
%  \includegraphics[height=4.5cm]{figs/fig_4f_bl.png}
\caption{Numerical results of the solid body rotation test after one revolution (5000 time steps) with the weighted scheme \eqref{eq:26}-\eqref{eq:28} for various weights $\sigma$. Flux limiters are calculated by using the exact (upper) and approximate (lower) solutions of the linear programming problems \eqref{eq:244}-\eqref{eq:247}}
\label{fig:4}       % Give a unique label
\end{figure*}
%=====================================================%

%=====================================================%
%figure 5
\begin{figure*}[!tb]
\centering
  \includegraphics[width=0.3\textwidth]{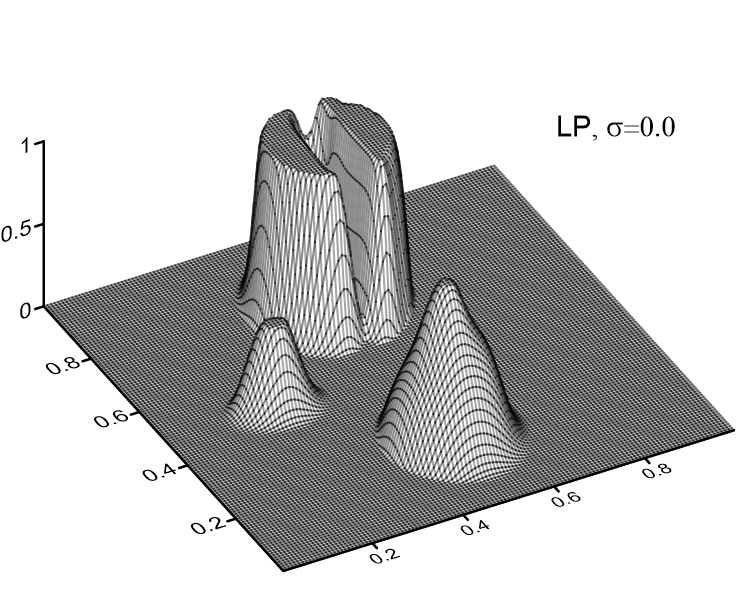}
  \includegraphics[width=0.3\textwidth]{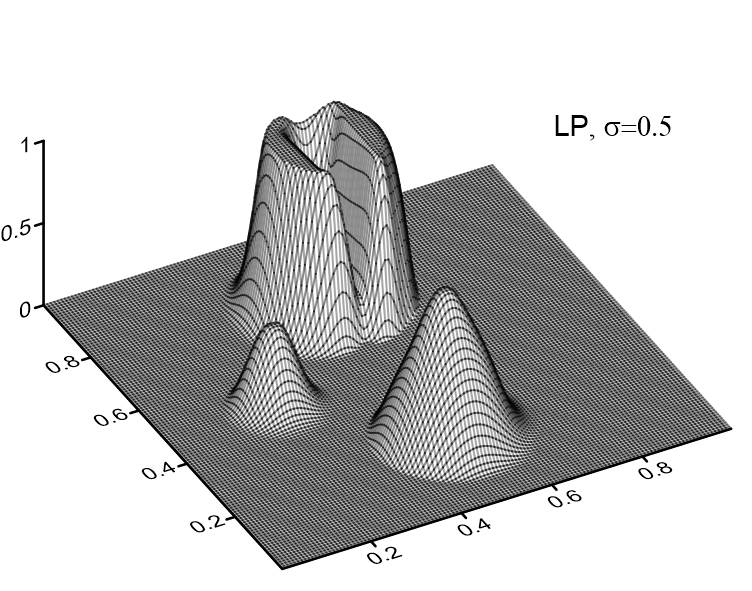}
  \includegraphics[width=0.3\textwidth]{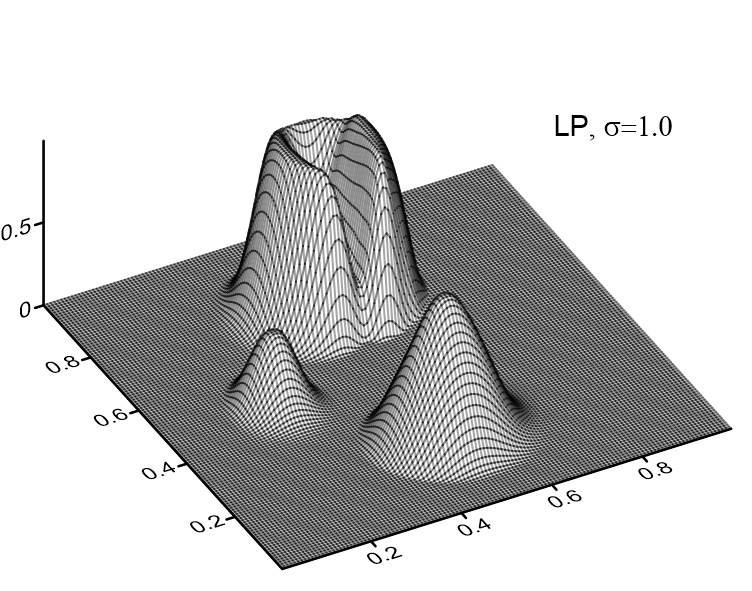}
  \includegraphics[width=0.3\textwidth]{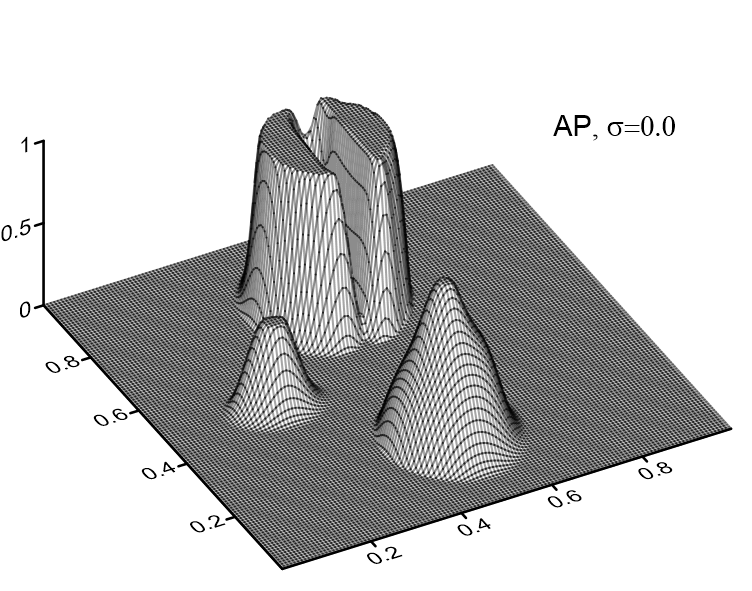}
  \includegraphics[width=0.3\textwidth]{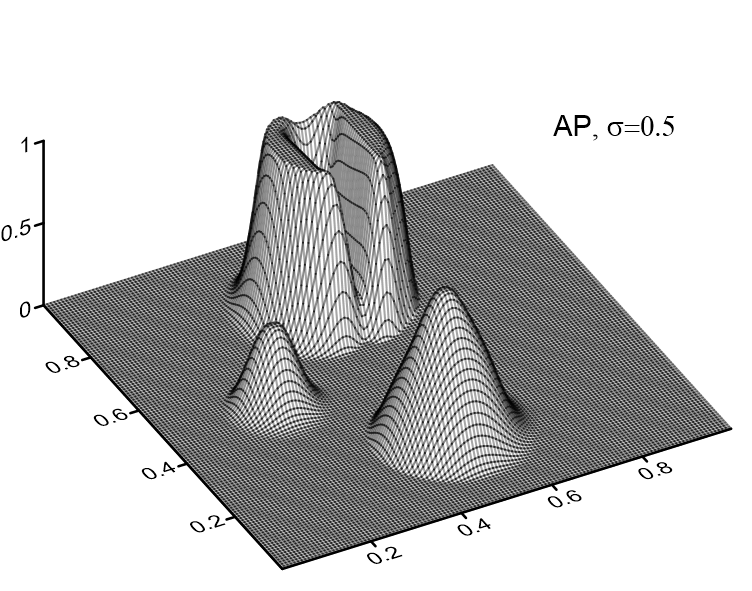}
  \includegraphics[width=0.3\textwidth]{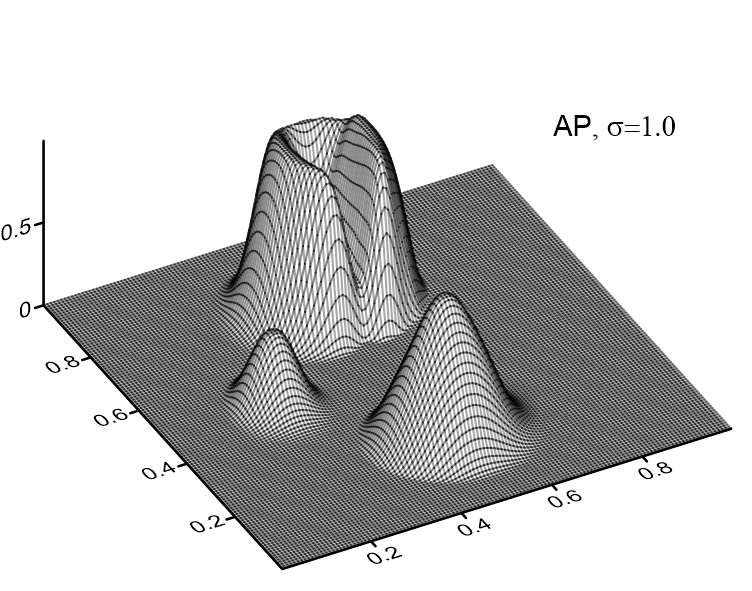}
\caption{Numerical results of the solid body rotation test after one revolution (5000 time steps) with the QUICK scheme for various weights $\sigma$. The flux limiters are calculated by using the exact (upper) and approximate (lower) solutions of the linear programming problems \eqref{eq:244}-\eqref{eq:247}}
\label{fig:5}       % Give a unique label
\end{figure*}
%=====================================================%

The cone of also radius $r_0=0.15$ and height 1 is centered at point $({x_0},{y_0}) = (0.25,0.5)$ and
\[ \rho (x,y,0) = 1 - r(x,y) \] 
where 
\[r(x,y) = \frac{{\min (\sqrt {{{(x - {x_0})}^2} + {{(y - {y_0})}^2}} ,{r_0})}}{{{r_0}}} \]

The hump is given by
 \[ \rho (x,y,0) = \frac{1}{4}(1 + \cos (\pi r(x,y)) \]
where $({x_0},{y_0}) = (0.5,0.25)$ and $r_0=0.1$.

The flow velocity is calculated by $u(x,y) = \left( { - 2\pi (y - 0.5),2\pi (x - 0.5)} \right)$ and in result of which the counterclockwise rotation takes place about domain point (0.5, 0.5). The computational grid consists of uniform $128 \times 128$ cells. The exact solution of \eqref{eq:58}(5.8) reproduces by the initial state after each full revolution.

The numerical results produced by the weighted scheme \eqref{eq:26}-\eqref{eq:27} with the second-order flux \eqref{eq:28} and the QUICK flux \eqref{eq:57} after one full revolution (5000 time steps) with different weights $ \sigma $ are presented in Fig.~\ref{fig:4} and Fig.~\ref{fig:5}, respectively. We calculate flux limiters by using exact (LP) and approximate (AP) solutions of the linear programming problem \eqref{eq:243}-\eqref{eq:246}. The $L^1$-norm of errors and maximum values of the numerical results are given in Table~\ref{tab2} and Table~\ref{tab3}. 
The numerical flux \eqref{eq:28} is a second-order accuracy flux and the QUICK numerical flux \eqref{eq:57} is a third-order accuracy flux. 
Thus, we can see improvement of the numerical results with increasing accuracy order of the high order flux. We also note a good agreement between the numerical results obtained with the exact and approximate solutions of the linear programming problem. 

\begin{table*}[!t]
\caption{\label{tab3} $L^1$-norm of errors and maximum values of numerical solutions for the solid body rotation test with the weighted scheme (2.6)–(2.8)\eqref{eq:26}-\eqref{eq:28}. Flux limiters are calculated by using the exact (LP) and approximate (AP) solutions of the linear programming problem \eqref{eq:244}-\eqref{eq:247}}
\centering
%\begin{tabular}{@{}p{2.45cm} p{1.2cm} *{2}{ p{2.55cm} p{1.5cm}@{}}}
\begin{tabular}{@{}p{3.0cm} p{1.2cm} p{2.55cm} p{1.5cm} p{2.55cm} p{1.5cm} @{}}
\hline
 &  &  \multicolumn{2}{c}{LP} & \multicolumn{2}{c}{AP}  \\
 \cline{3-4} \cline{5-6}
  & $\sigma$  &  $L^1$ error & $y_{max}$ & $L^1$ error & $y_{max}$ \\[3pt] \hline
\multirow{3}{*}{Slotted Cylinder}  
	& 0.0 &	2.5900$\times 10^{-2}$ & 1.0000 & 2.5958$\times 10^{-2}$ & 1.0000 \\
	& 0.5 &	2.8022$\times 10^{-2}$ & 0.9912 & 2.8005$\times 10^{-2}$ & 0.9894 \\
	& 1.0 &	3.0557$\times 10^{-2}$ & 0.9681 & 3.0564$\times 10^{-2}$ & 0.9674 \\
\hline	
\multirow{3}{*}{Cone}
	& 0.0 &	2.9773$\times 10^{-3}$ & 0.8709 & 2.9854$\times 10^{-3}$ & 0.8725 \\
	& 0.5 &	2.1664$\times 10^{-3}$ & 0.8434 & 2.1676$\times 10^{-3}$ & 0.8430 \\
	& 1.0 &	2.4633$\times 10^{-3}$ & 0.8190 & 2.4643$\times 10^{-3}$ & 0.8188 \\
\hline	
\multirow{3}{*}{Hump}
	& 0.0 &	1.2495$\times 10^{-3}$ & 0.4947 & 1.2527$\times 10^{-3}$ & 0.4946 \\
	& 0.5 &	1.2132$\times 10^{-3}$ & 0.4645 & 1.2100$\times 10^{-3}$ & 0.4631 \\
	& 1.0 &	1.4077$\times 10^{-3}$ & 0.4247 & 1.4060$\times 10^{-3}$ & 0.4248 \\
\hline	
\end{tabular}
\end{table*}

\begin{table*}[!t]
\caption{\label{tab4} $L^1$-norm of errors and maximum values of numerical solutions for the solid body rotation test with the QUICK scheme. Flux limiters are calculated by using the exact (LP) and approximate (AP) solutions of the linear programming problem \eqref{eq:244}-\eqref{eq:247}}
\centering
%\begin{tabular}{@{}p{2.45cm} p{1.2cm} *{2}{ p{2.55cm} p{1.5cm}@{}}}
\begin{tabular}{@{}p{3.0cm} p{1.2cm} p{2.55cm} p{1.5cm} p{2.55cm} p{1.5cm} @{}}
\hline
 &  &  \multicolumn{2}{c}{LP} & \multicolumn{2}{c}{AP}  \\
 \cline{3-4} \cline{5-6}
  & $\sigma$  &  $L^1$ error & $y_{max}$ & $L^1$ error & $y_{max}$ \\[3pt] \hline
\multirow{3}{*}{Slotted Cylinder}  
	& 0.0 &	1.3892$\times 10^{-2}$ & 1.0000 & 1.3860$\times 10^{-2}$ & 1.0000 \\
	& 0.5 &	1.9927$\times 10^{-2}$ & 1.0000 & 1.9944$\times 10^{-2}$ & 1.0000 \\
	& 1.0 &	2.5260$\times 10^{-2}$ & 0.9917 & 2.5271$\times 10^{-2}$ & 0.9921 \\
\hline	
\multirow{3}{*}{Cone}
	& 0.0 &	1.5878$\times 10^{-3}$ & 0.9328 & 1.5886$\times 10^{-3}$ & 0.9321 \\
	& 0.5 &	8.4318$\times 10^{-4}$ & 0.8780 & 8.4432$\times 10^{-4}$ & 0.8777 \\
	& 1.0 &	1.6144$\times 10^{-3}$ & 0.8336 & 1.6151$\times 10^{-3}$ & 0.8336 \\
\hline	
\multirow{3}{*}{Hump}
	& 0.0 &	7.6481$\times 10^{-4}$ & 0.4952 & 7.6600$\times 10^{-4}$ & 0.4950 \\
	& 0.5 &	4.2010$\times 10^{-4}$ & 0.4663 & 4.2036$\times 10^{-4}$ & 0.4666 \\
	& 1.0 &	8.9529$\times 10^{-4}$ & 0.4187 & 8.9523$\times 10^{-4}$ & 0.4191 \\
\hline	
\end{tabular}
\end{table*}

\subsection{One-Dimensional Scalar Nonconvex Conservation Law}  \label{Sec53}
In this section we consider the following Riemann problem for nonconvex conservation law~\cite{b34} 
\begin{equation}
\label{eq:59} 
  {\rho _t} + {f_x}(\rho ) = 0										
\end{equation}

\begin{equation}
\label{eq:510} 
 \rho (x,0) = \begin{cases}
2 \qquad  &{\rm if} \;\; x \le 1 \\
- 2 \qquad & {\rm if} \;\; x > 1
\end{cases} 									
\end{equation}
where 
\begin{equation}
\label{eq:511} 
 f(\rho ) = \frac{1}{4}\left( {{\rho ^2} - 1} \right)\left( {{\rho ^2} - 4} \right)  							\end{equation}

%=====================================================%
%figure 6
\begin{figure*}[!t]
  \centering 
  \includegraphics[scale=0.73]{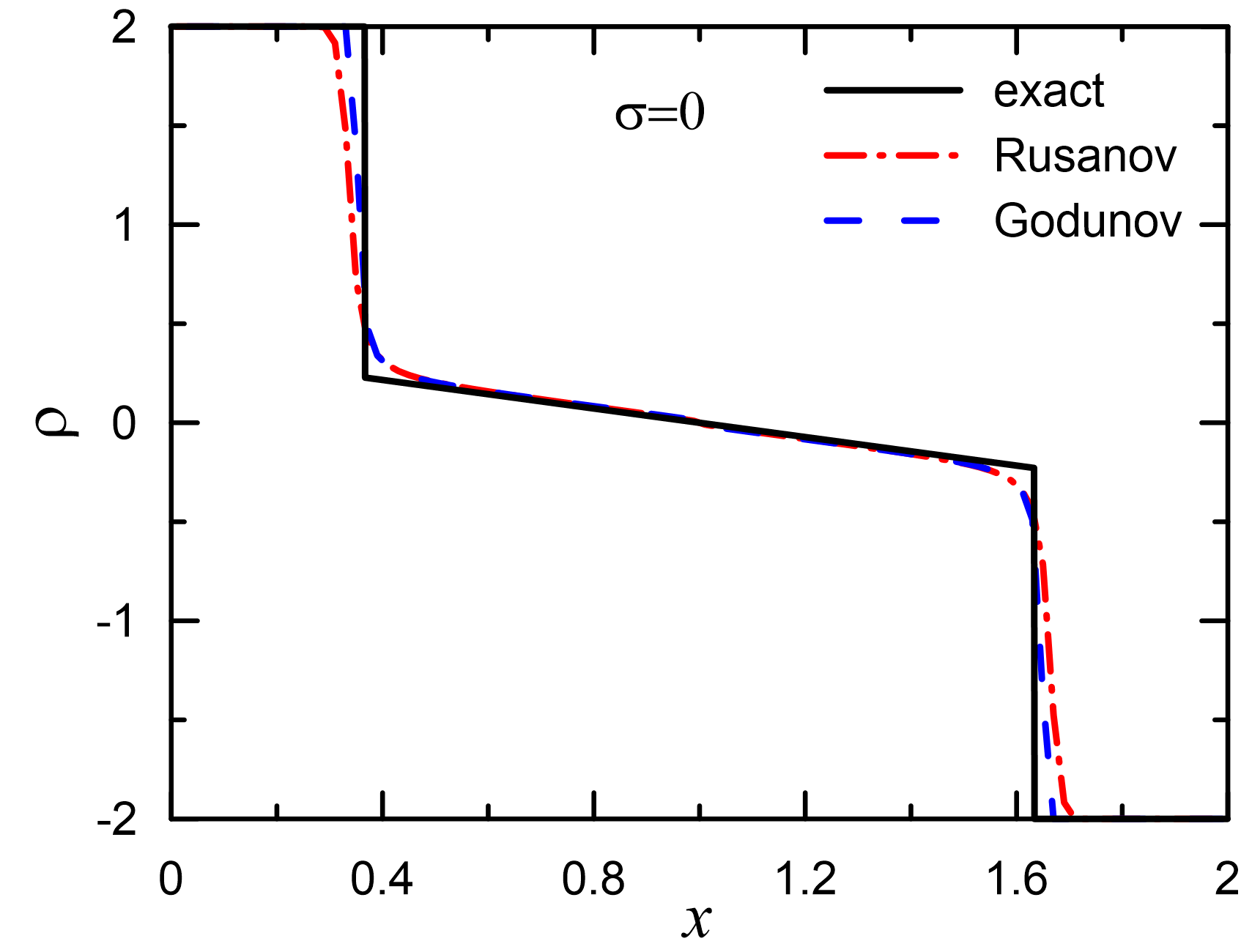}
  \includegraphics[scale=0.73]{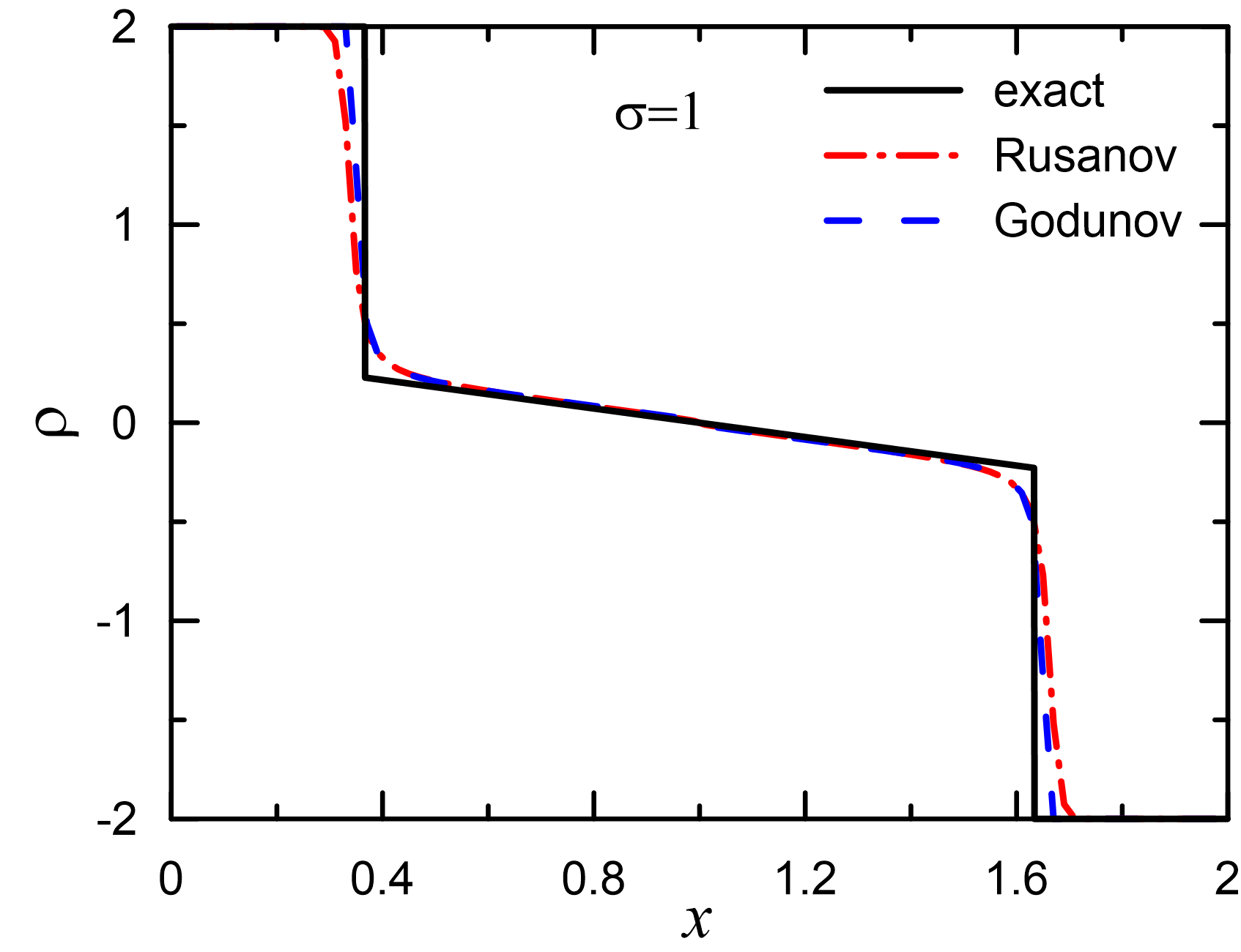}
%  \includegraphics[height=4.5cm]{figs/fig_6b.png}
\caption{Numerical solutions of the Riemann problem for the nonconvex conservation law \eqref{eq:59}-\eqref{eq:511} obtained by using the explicit and implicit monotone schemes with the Godunov and the Rusanov numerical fluxes}
\label{fig:6}       % Give a unique label
\end{figure*}
%=====================================================%
%=====================================================%
%figure 7
\begin{figure*}[!tb]
  \centering 
  \includegraphics[scale=0.73]{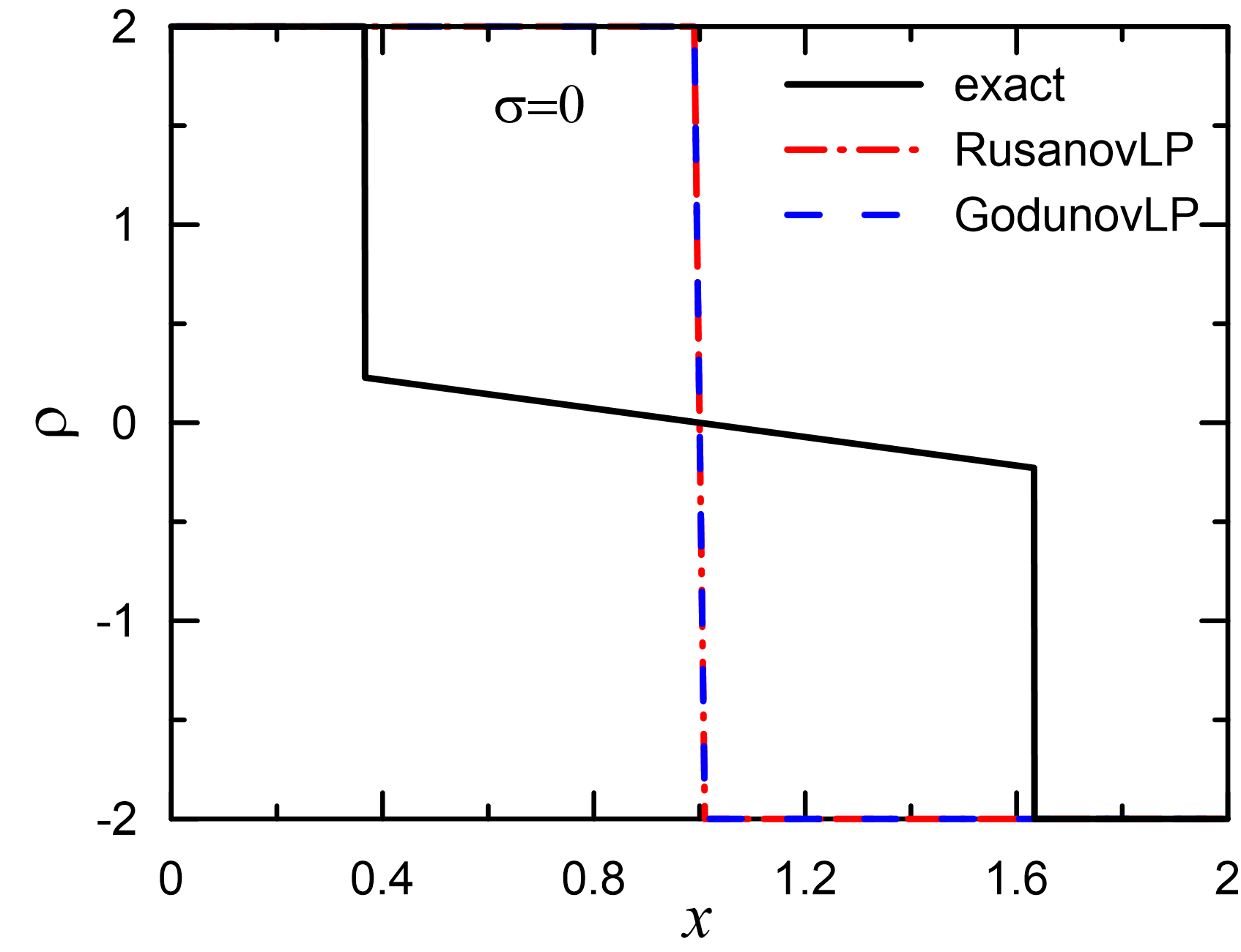}
  \includegraphics[scale=0.73]{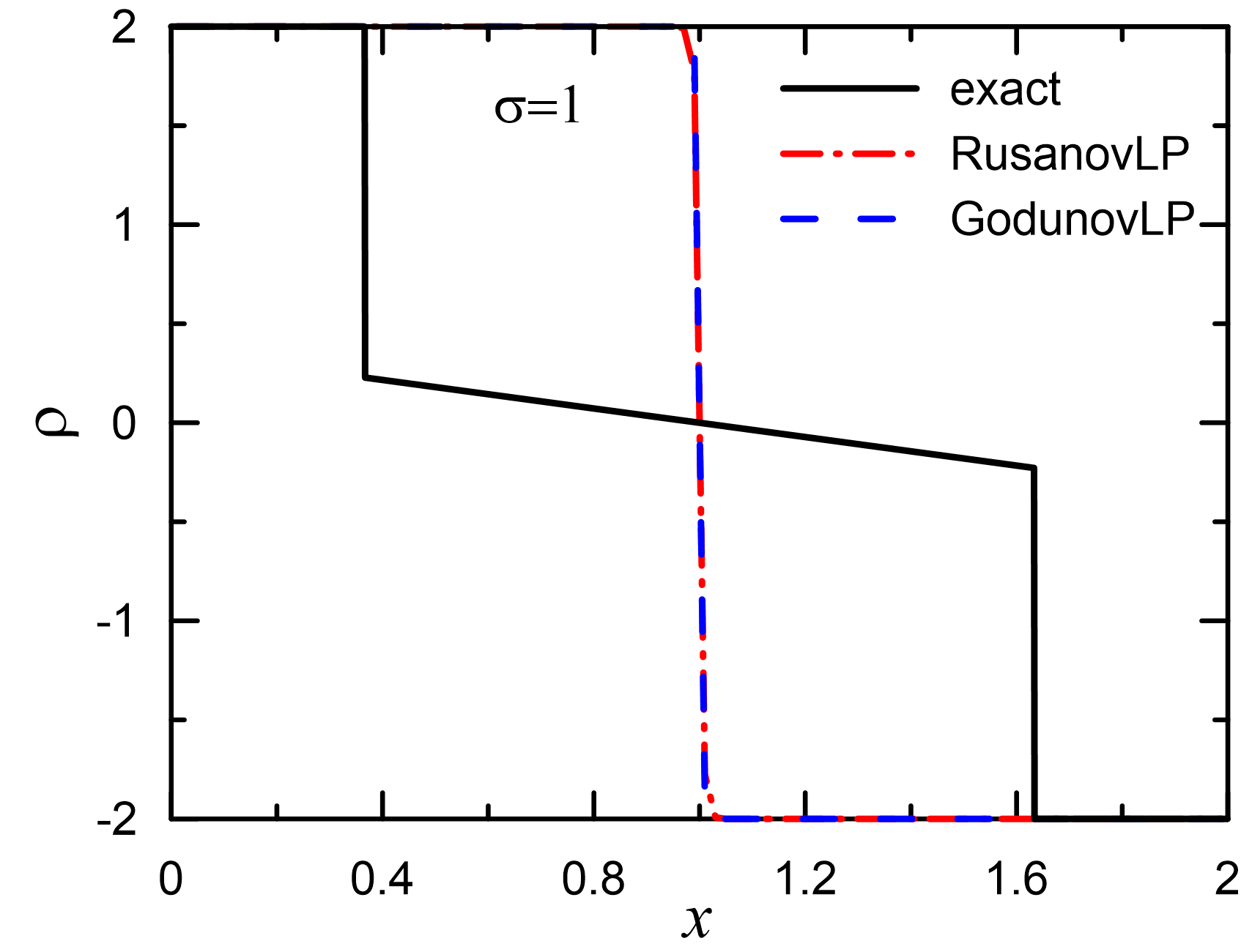}
%  \includegraphics[height=4.5cm]{figs/fig_6b.png}
\caption{Numerical solutions of the Riemann problem for the nonconvex conservation law \eqref{eq:59}-\eqref{eq:511} obtained by using the weighted conservative scheme \eqref{eq:24}-\eqref{eq:25} with the Godunov and Rusanov numerical fluxes. Flux limiters are calculated by linear programming without taking into account the discrete entropy inequality}
\label{fig:7}       % Give a unique label
\end{figure*}
%=====================================================%

%=====================================================%
%figure 8
\begin{figure*}[!tb]
  \centering 
  \includegraphics[scale=0.73]{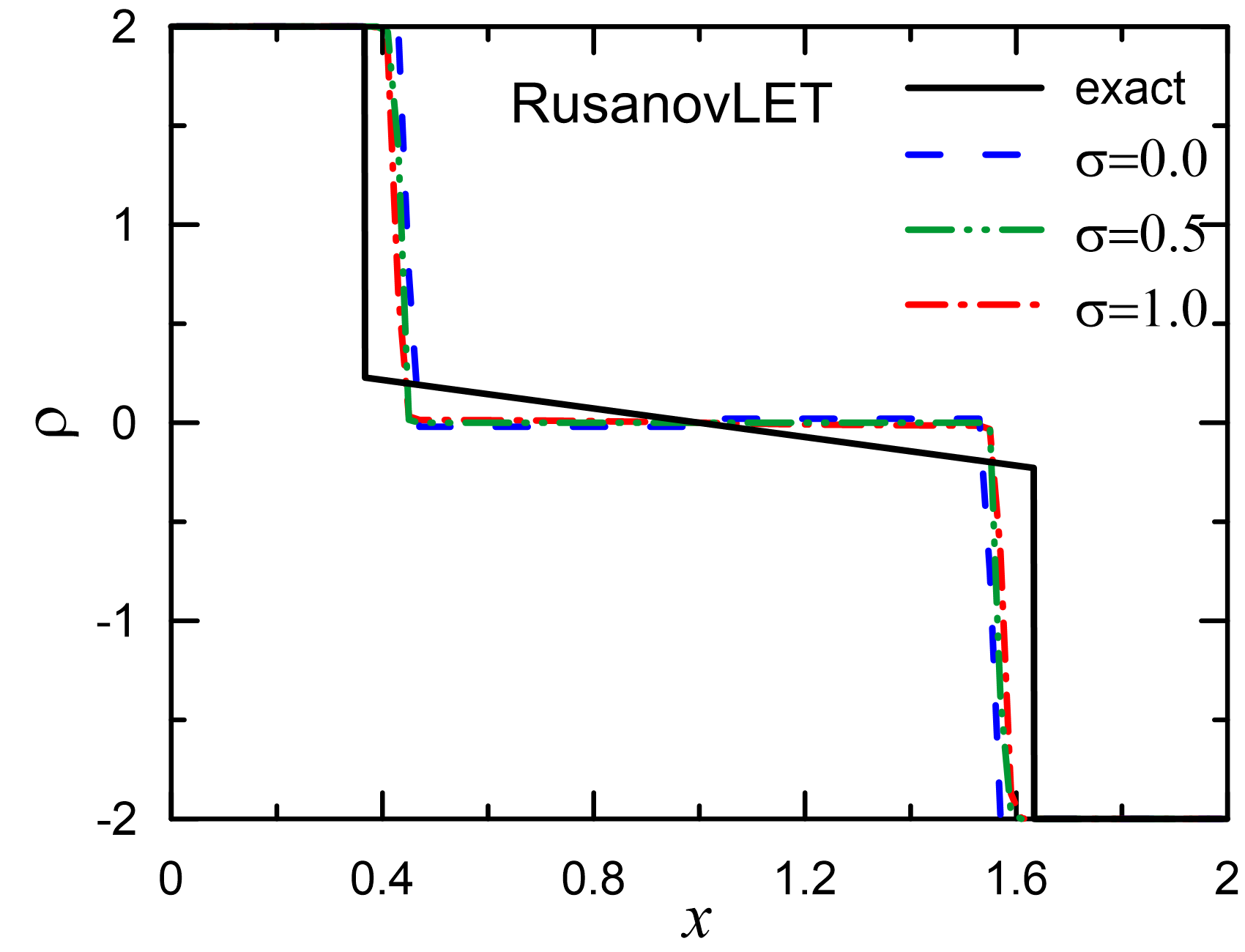}
  \includegraphics[scale=0.73]{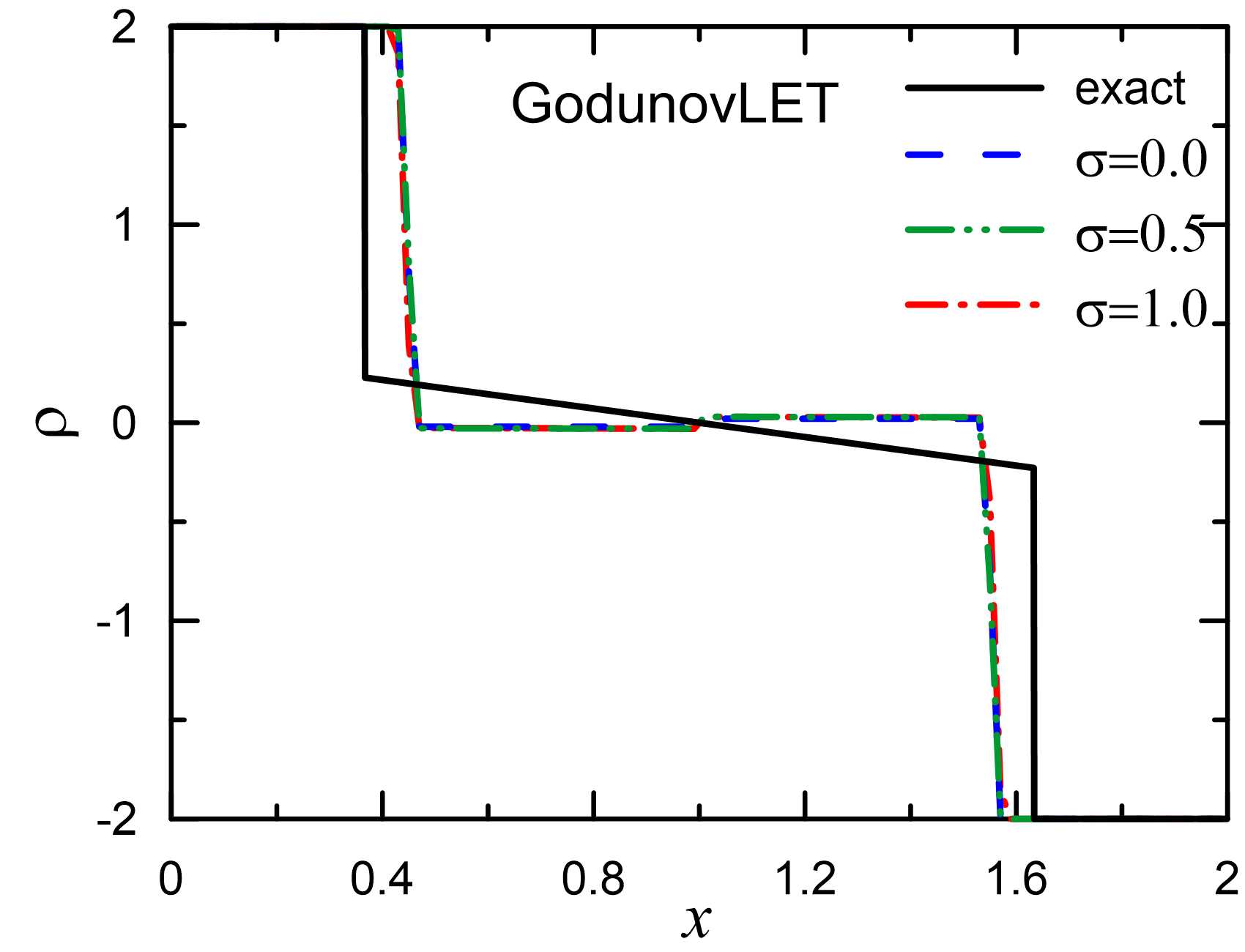}
%  \includegraphics[height=4.5cm]{figs/fig_6b.png}
\caption{Numerical solutions of the Riemann problem for the nonconvex conservation law \eqref{eq:59}-\eqref{eq:511} obtained by using the weighted conservative scheme \eqref{eq:24}-\eqref{eq:25} with the Godunov and Rusanov numerical fluxes. Flux limiters are computed by linear programming and satisfy the discrete entropy inequality \eqref{eq:514} with the Tadmor’s numerical entropy flux}
\label{fig:8}       % Give a unique label
\end{figure*}
%=====================================================%
%=====================================================%
%figure 9
\begin{figure*}[!tb]
  \centering 
  \includegraphics[scale=0.73]{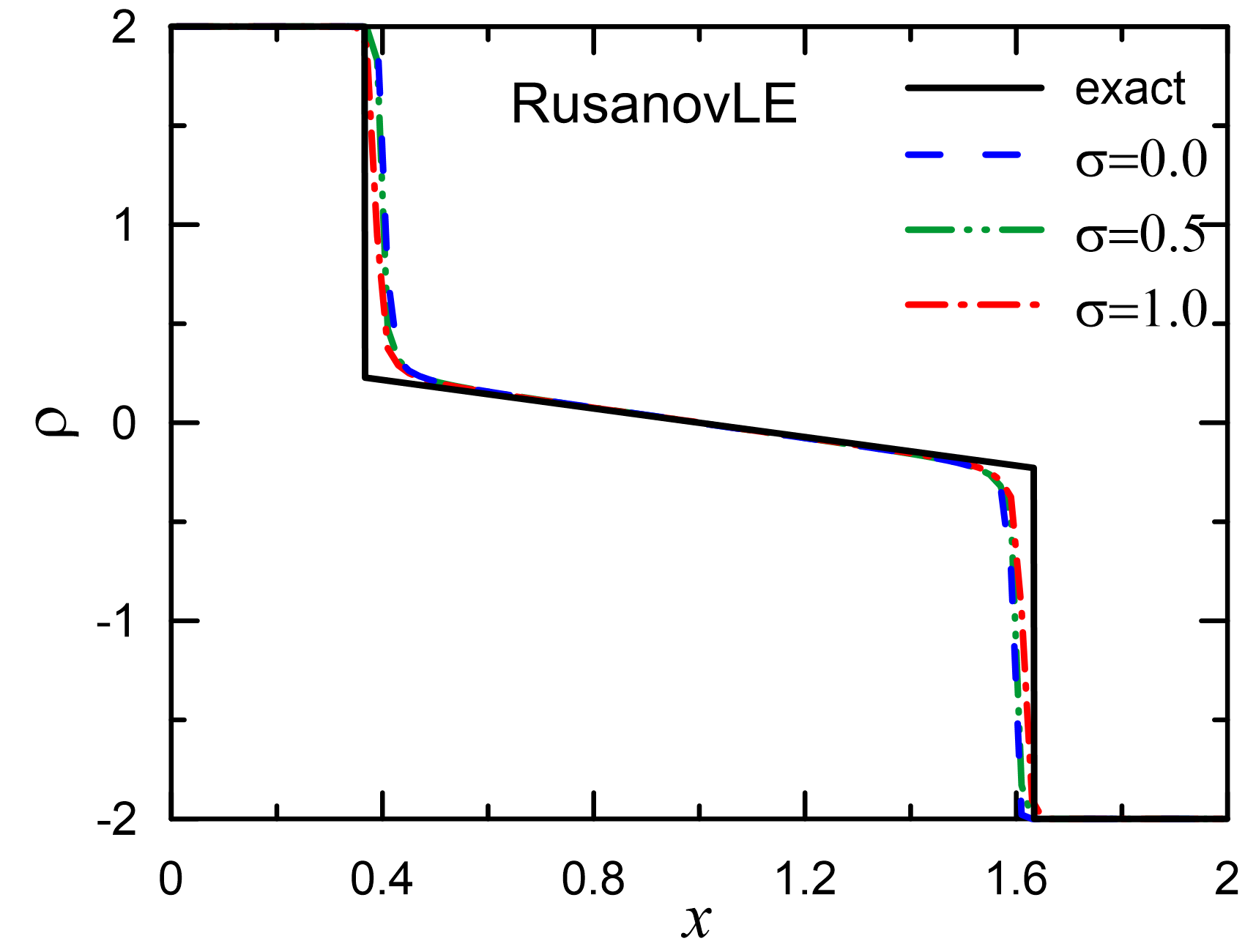}
  \includegraphics[scale=0.73]{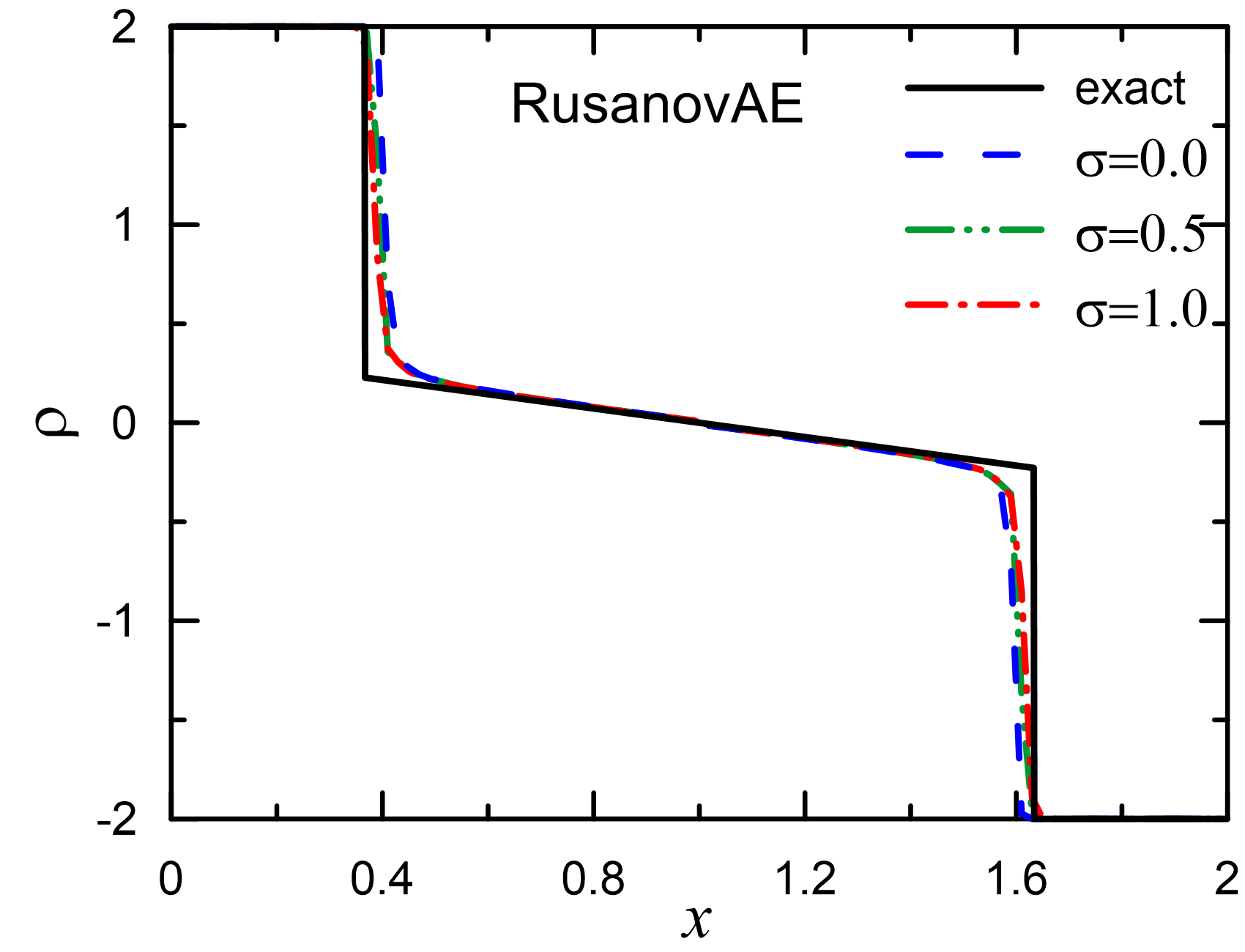}
%  \includegraphics[height=4.5cm]{figs/fig_6b.png}
\caption{Numerical solutions of the Riemann problem for the nonconvex conservation law \eqref{eq:59}-\eqref{eq:511} obtained by using the weighted conservative scheme \eqref{eq:51} with the Rusanov numerical flux. Flux limiters are calculated from exact (left) and approximate (right) solutions of the linear programming problem and satisfy the discrete entropy inequality \eqref{eq:52} with the proper numerical entropy flux}
\label{fig:9}       % Give a unique label
\end{figure*}
%=====================================================%

%=====================================================%
%figure 10
\begin{figure*}[!tb]
  \centering 
  \includegraphics[scale=0.73]{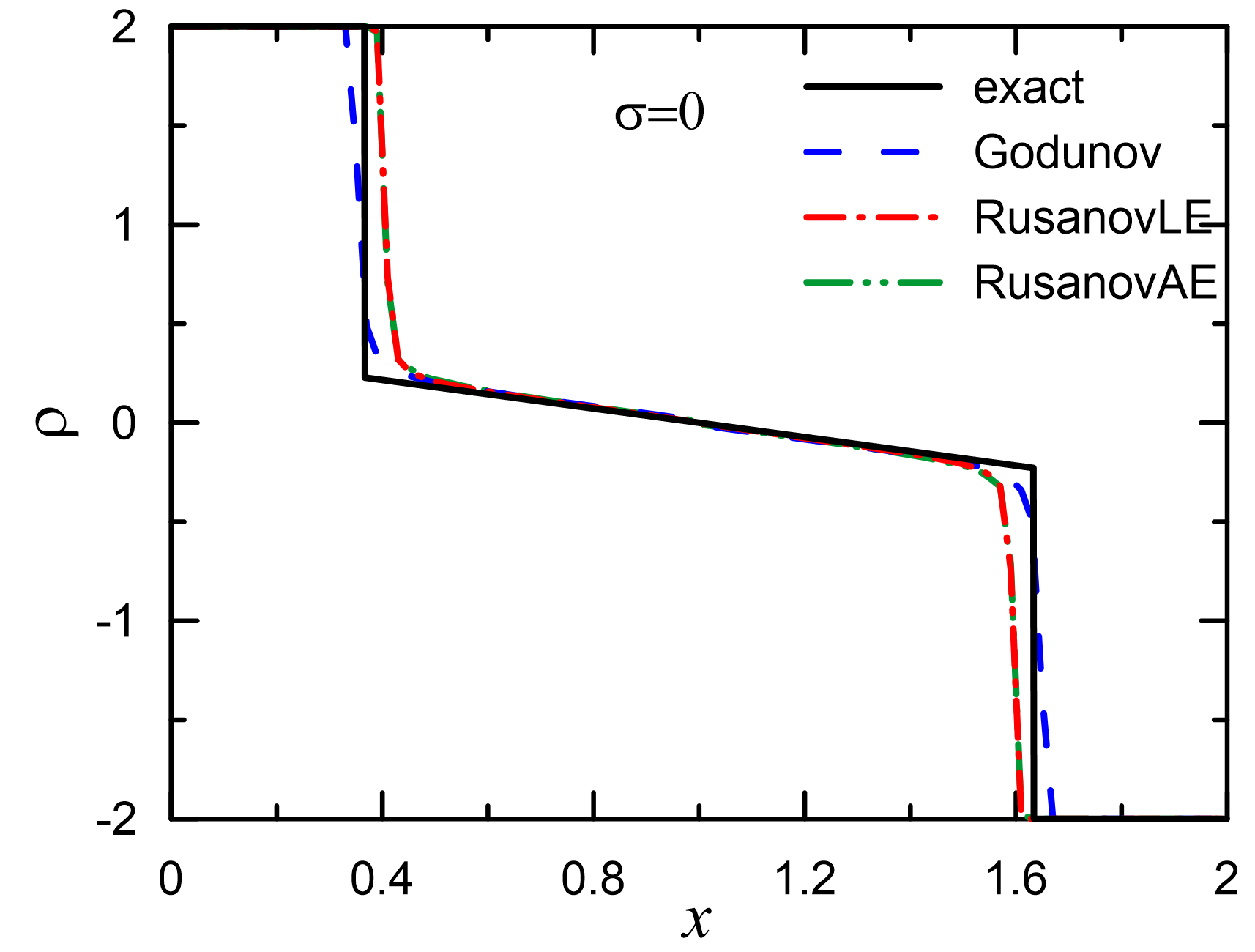}
  \includegraphics[scale=0.73]{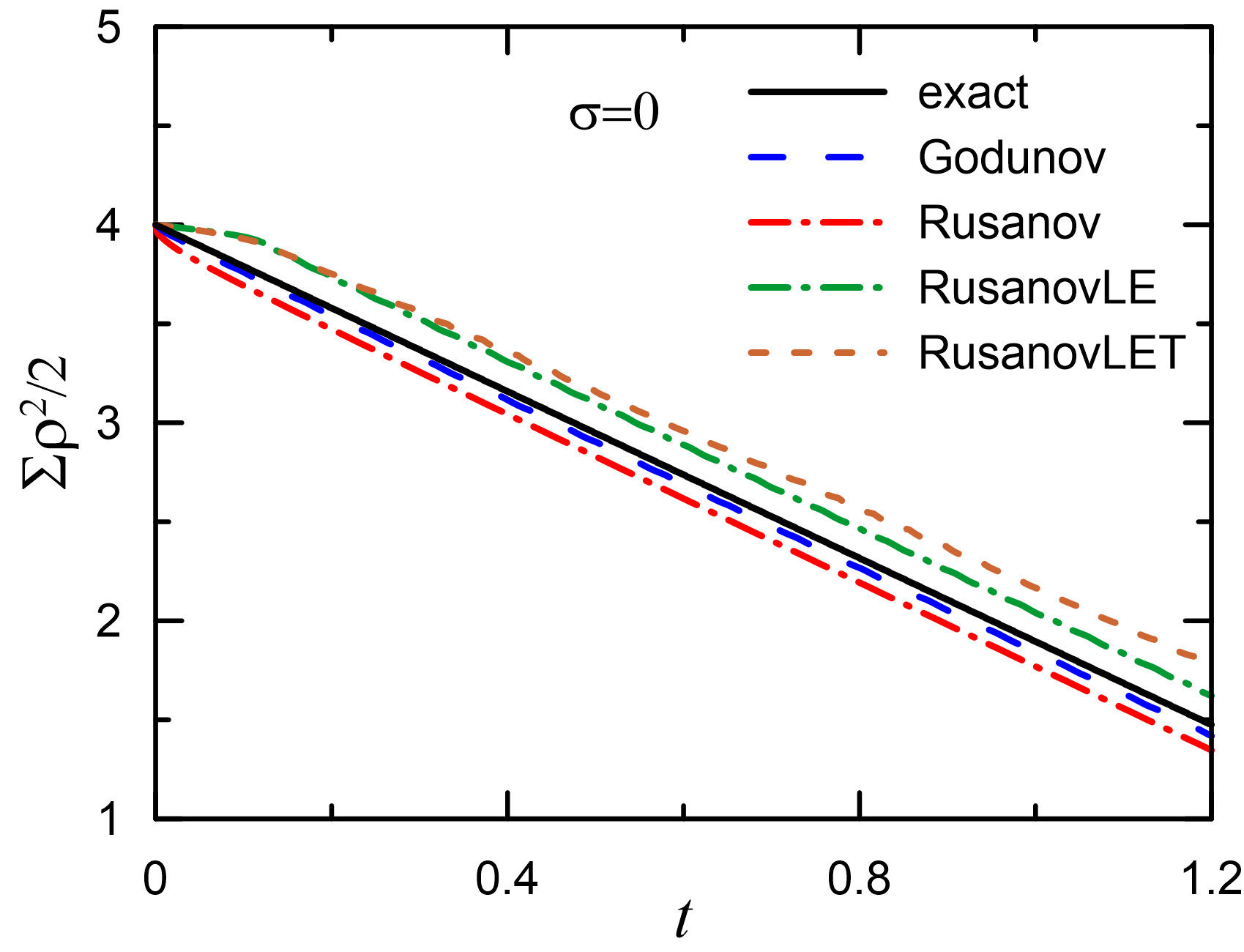}
%  \includegraphics[height=4.5cm]{figs/fig_6b.png}
\caption{Comparison of the numerical solutions of the Riemann problem for the nonconvex conservation law  \eqref{eq:59}-\eqref{eq:511} with explicit Godunov and Rusanov schemes (left). Time series of the entropy integral for the numerical solutions on the interval [0,2] (right)}
\label{fig:10}       % Give a unique label
\end{figure*}
%=====================================================%
%=====================================================%
%figure 11
\begin{figure*}[!tb]
  \centering 
  \includegraphics[scale=0.73]{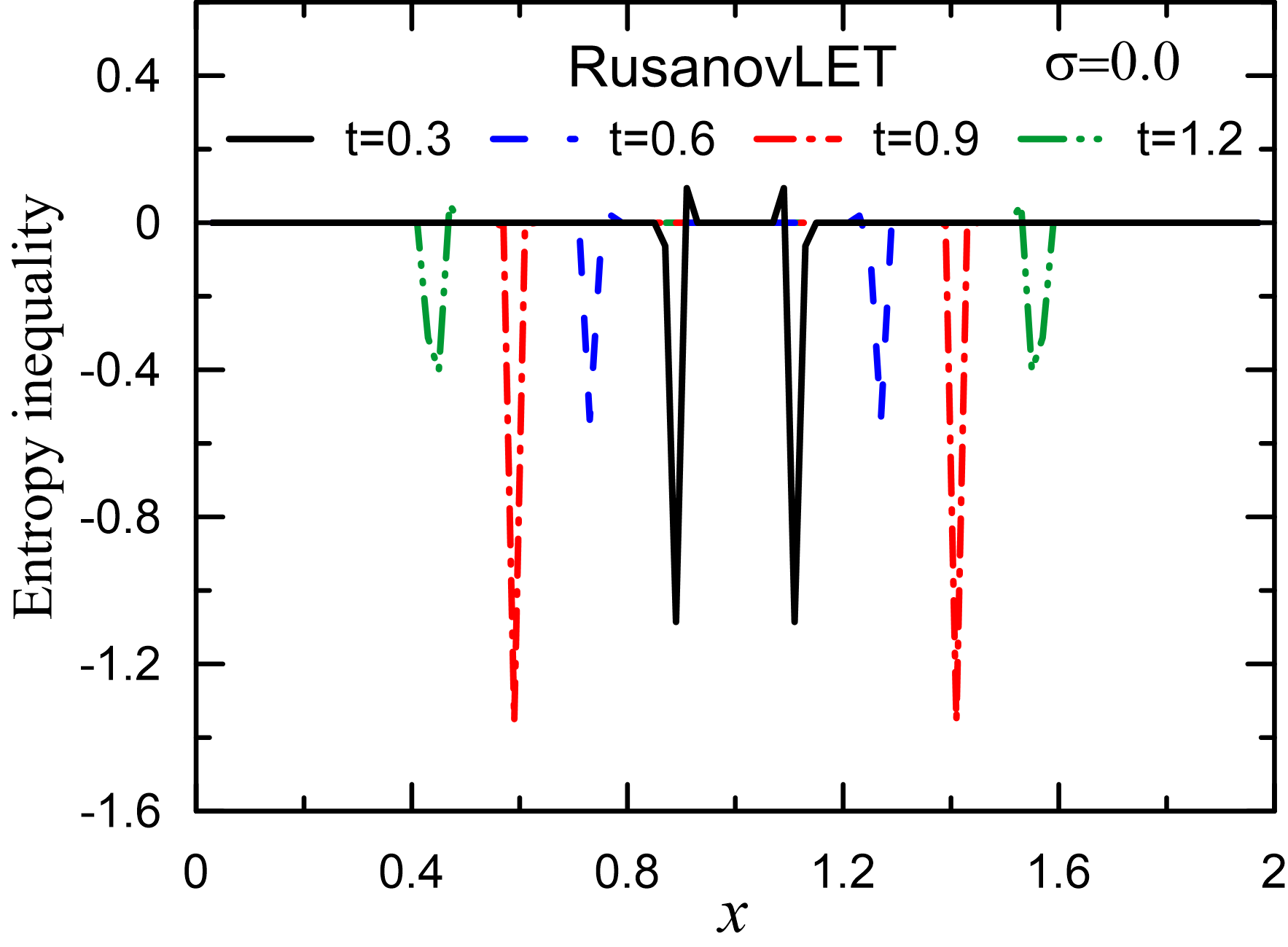}
  \includegraphics[scale=0.73]{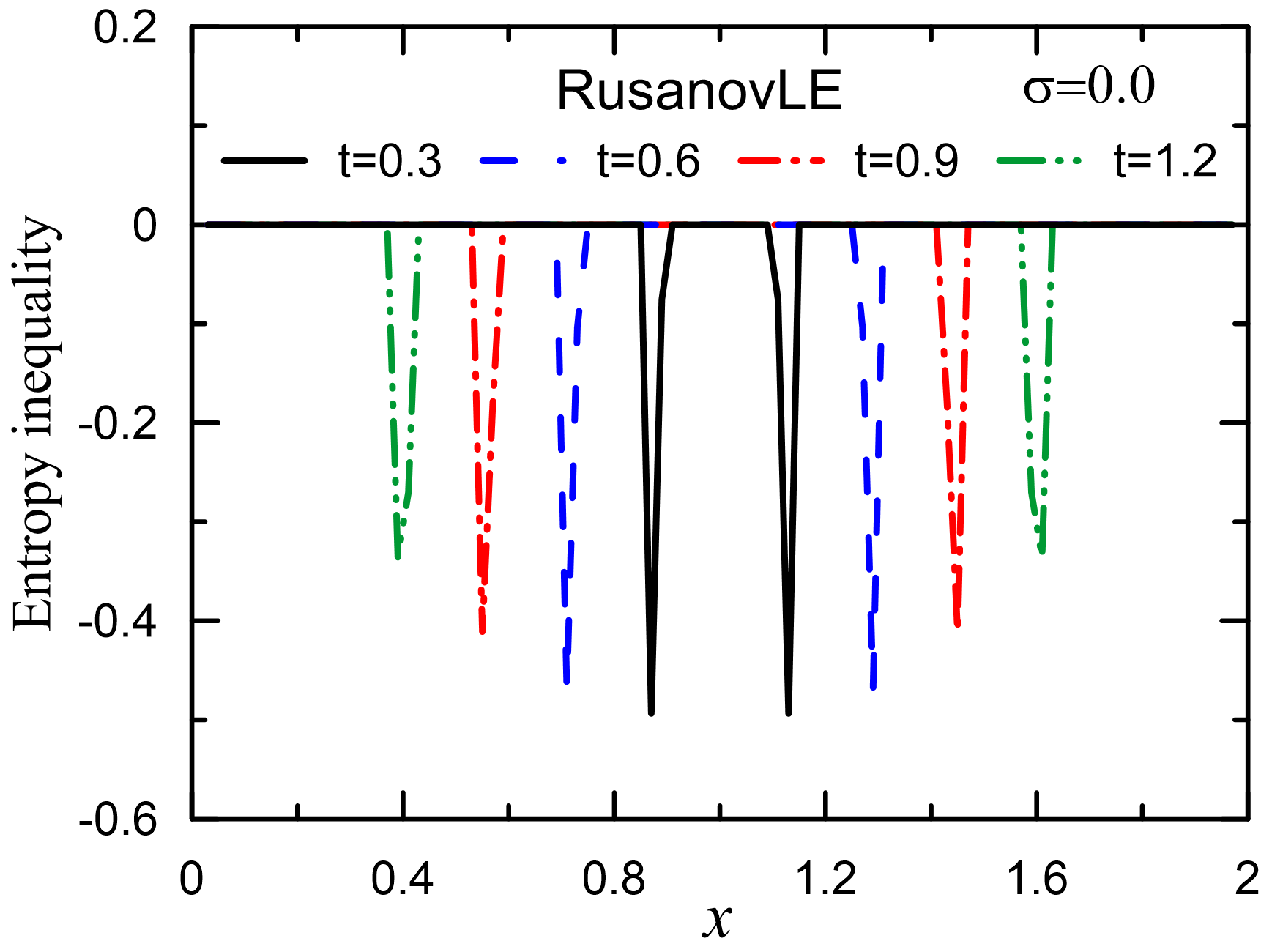}
%  \includegraphics[height=4.5cm]{figs/fig_6b.png}
\caption{IVP for the nonconvex conservation law \eqref{eq:59}-\eqref{eq:511}. The values of the discrete cell entropy inequality with the proper numerical entropy flux for explicit numerical solutions RusanovLET (left) and RusanovLE (right)  at different times}
\label{fig:11}       % Give a unique label
\end{figure*}
%=====================================================%

To solve the Riemann problem \eqref{eq:59}-\eqref{eq:511} numerically, we apply difference scheme \eqref{eq:24}-\eqref{eq:25}, for which along with the Rusanov numerical flux \eqref{eq:34} we use the Godunov numerical flux \eqref{eq:55}. The simulation is performed on the uniform grid with the spatial step of 0.02 and the temporal step of 0.002. All simulation results are depicted at time $T=1.2$.

The numerical solutions of the IVP for the nonconvex conservation law \eqref{eq:59}-\eqref{eq:511} obtained by using explicit and implicit  monotone schemes with the Godunov and Rusanov numerical fluxes are shown in Fig.~\ref{fig:6}. In Fig.~\ref{fig:7} we present the numerical solutions of the Riemann problem \eqref{eq:59}-\eqref{eq:511} that were obtained by using the weighted conservative scheme \eqref{eq:24}-\eqref{eq:25} with the Godunov and Rusanov numerical fluxes. In this case, the flux limiters were computed by linear programming without taking into account the discrete entropy condition. These numerical results are physically incorrect and correspond to the results of the traditional FCT method.

In addition to the proper numerical entropy flux \eqref{eq:311}, we also apply the numerical entropy flux proposed by Tadmor in~\cite{b26}
\begin{equation}
\label{eq:512} 
 H_{i+1/2}^T = \frac{1}{2}(v_i + v_{i+1}) \, {g_{i+1/2}} - \frac{1}{2}[\psi (v_i) + \psi (v_{i+1})]						
\end{equation}
where $v = U'(\rho )$ and the entropy function $U$ is strictly convex. Therefore, the mapping $\rho  \leftrightarrow v$ is one-to-one, and we can consider $\rho $ as a function of $v$. The potential flux $\psi (v)$ is defined as
\begin{equation}
\label{eq:513} 
 g(v) = f(\rho (v)) = \frac{d\psi}{dv}(v), \qquad 	\psi (v) = v \, g(v) - F(\rho (v))   					
\end{equation}
It is easy to verify that the Tadmor’s numerical entropy flux is not the proper numerical entropy flux.

According to Tadmor~\cite[p.~464 Eq.~(3.9)]{b26}, we approximate the entropy inequality \eqref{eq:33} by the following semi-discrete scheme 
\begin{equation}
\label{eq:514} 
\begin{split}
 \frac{d}{dt}U(y_i) + \frac{1}{\Delta x_i} \left[ H_{i+1/2}^T - H_{i - 1/2}^T \right] = \frac{1}{2\Delta x_i}\left( {g_{i+1/2} \; 
 \Delta _{i+1/2}v \, - \Delta _{i+1/2}\psi } \right) \\
 + \frac{1}{2\Delta x_i}\left( g_{i-1/2} \; \Delta _{i-1/2}v \, - \Delta _{i-1/2}\psi \right) \le 0 
\end{split}  
\end{equation}

For the entropy function $U(\rho ) = 0.5{\rho ^2}$, the entropy flux and the entropy flux potential are defined as
\[ F(\rho) = \left( \frac{1}{5} \, {\rho ^2}   -  \frac{5}{6} \right)\,{\rho ^3}  \qquad {\rm and} \qquad  \psi (\rho ) = v \, g - F = \frac{1}{20}{\, \rho ^5} - \frac{5}{12} \, {\rho ^3} + \rho \]

Then, for the discrete entropy inequality \eqref{eq:514} to be valid for the weighted scheme \eqref{eq:24}-\eqref{eq:25}, the flux limiters should satisfy the following inequalities
%\begin{multline}
\begin{eqnarray}
\label{eq:515}  
%\begin{split} 
 & \sigma  \left[ {\alpha_{i+1/2}^{n+1} \left( {\dfrac{v_{i+1}^{n+1} + v_i^{n+1}}{2} - v_i^{(\sigma)}} \right) q_{i+1/2}^{G/{Rus},n+1} - \alpha_{i-1/2}^{n+1} \left( {\dfrac{v_i^{n+1} + v_{i-1}^{n+1}}{2} - v_i^{(\sigma )}} \right) q_{i-1/2}^{G /{Rus},n+1}} \right]  \nonumber \\ 
 &  +  (1- \sigma)   \left[ {\alpha _{i+1/2}^n \left( {\dfrac{v_{i+1}^n + v_i^n}{2} - v_i^{(\sigma )}} \right) q_{i+1/2}^{G/{Rus},n} - \alpha _{i-1/2}^n \left( {\dfrac{v_i^n + v_{i - 1}^n}{2} - v_i^{(\sigma )}} \right)  q_{i-1/2}^{G/{Rus},n}} \right]  \nonumber \\
 &  \le  \sigma  \left[ {\dfrac{1}{2} \Delta_{i+1/2} \psi^{n+1} - \left( {\dfrac{v_{i+1}^{n+1} + v_i^{n+1}}{2} - v_i^{(\sigma )}} \right) \, h_{i+1 /2}^{G/{Rus},n+1} + \dfrac{1}{2} \Delta_{i-1/2} \psi^{n+1} } \right. \\
& +  \left. { \left( \dfrac{v_i^{n+1} + v_{i-1}^{n+1}}{2} - v_i^{(\sigma)} \right)\, h_{i-1/2}^{G/{Rus},n+1}} \right]  
   +  \dfrac{\Delta {x_i}}{\Delta t} \left[ \left( {\sigma v_i^{n+1} + (1 - \sigma ) v_i^n} \right) \, \left(y_i^{n+1} - y_i^n \right) - \left( U_i^{n+1} - U_i^n \right) \right]  \nonumber  \\ 
 &  +  (1 - \sigma)  \left[ {\dfrac{1}{2} \Delta_{i+1/2} \, \psi^n - \left( \dfrac{v_{i+1}^n + v_i^n}{2} - v_i^{(\sigma)} \right) \, h_{i+1/2}^{G/{Rus},n} + \dfrac{1}{2} \Delta_{i-1/2}\psi^n + \left( \dfrac{v_i^n + v_{i-1}^n}{2} - v_i^{(\sigma )} \right) \, h_{i-1/2}^{G/{Rus},n}} \right]  \nonumber 
%\end{split} 
\end{eqnarray}
%\end{multline}
where $q_{i + {1/2}}^{G/{Rus}} = 0.5({f_i} + {f_{i + 1}}) - h_{i + {1 /2}}^{G/{Rus}}$.

Numerical solutions of the Riemann problem \eqref{eq:59}-\eqref{eq:511} presented in Fig.~\ref{fig:8} are obtained by using the weighted conservative scheme \eqref{eq:24}-\eqref{eq:25}, the flux limiters of which are computed by linear programming and satisfy the inequalities \eqref{eq:515}. Unfortunately, these numerical results are also physically incorrect. Thus, if a numerical solution of the Riemann problem for scalar hyperbolic conservation law satisfies the discrete entropy inequality with the Tadmor’s numerical entropy flux, this does not mean that this numerical solution is physically correct. Therefore, further, we will only consider the discrete entropy inequality \eqref{eq:52} with the proper numerical entropy flux.

Fig.~\ref{fig:9} shows that the weighted conservative scheme \eqref{eq:51}, the flux limiters of which are computed from exact or approximate solutions of linear programming by taking into account the discrete entropy inequality \eqref{eq:52}, yields the physically correct solution of the Riemann problem \eqref{eq:59}-\eqref{eq:511}. The comparison of the numerical solutions that are obtained by various explicit difference schemes is shown in Fig.~\ref{fig:10} (left). We note a good agreement between the numerical solutions RusanovLE and RusanovAE. Time series of the entropy integral on the interval [0,2] for the numerical solutions that are obtained by different schemes are given in Fig.~\ref{fig:10} (right). Values of the discrete entropy inequality \eqref{eq:52} for explicit numerical solutions RusanovLET and RusanovLE of the IVP \eqref{eq:59}-\eqref{eq:511} at different times are presented in Fig.~\ref{fig:11}. It is easy to see that the solution RusanovLET violates the discrete entropy inequality \eqref{eq:52}.

%=====================================================%
%figure 12
\begin{figure*}[!tb]
\centering 
\begin{tabular}[t]{lcc} 
  \multirow{2}{*}[5.0em]{\includegraphics[scale=0.78]{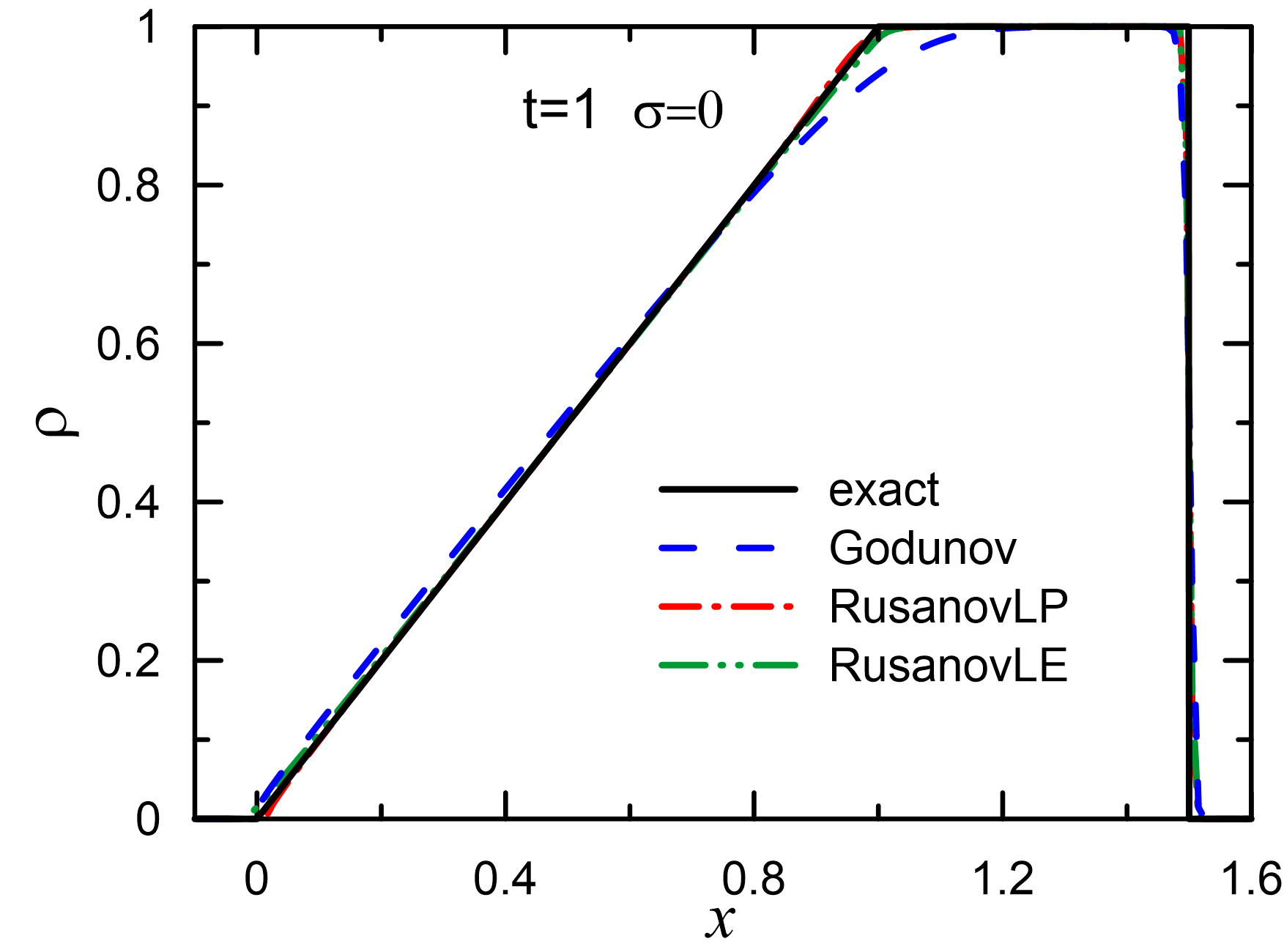}}  
 & \includegraphics[scale=0.52]{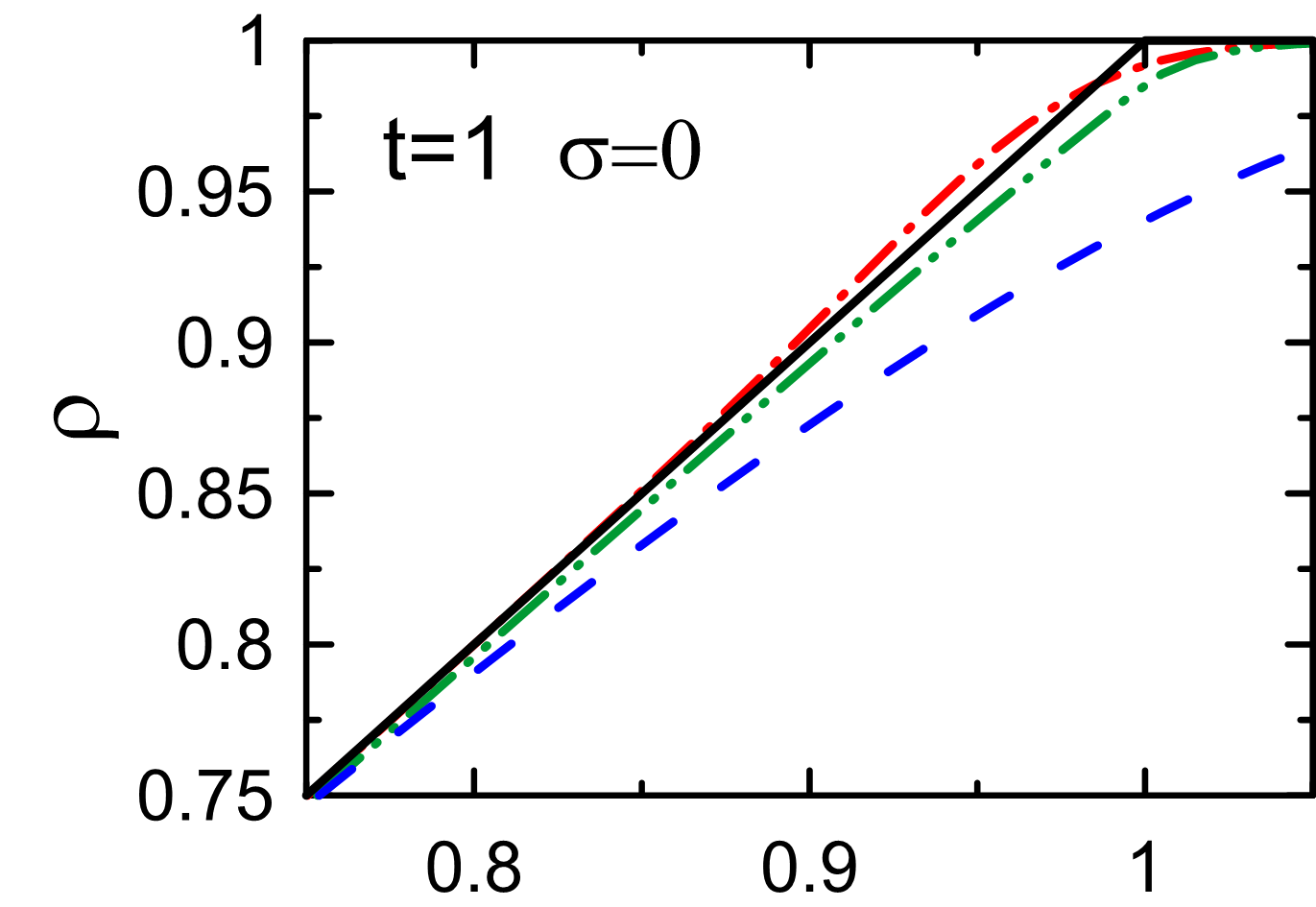} &
  \includegraphics[scale=0.52]{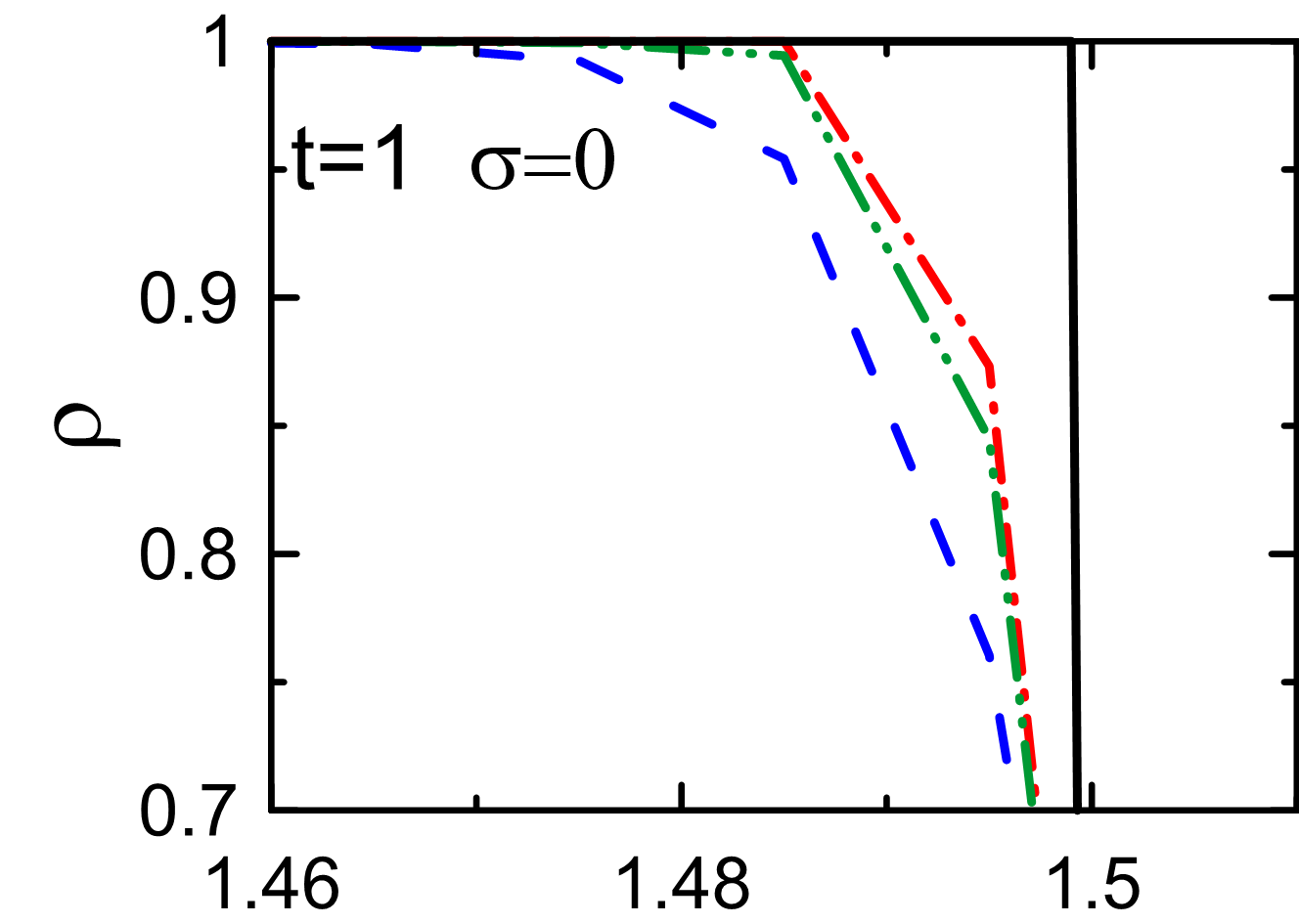} \\
 & \includegraphics[scale=0.52]{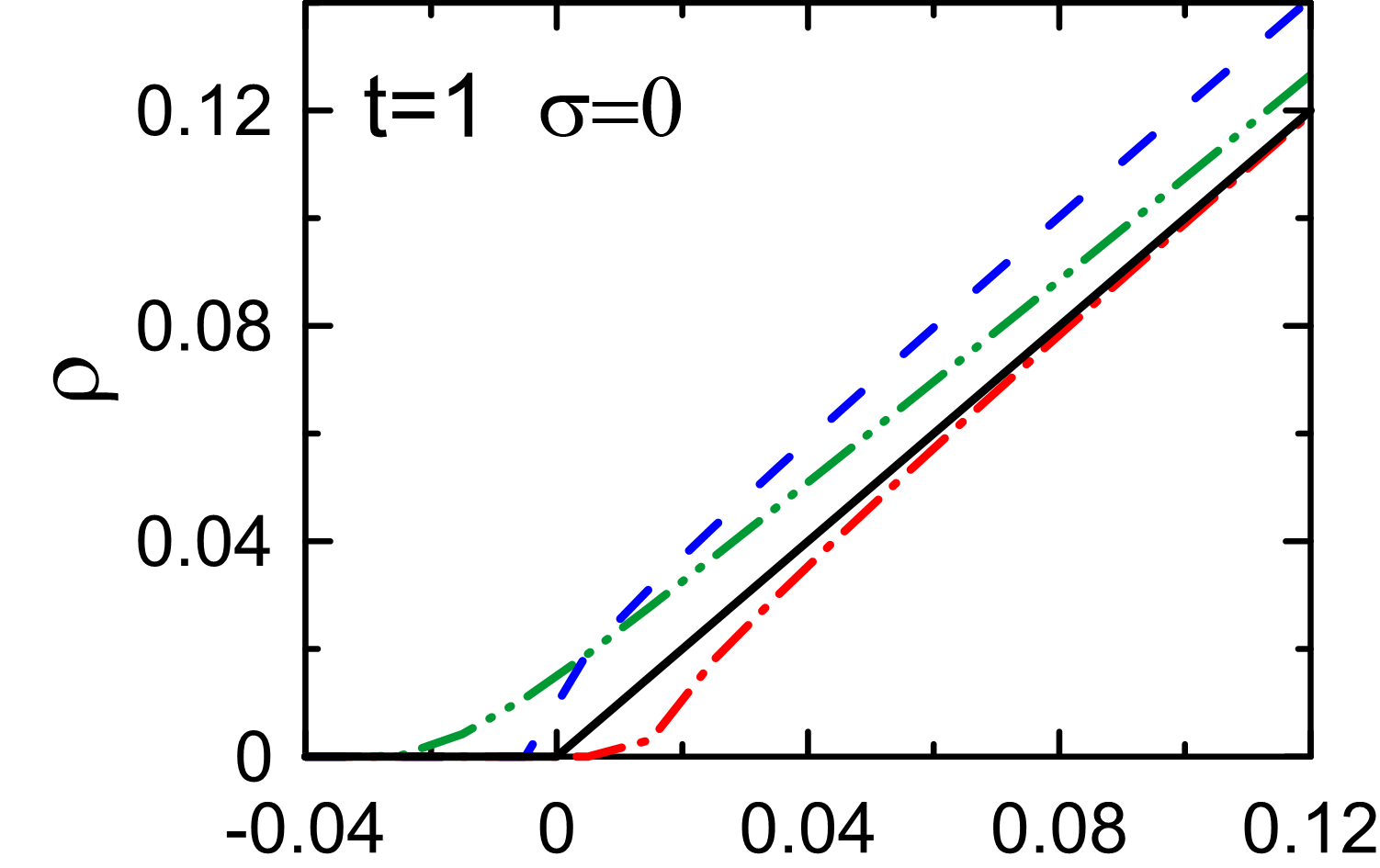} &
  \includegraphics[scale=0.52]{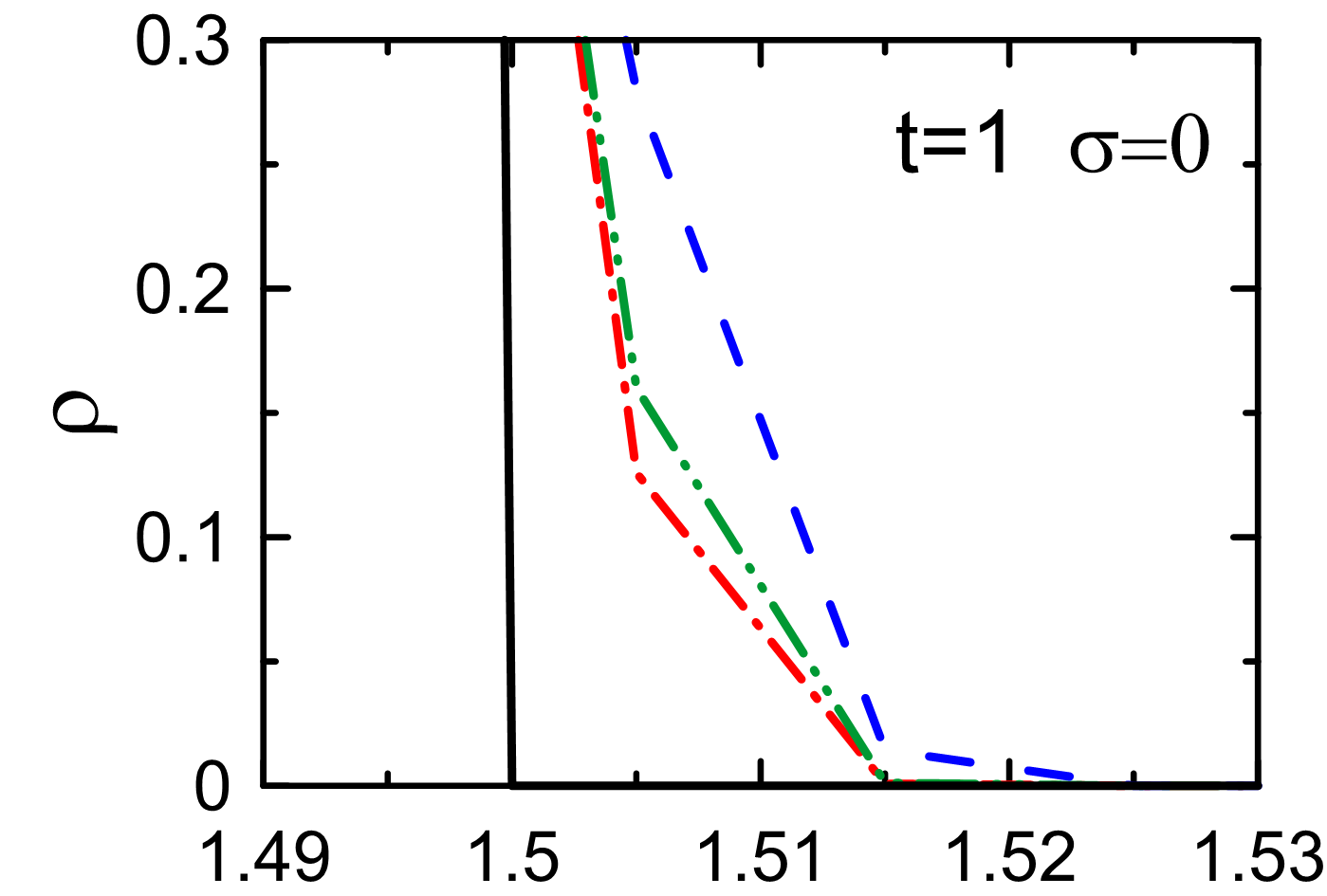} \\
  \multirow{2}{*}[5.0em]{\includegraphics[scale=0.78]{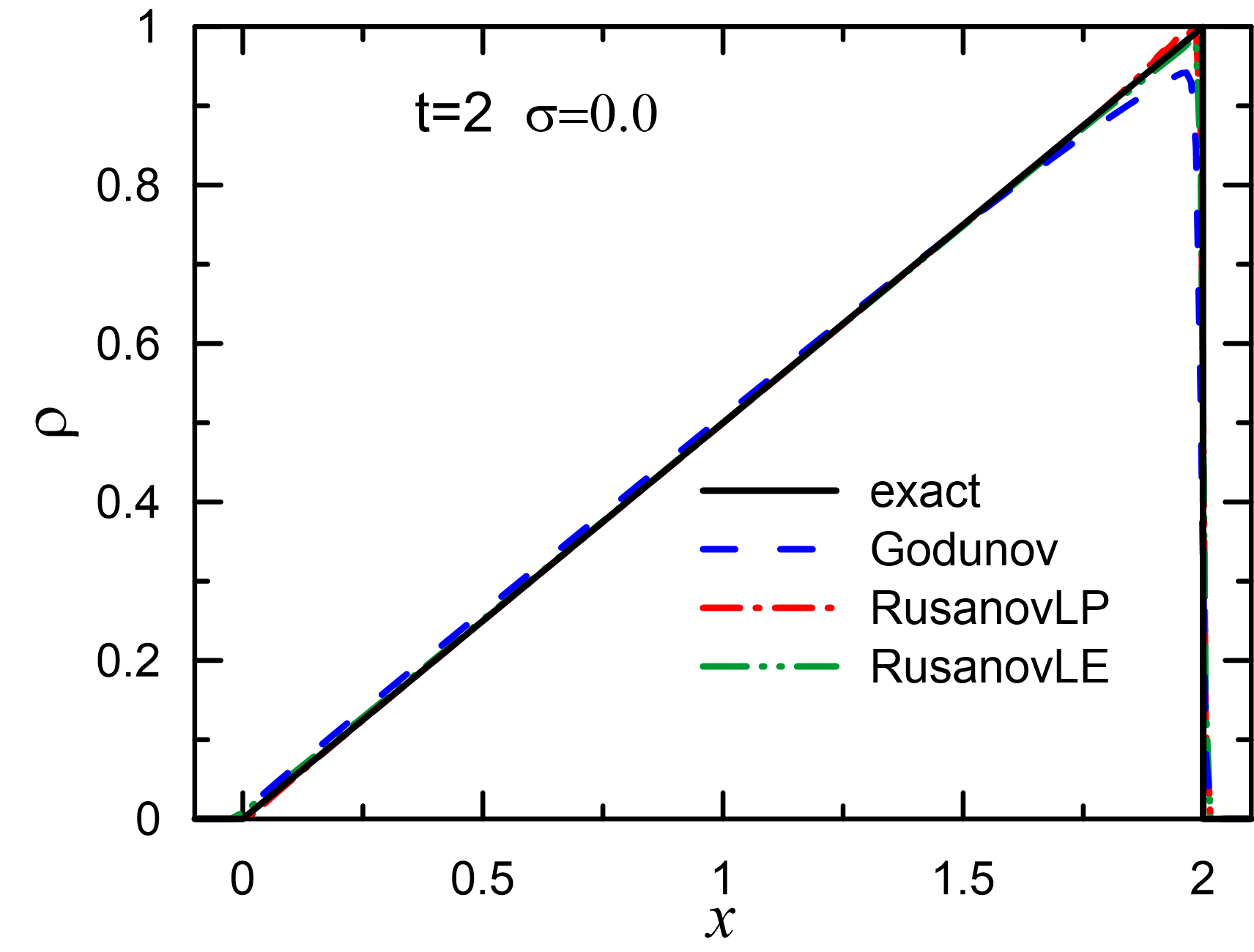}}  
 &  &
  \includegraphics[scale=0.52]{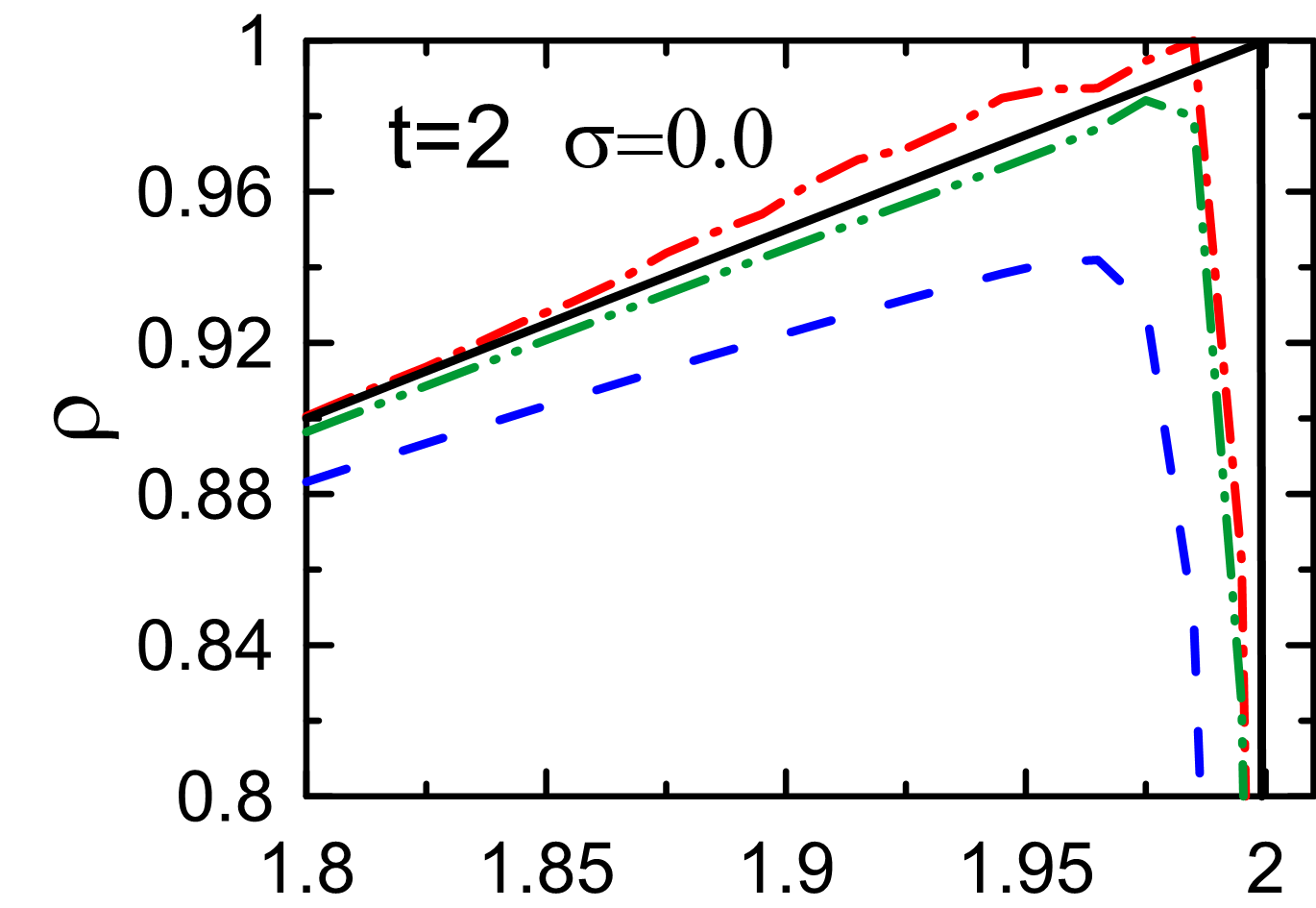} \\
 & \includegraphics[scale=0.52]{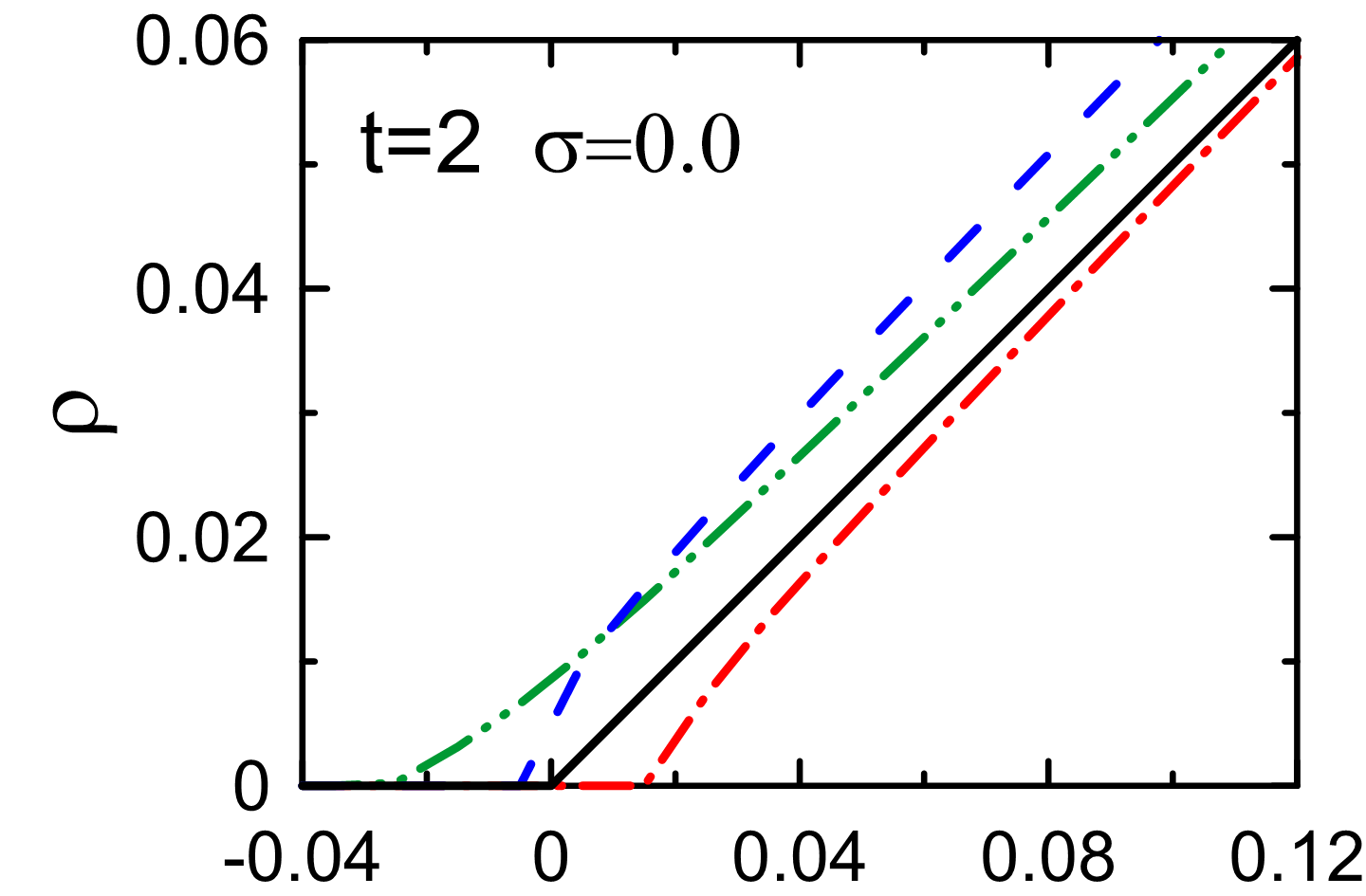} &
  \includegraphics[scale=0.52]{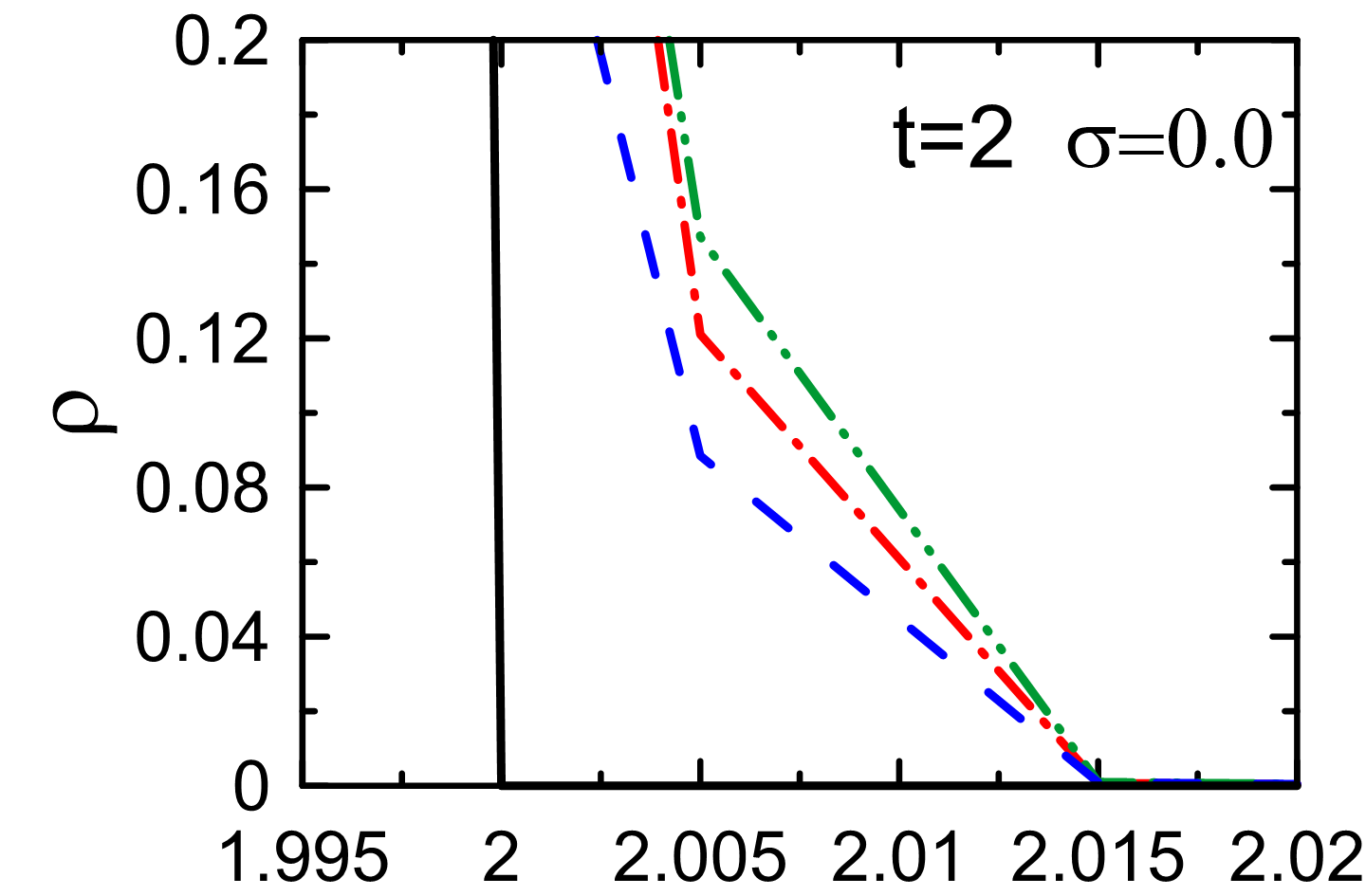} \\
  \multirow{2}{*}[5.0em]{\includegraphics[scale=0.78]{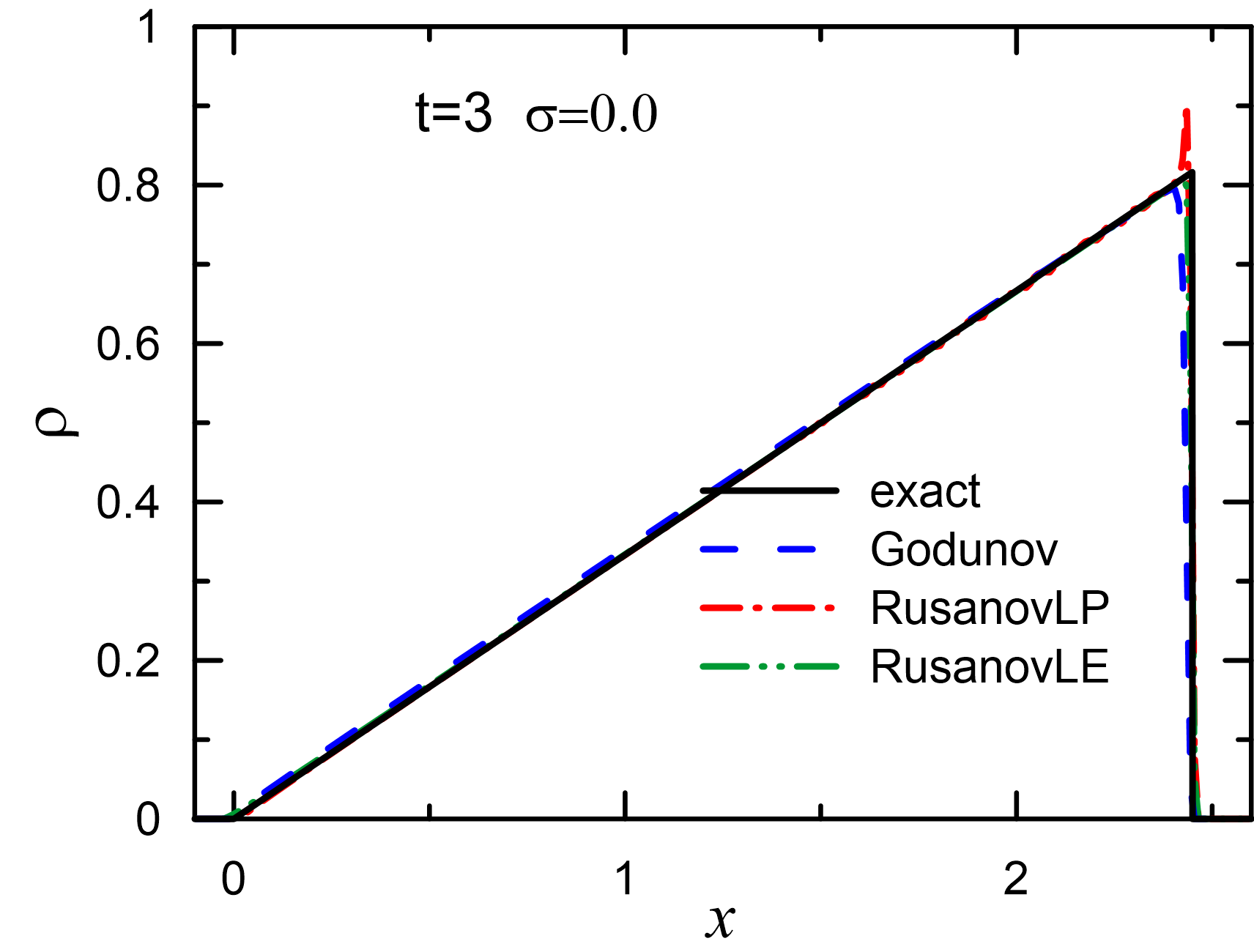}}  
 &  &
  \includegraphics[scale=0.52]{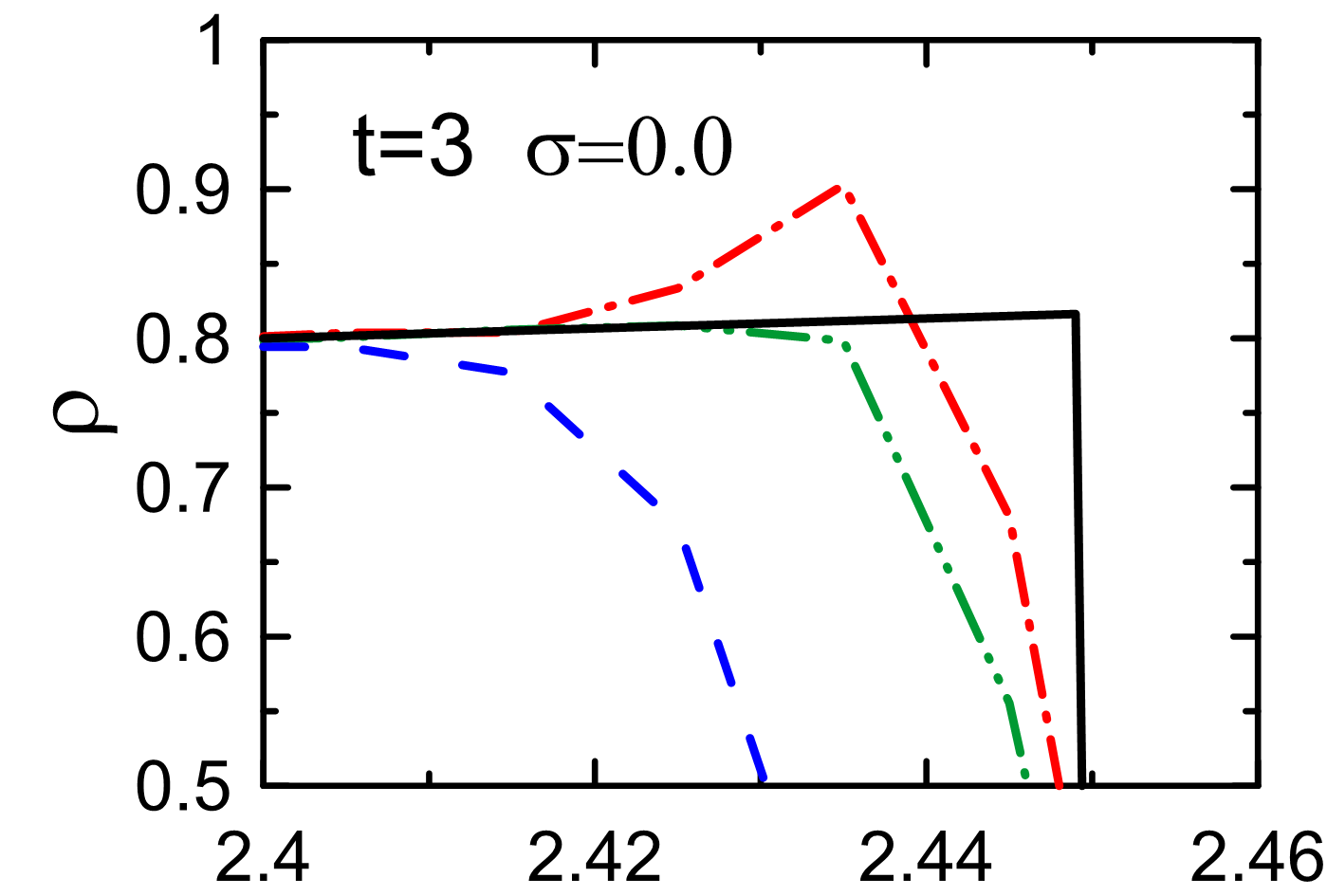} \\
 & \includegraphics[scale=0.52]{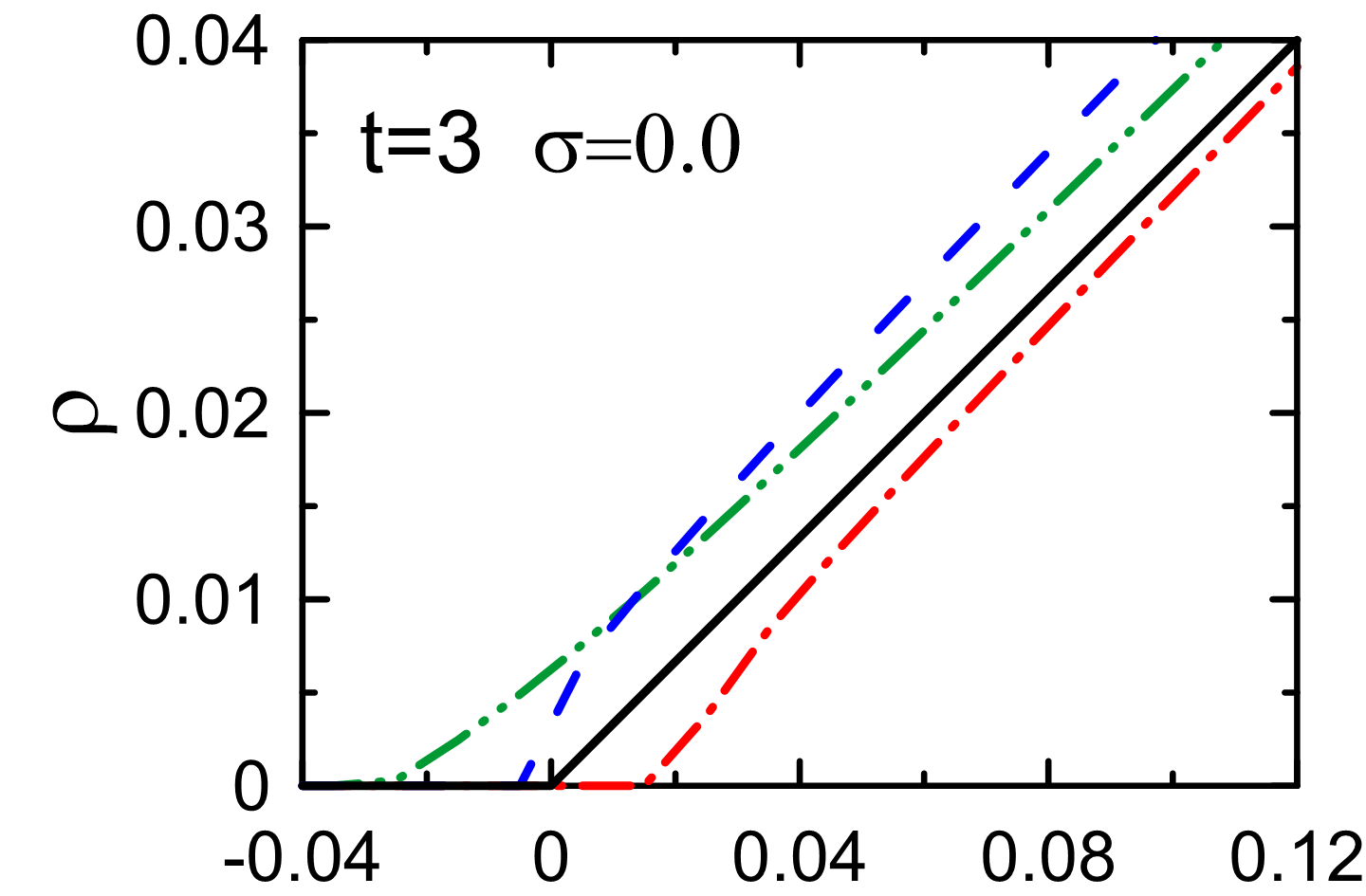} &
  \includegraphics[scale=0.52]{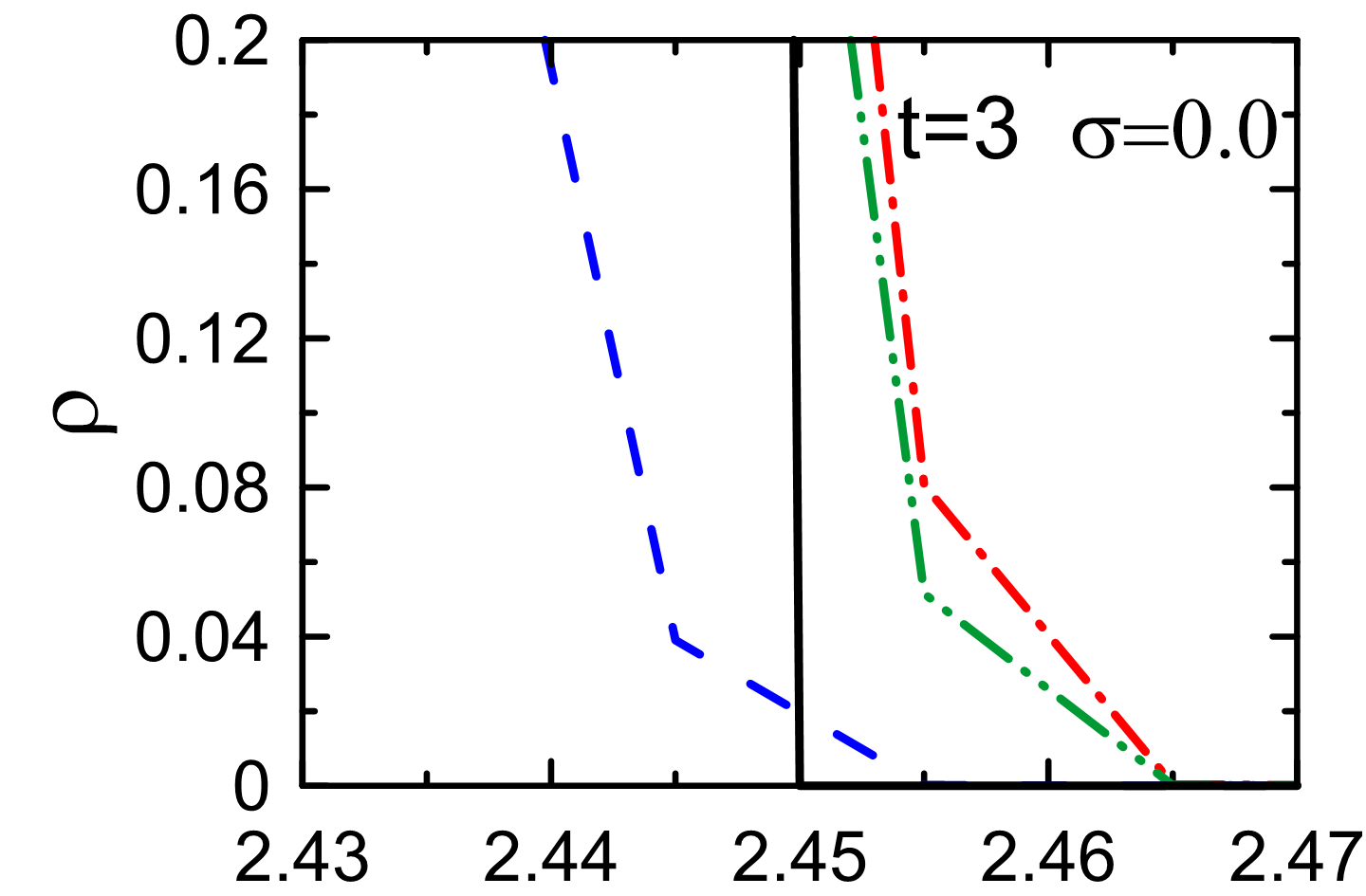} 
\end{tabular}   
\caption{Numerical solutions of the IVP for the inviscid Burgers equation \eqref{eq:516}-\eqref{eq:517} obtained with various explicit schemes at times t=1, t=2  and t=3. Enlarged fragments of the numerical solutions in the corner points are shown in right}
\label{fig:12}       % Give a unique label
\end{figure*}
%=====================================================%

\subsection{Inviscid Burgers Equation}  \label{Sec54}
Consider the IVP for the inviscid Burgers equation
\begin{equation}
\label{eq:516} 
 \frac{\partial \rho }{\partial t} + \frac{1}{2}\frac{\partial \rho^2}{\partial x} = 0 	  	   							
\end{equation}
with the initial condition
\begin{equation}
\label{eq:517} 
  \rho (x,0) = \begin{cases}
0 \qquad & {\rm if} \;\; x < 0 \\
1 \qquad & {\rm if} \;\;  0 \le x \le 1 \\
0 \qquad & {\rm if} \;\; x > 1
\end{cases}
\end{equation}

The well-known solution of \eqref{eq:516}-\eqref{eq:517} for $t \le 2$ consists of a rarefaction wave starting from $x=0$ and a shock wave that travels to the right from $x=1$ with a propagation velocity of 0.5. The shock and rarefaction waves appear together. Then the rarefaction wave encounters with the shock at $t=2$, and a second shock wave is formed at $x=2$ with a speed of $\sqrt 2 $. 

To numerically solve the IVP \eqref{eq:516}-\eqref{eq:517}, we apply the weighted difference scheme \eqref{eq:26} with the Godunov numerical flux \eqref{eq:55} and the weighted difference scheme \eqref{eq:51} with the Rusanov numerical flux. We perform all simulations on uniform grid with the spatial grid size $\Delta x=0.01$ and temporal step size $\Delta t=0.002$. Flux limiters for the scheme \eqref{eq:51} are computed by using exact solution of the linear programming problem \eqref{eq:317}, \eqref{eq:53} with and without taking into account the discrete entropy condition \eqref{eq:54}. For the entropy function $U = \rho^2/2 $, the entropy flux is $F = \rho^3/3$.

%=====================================================%
%figure 15
\begin{figure*}[!t]
  \centering 
  \includegraphics[scale=0.78]{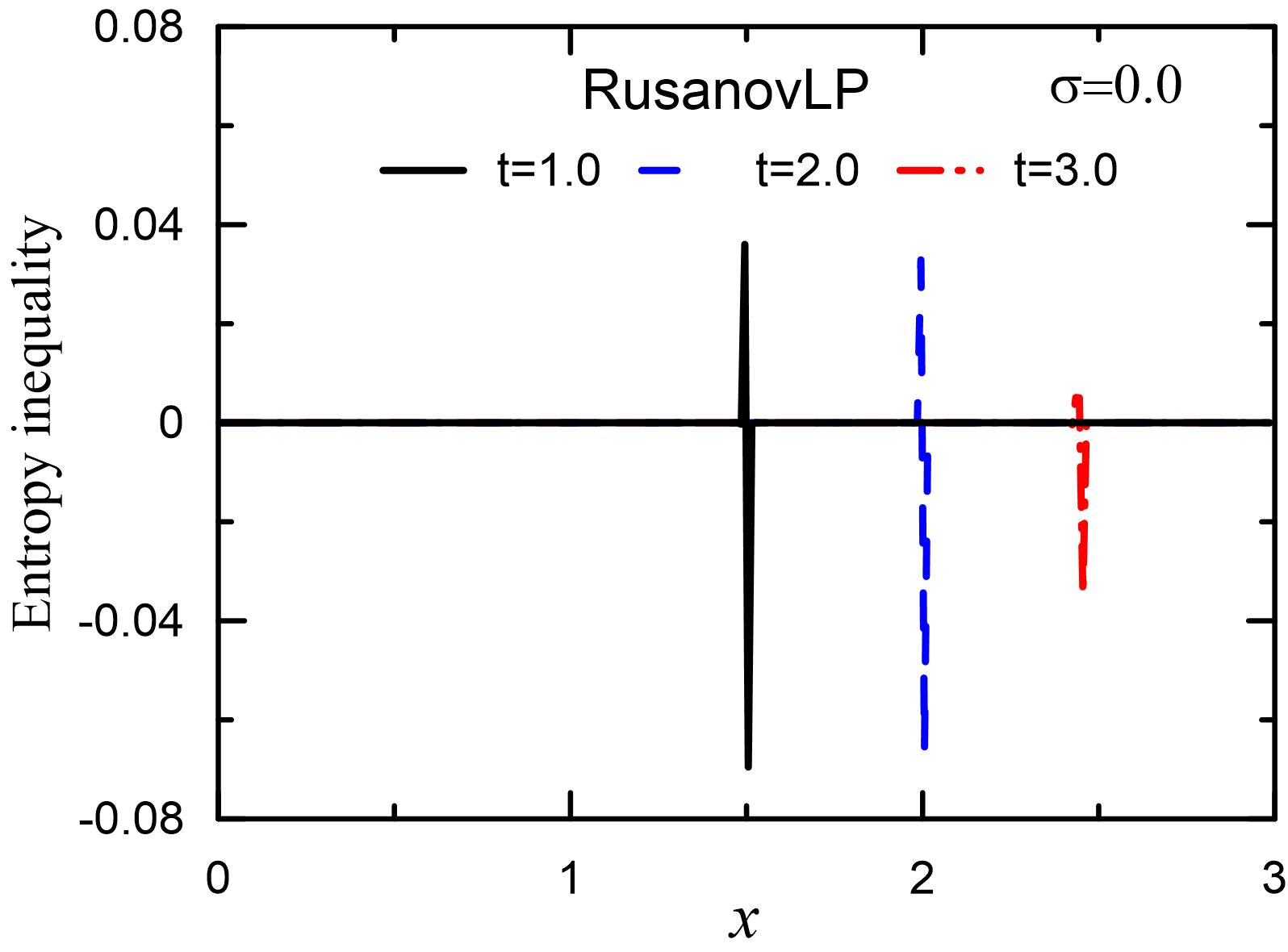}
  \includegraphics[scale=0.78]{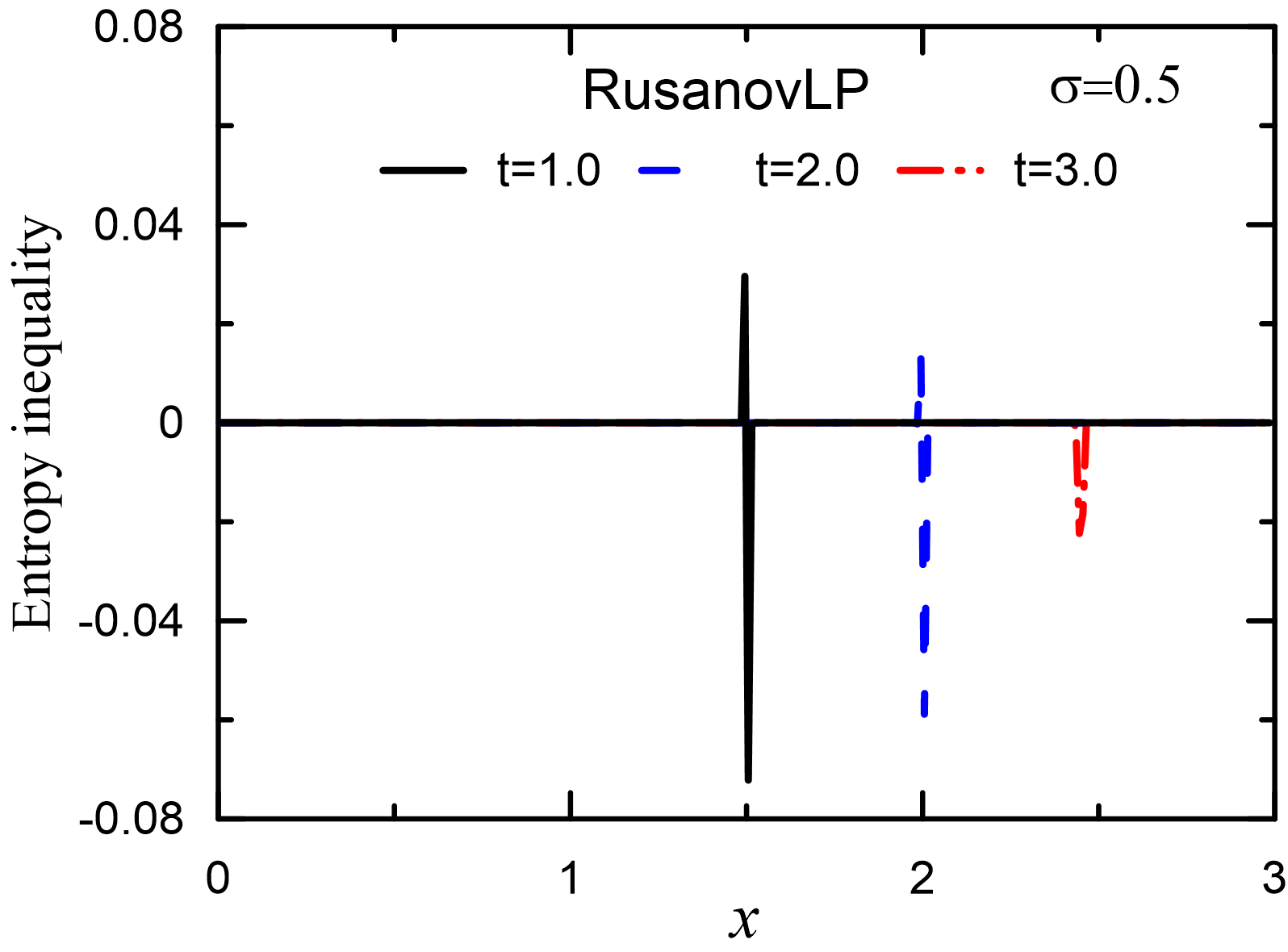}
  \includegraphics[scale=0.78]{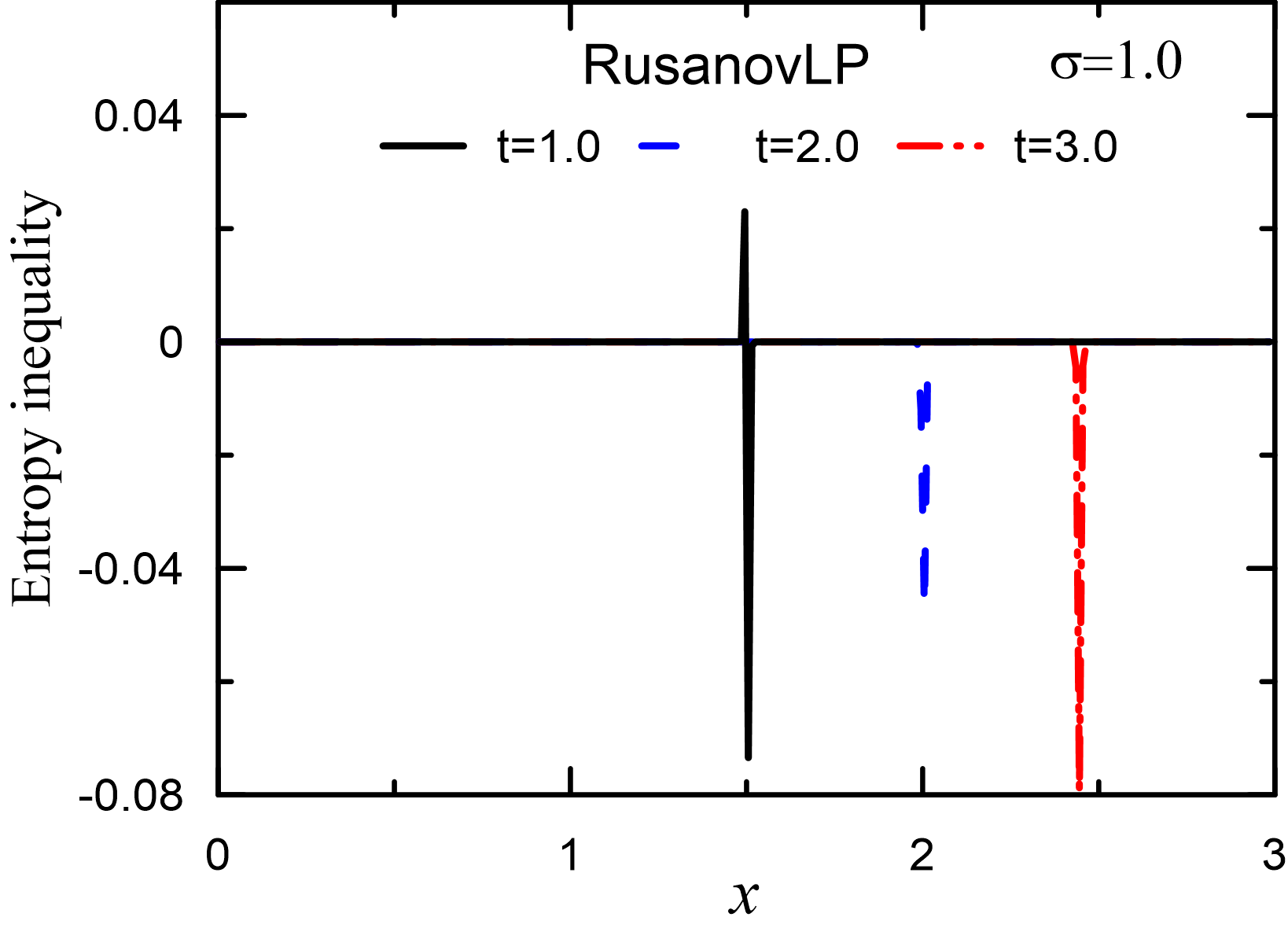}
  \includegraphics[scale=0.78]{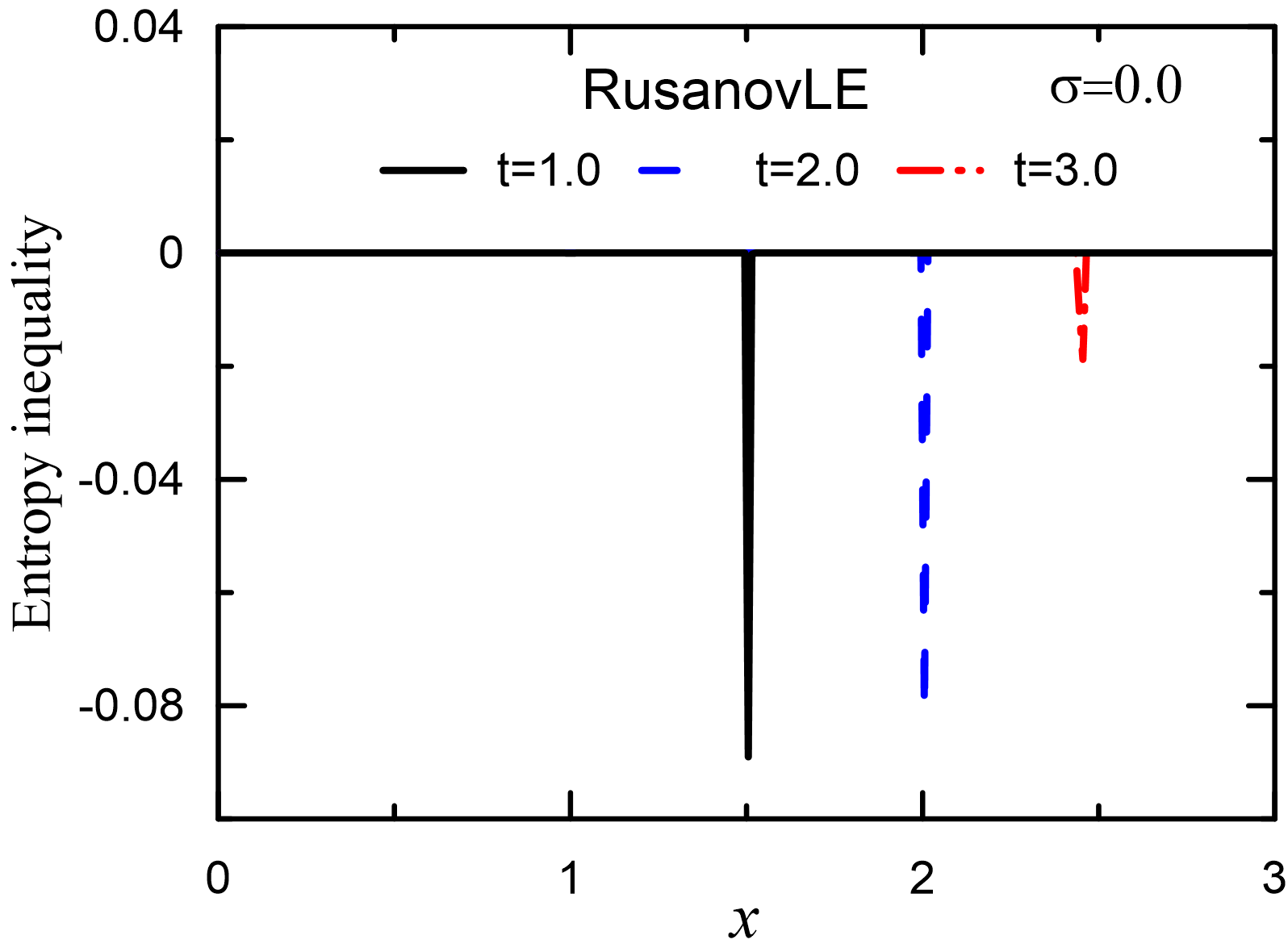}
  \includegraphics[scale=0.78]{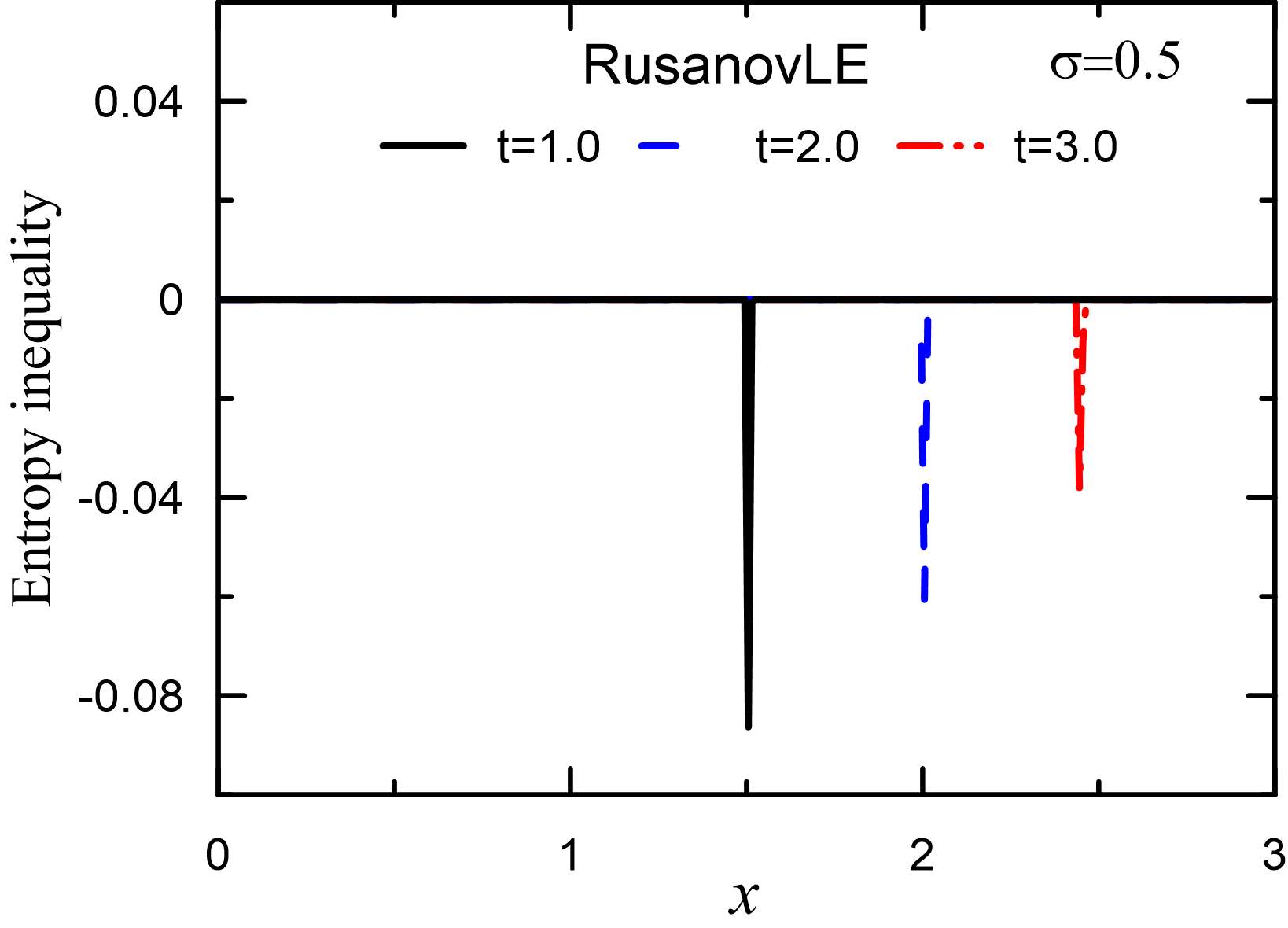}
  \includegraphics[scale=0.78]{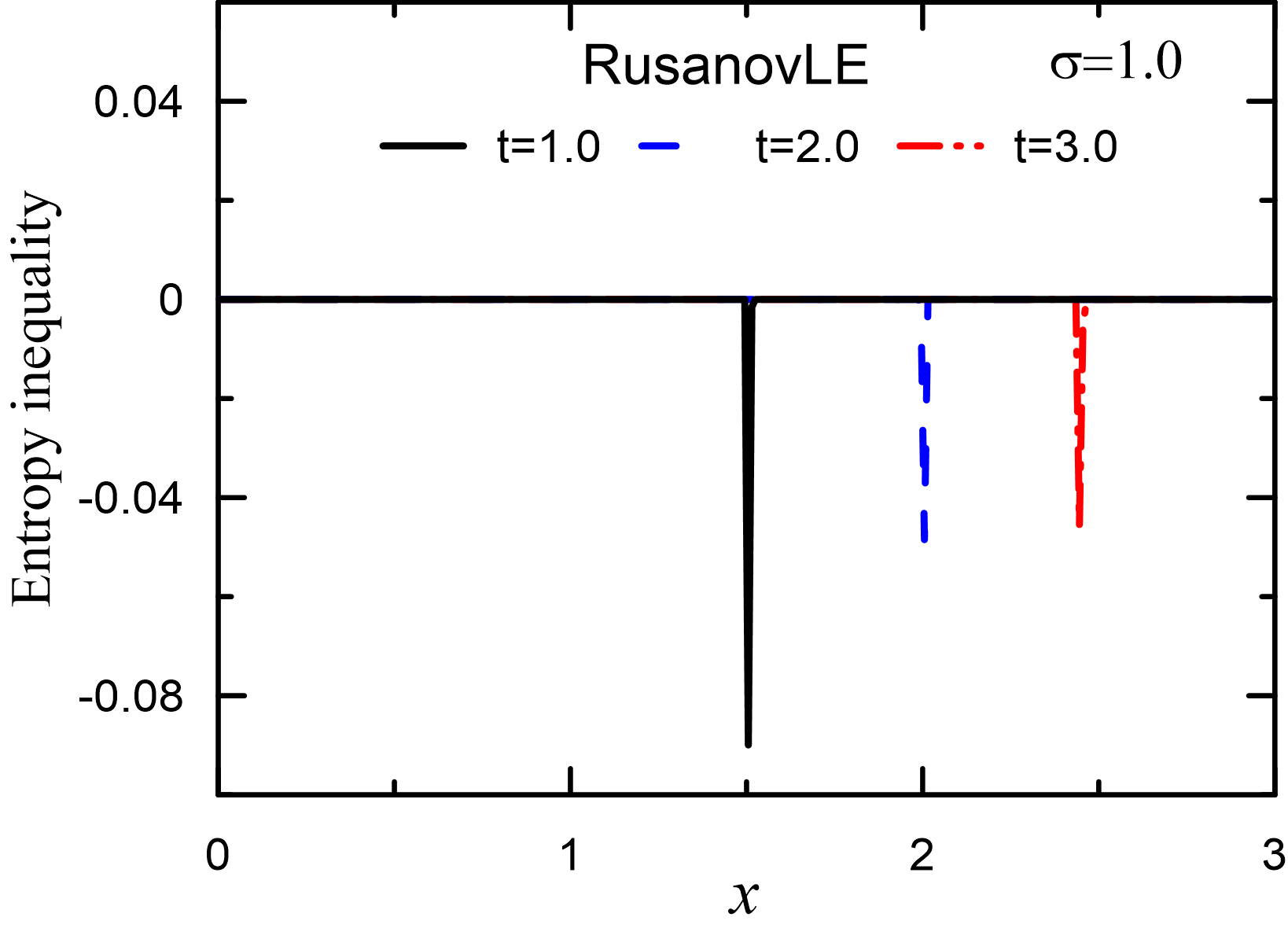}
\caption{The IVP for the inviscid Burgers equation \eqref{eq:516}-\eqref{eq:517}. Values of the discrete cell entropy inequality \eqref{eq:52} for the numerical solutions RusanovLP (upper) and RusanovLE (lower) for various $\sigma$ at different times}
\label{fig:15}       % Give a unique label
\end{figure*}
%=====================================================%

%=====================================================%
%figure 16
\begin{figure*}[!tb]
  \centering 
  \includegraphics[scale=0.8]{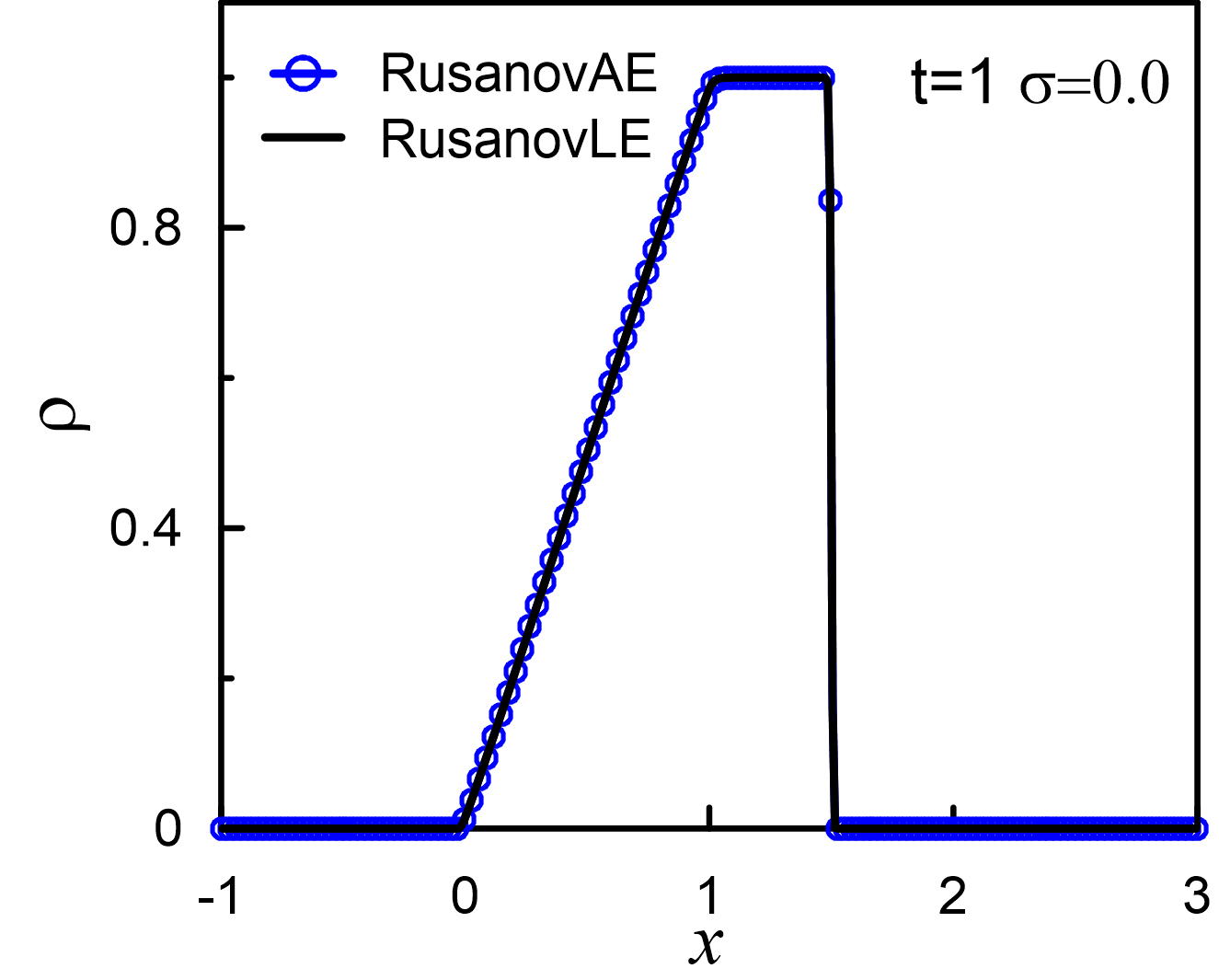}
  \includegraphics[scale=0.8]{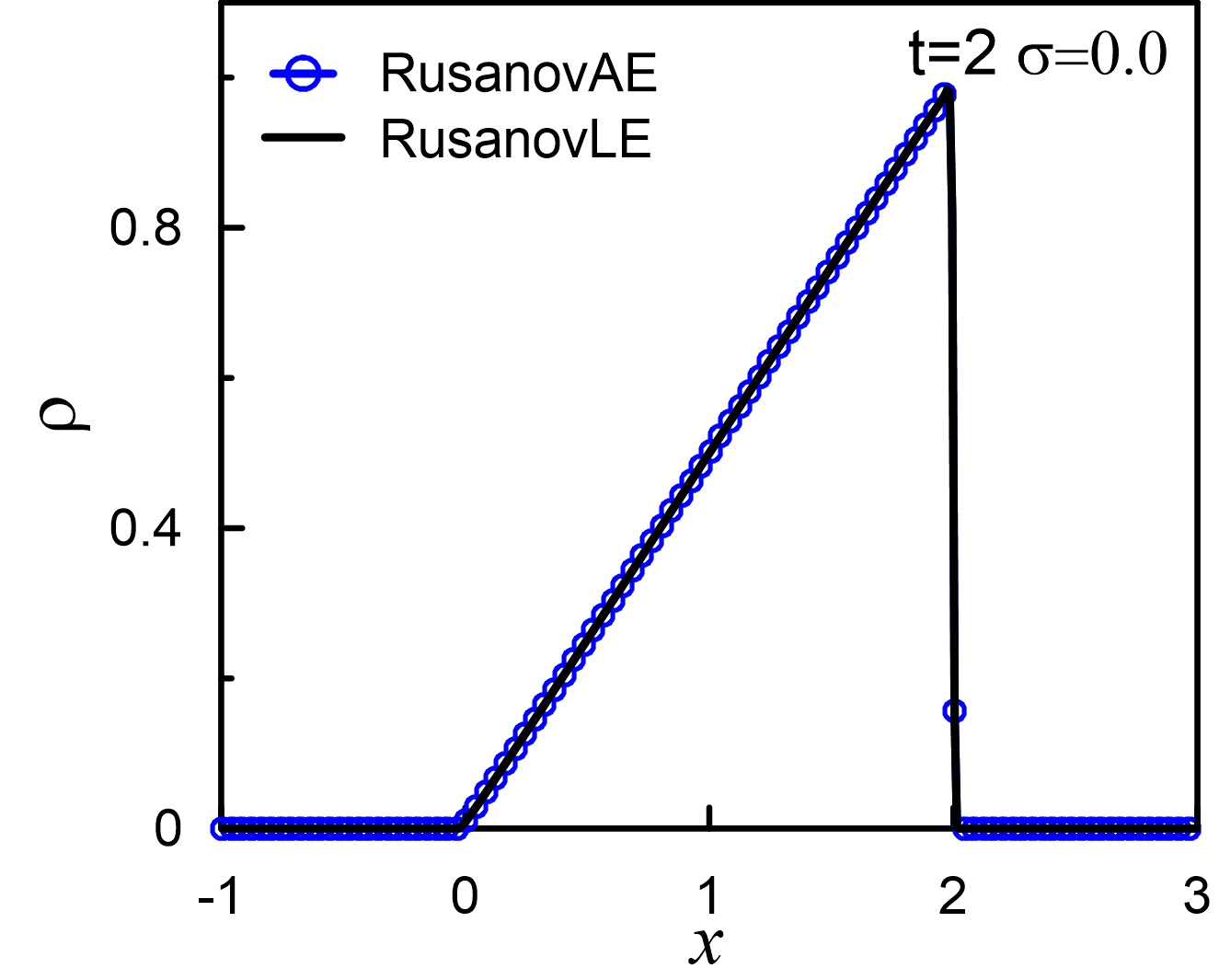}
  \includegraphics[scale=0.8]{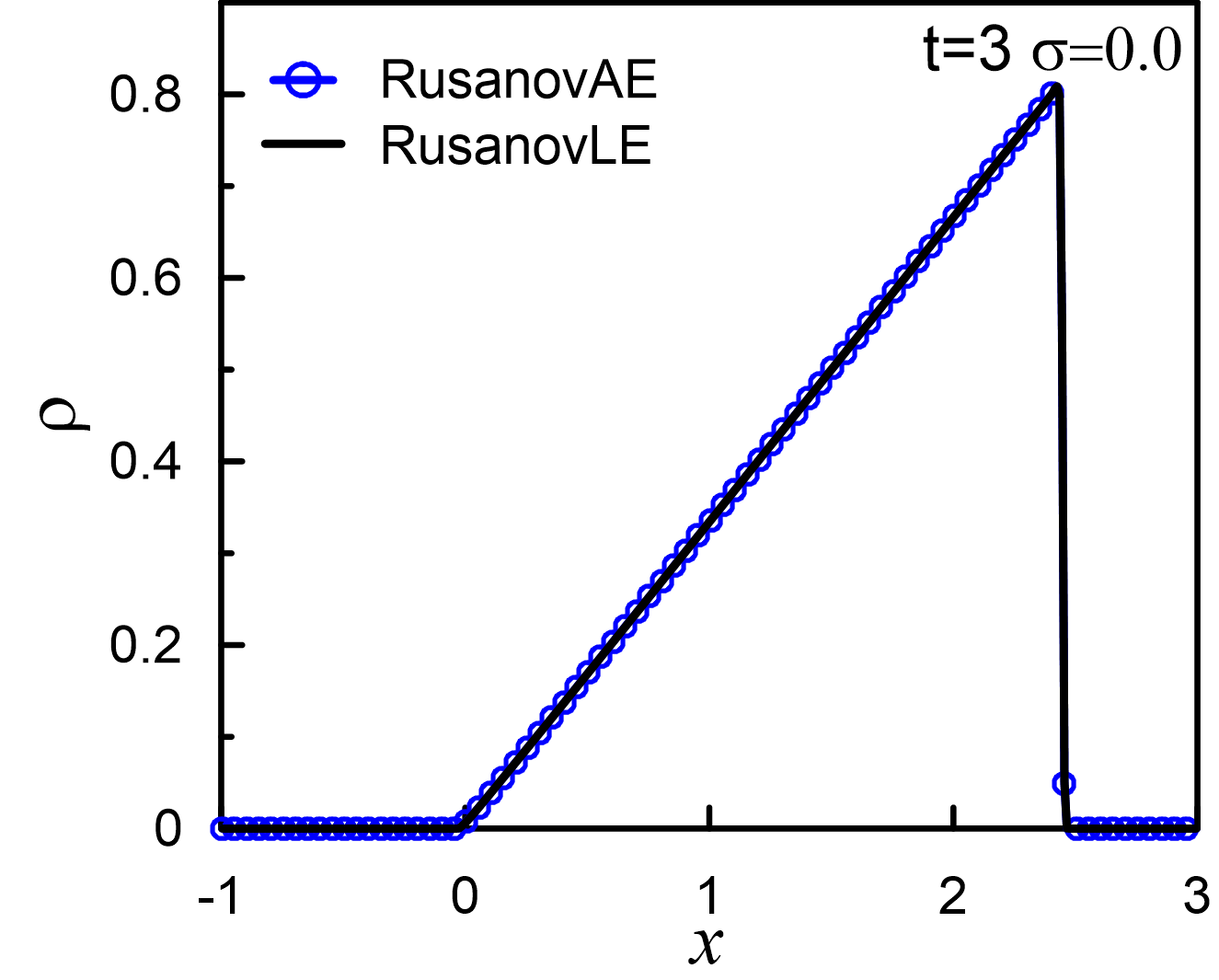}
  \includegraphics[scale=0.8]{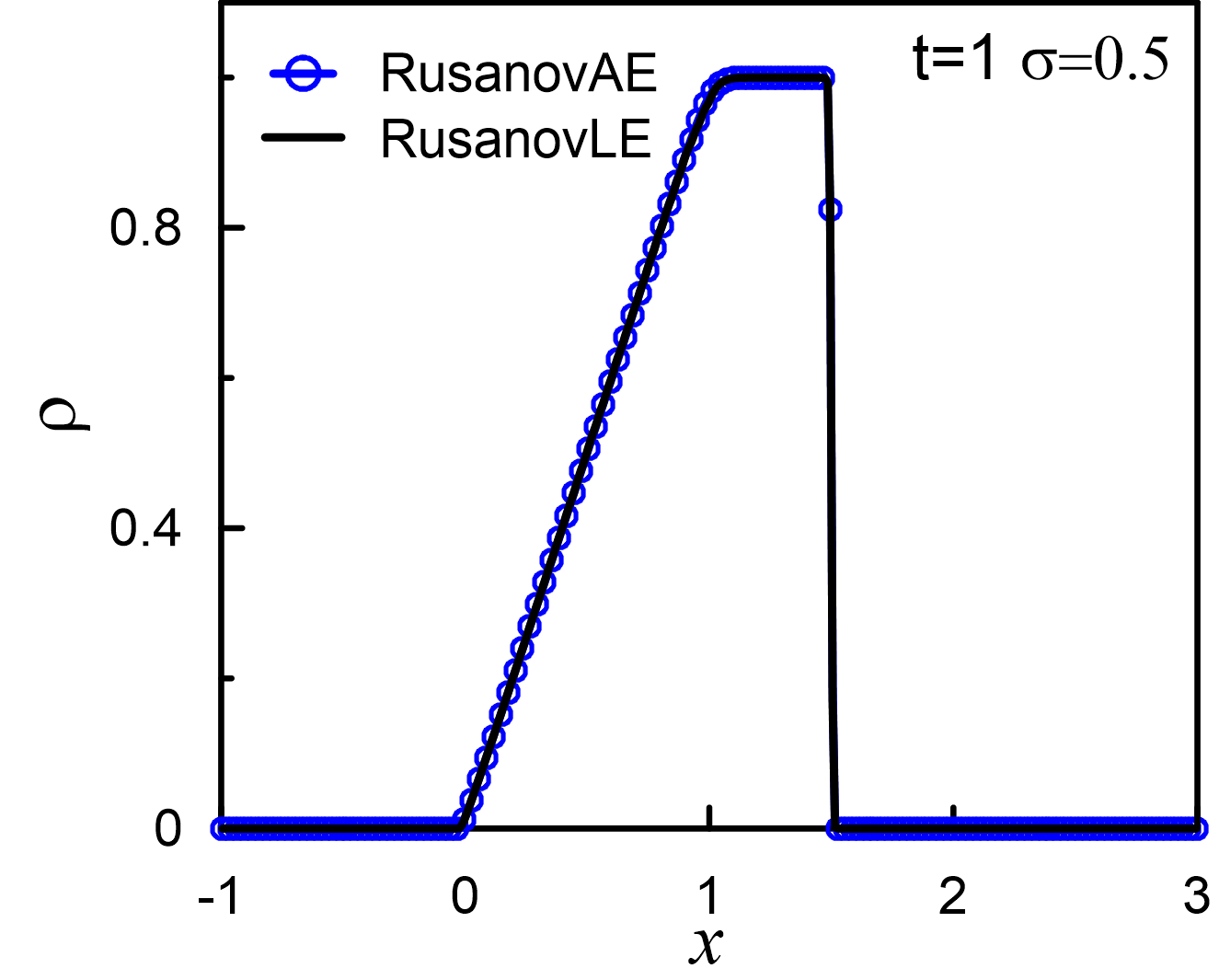}
  \includegraphics[scale=0.8]{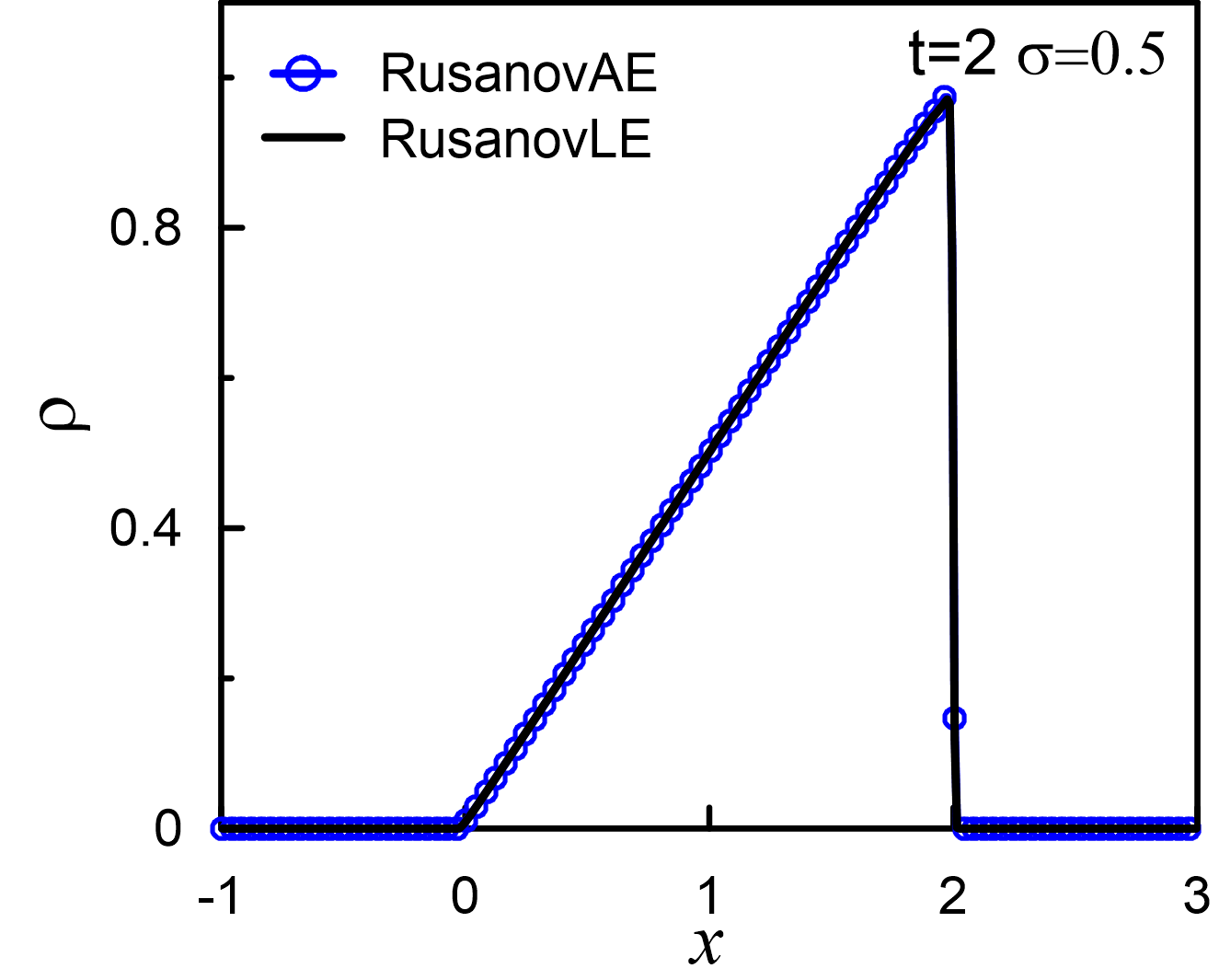}
  \includegraphics[scale=0.8]{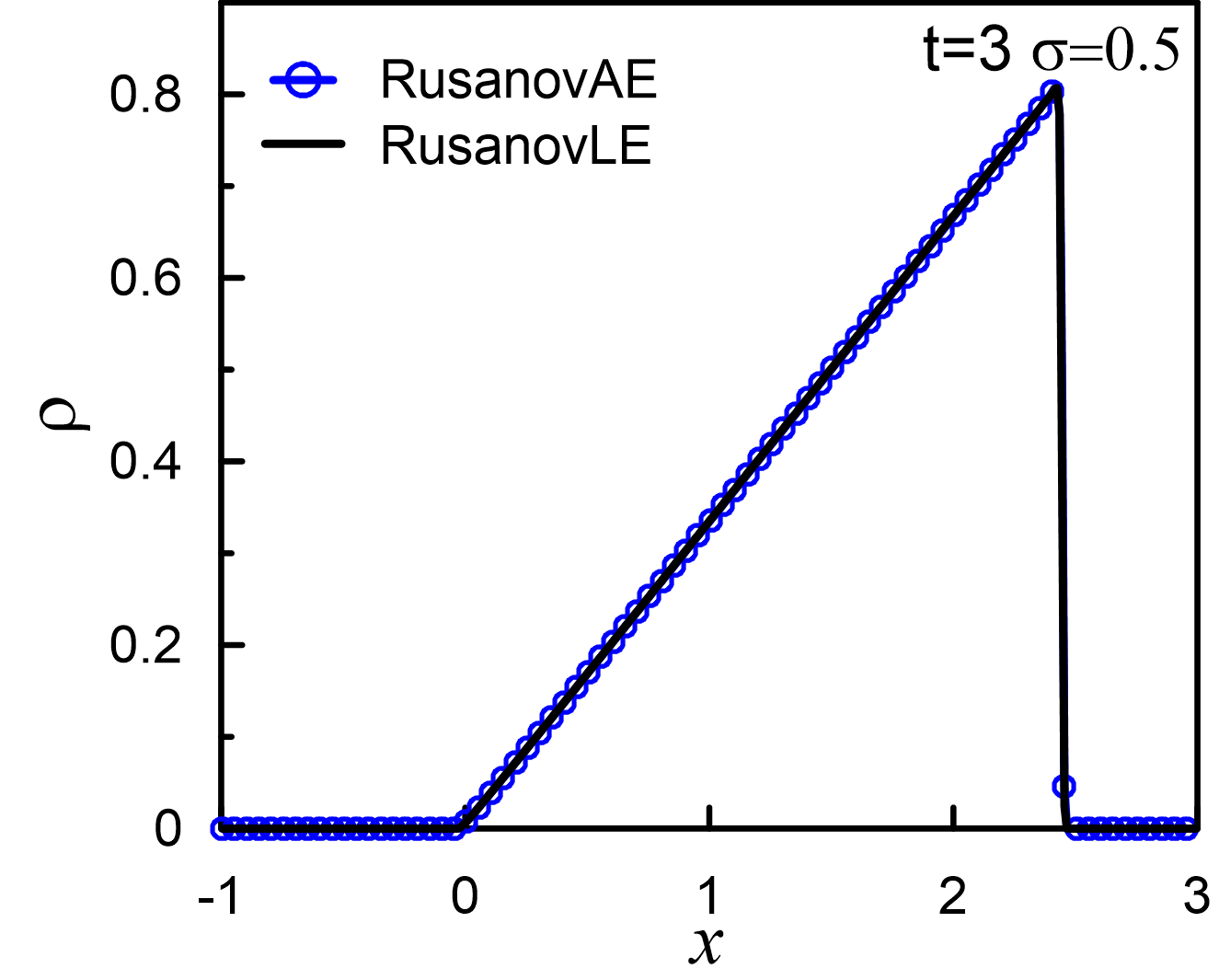}
  \includegraphics[scale=0.8]{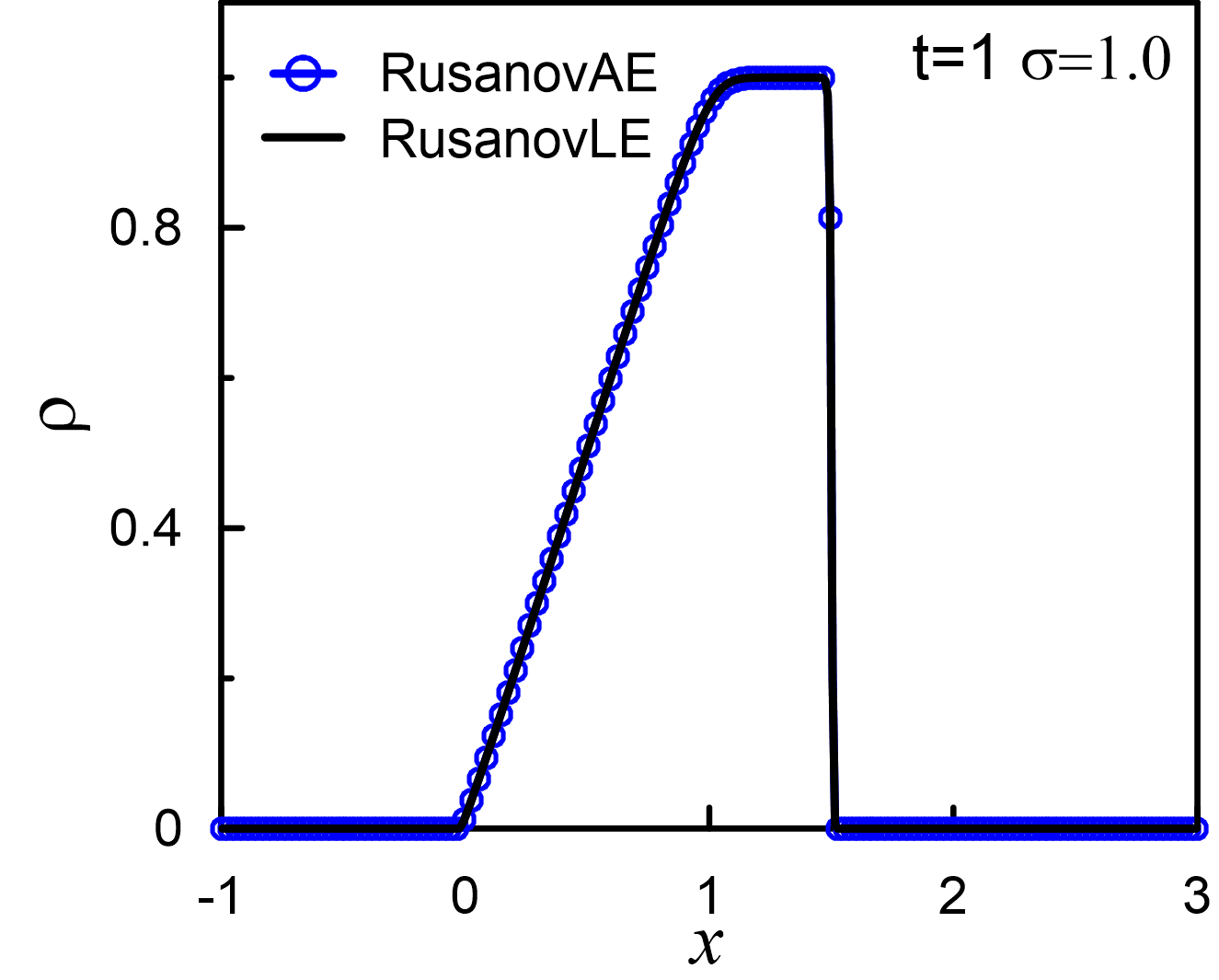}
  \includegraphics[scale=0.8]{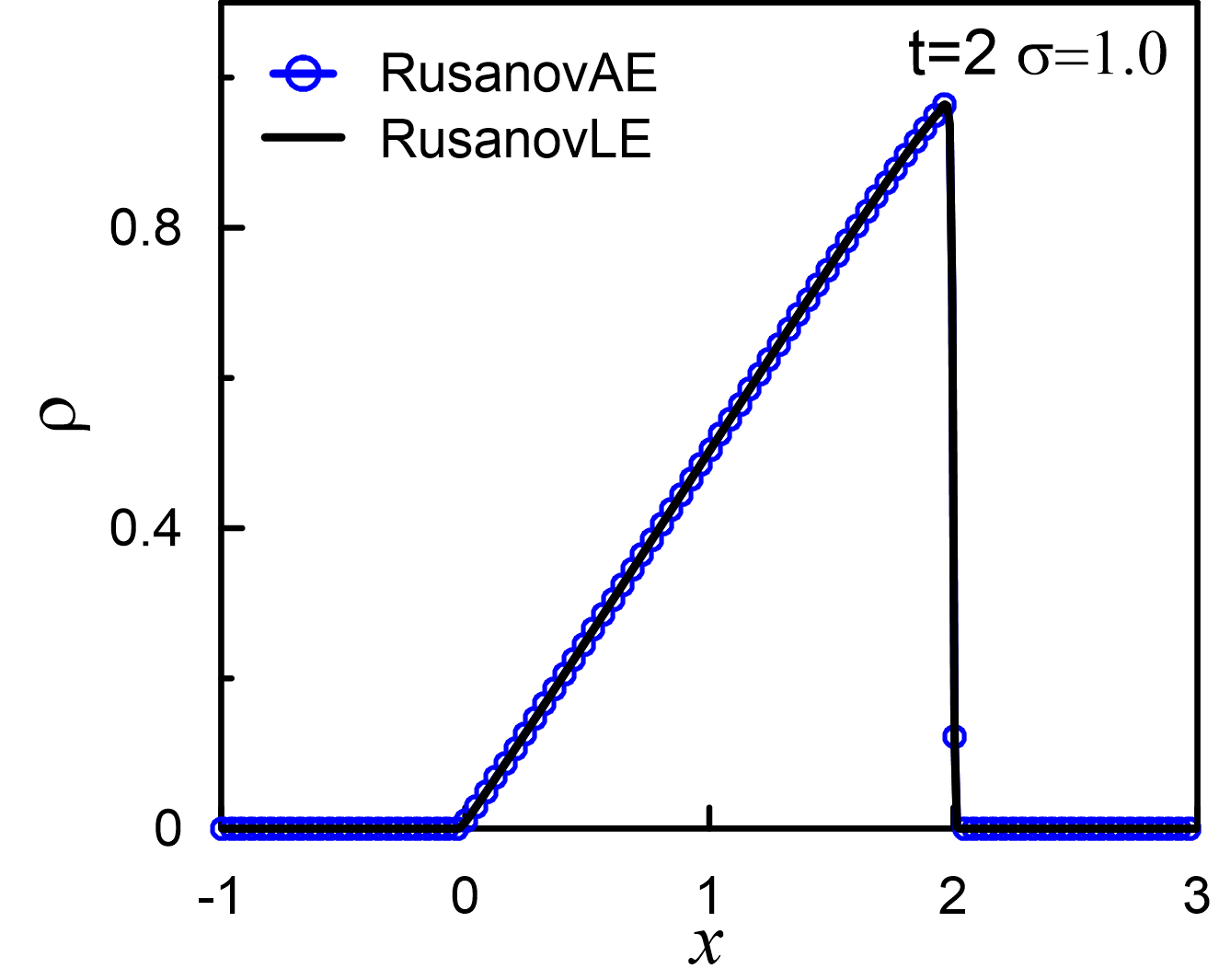}
  \includegraphics[scale=0.8]{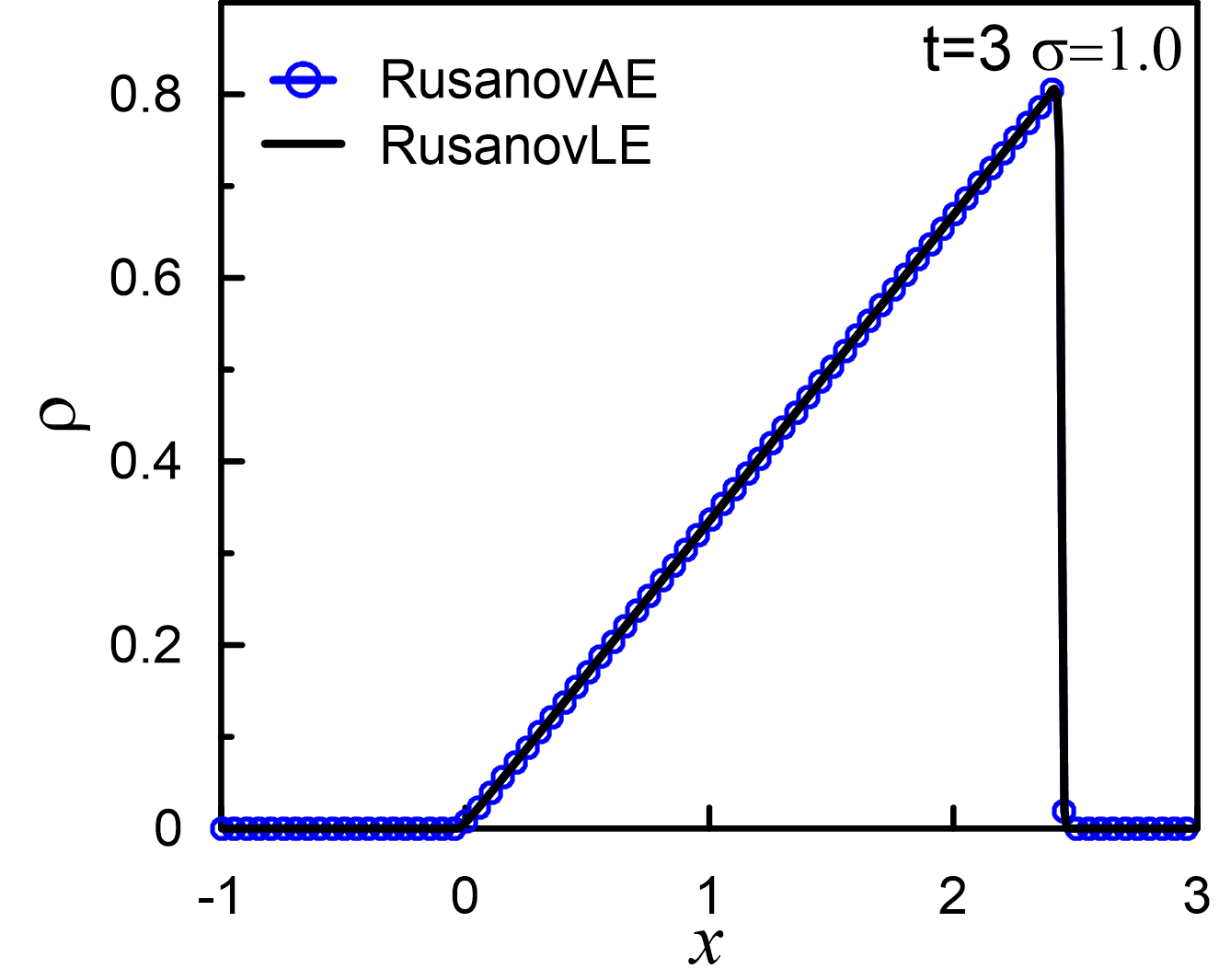}
\caption{The IVP for the inviscid Burgers equation \eqref{eq:516}-\eqref{eq:517}. Comparison of numerical results for which flux limiters are calculated from exact (RusanovLE) and approximate (RusanovAE) solutions of the linear programming problem \eqref{eq:317},\eqref{eq:53}-\eqref{eq:54}}
\label{fig:16}       % Give a unique label
\end{figure*}
%=====================================================%

%=====================================================%
%figure 17
\begin{figure*}[!tb]
  \centering 
  \includegraphics[scale=0.9]{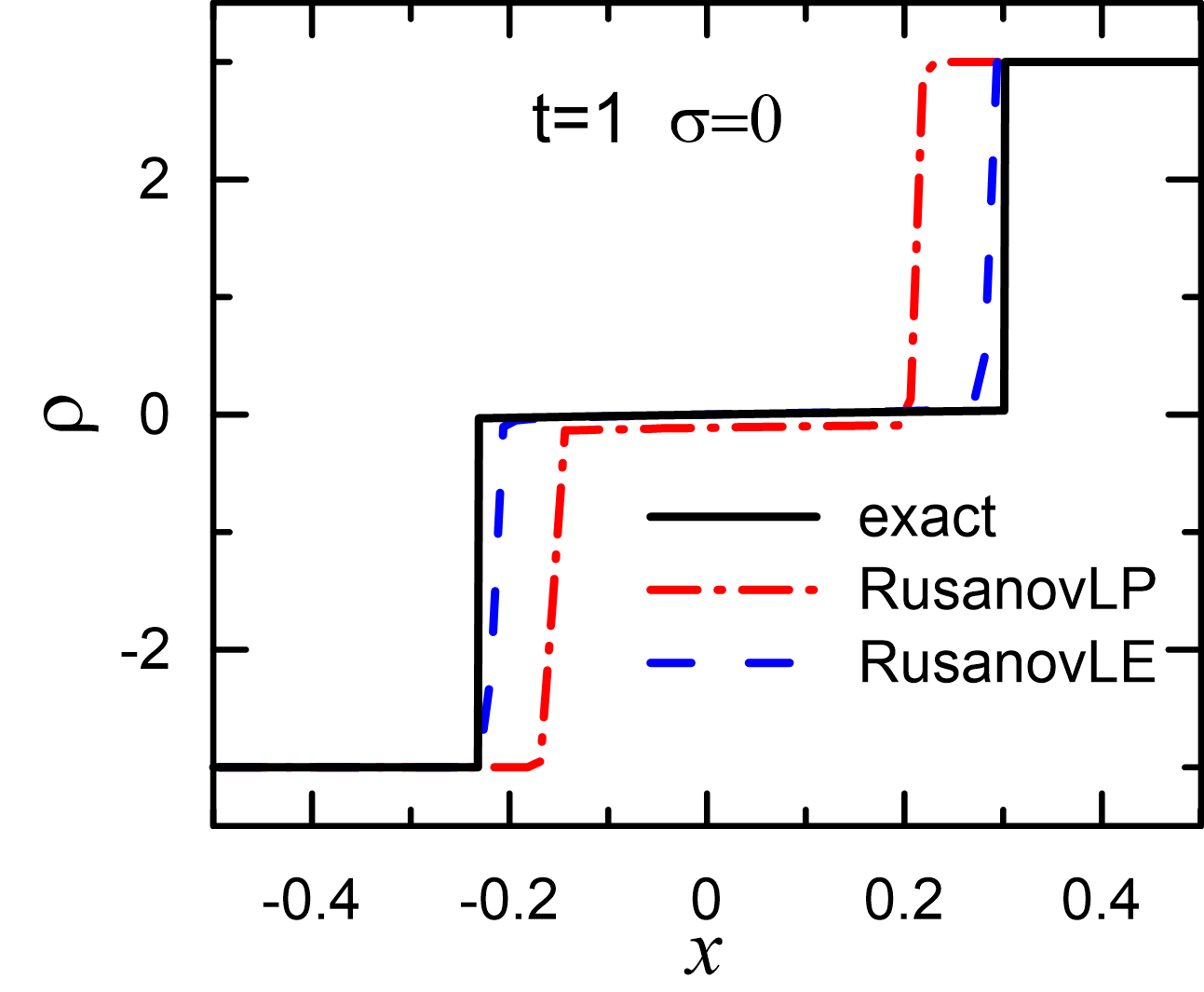}
  \includegraphics[scale=0.9]{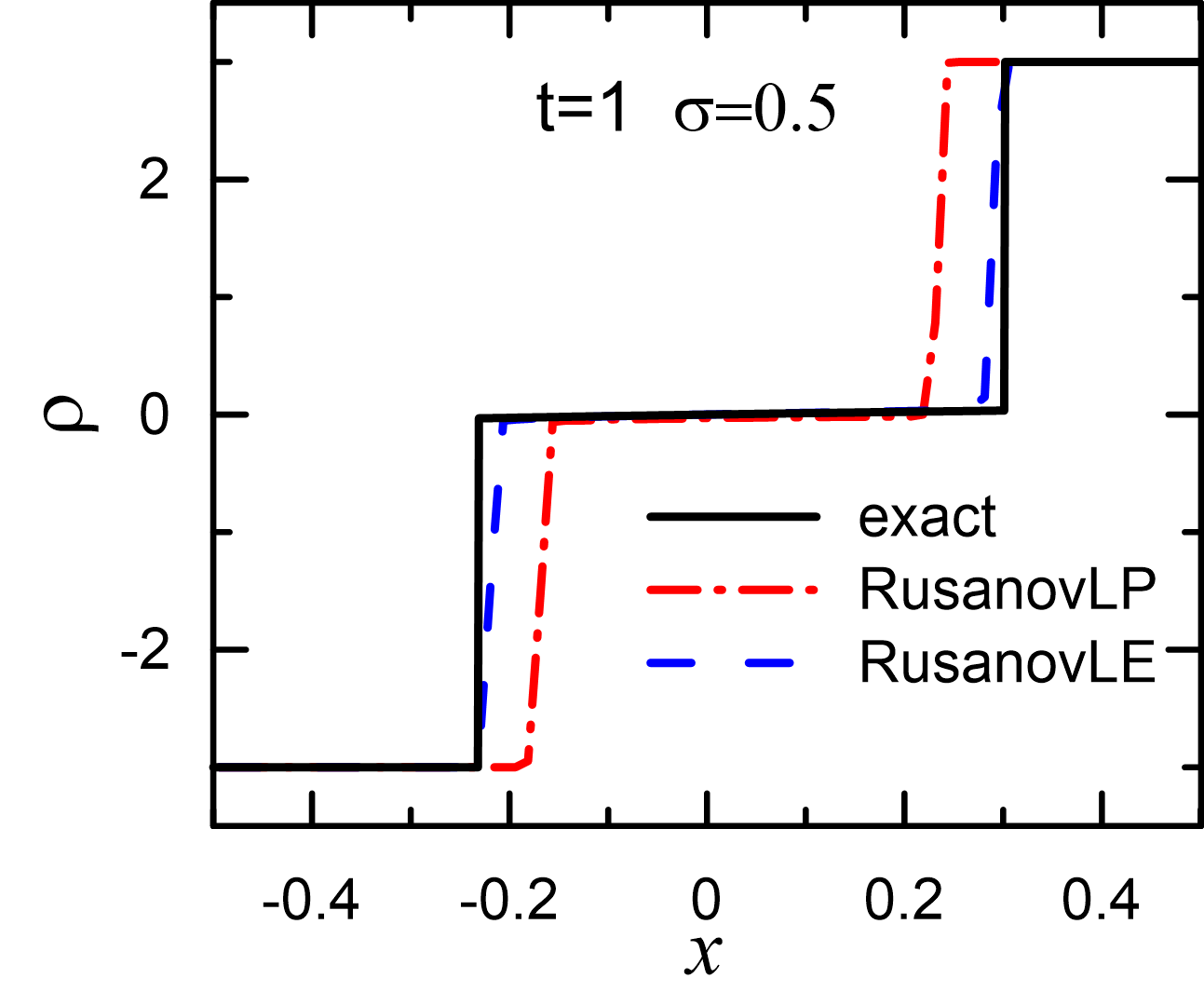}
  \includegraphics[scale=0.9]{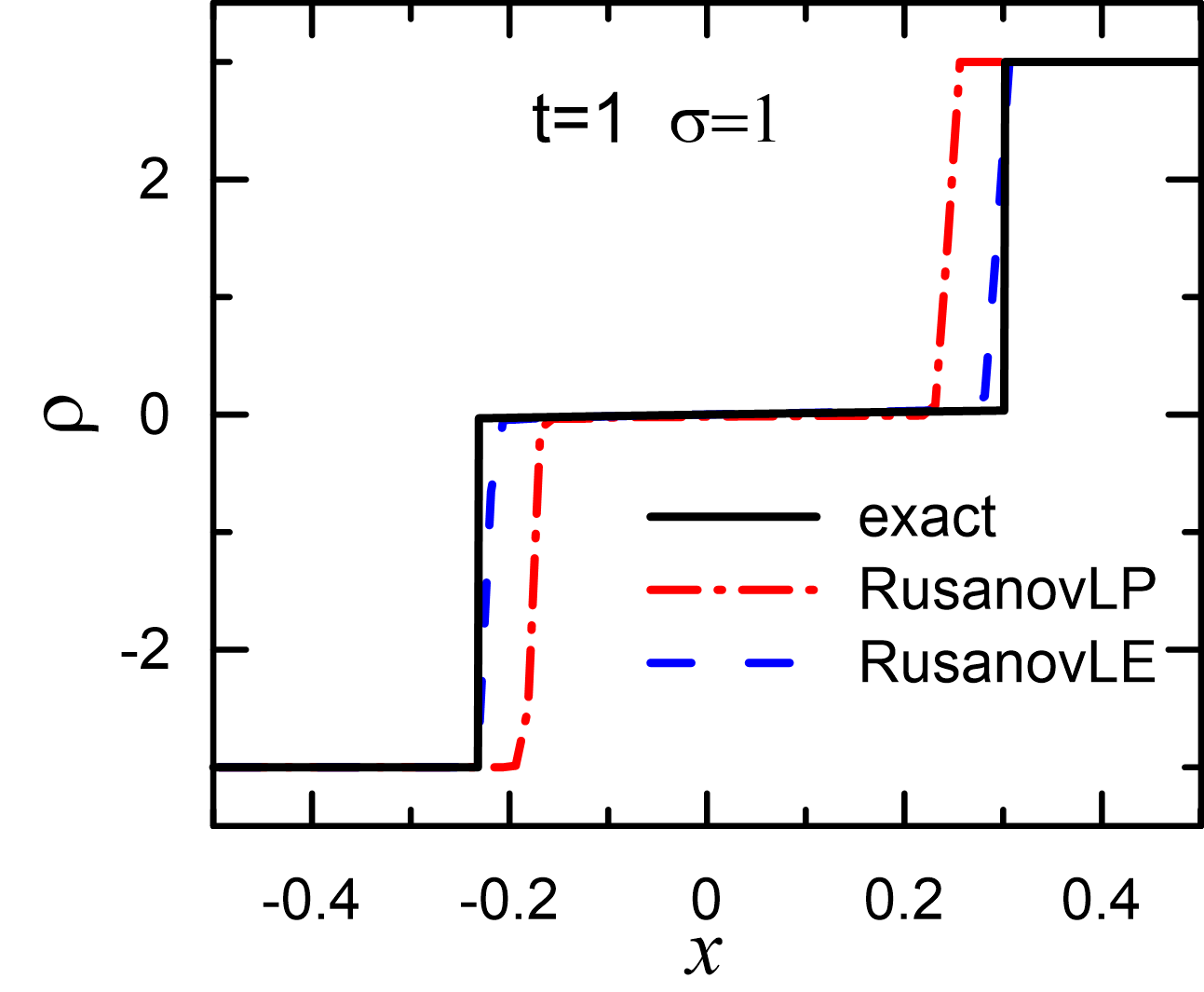}
\caption{Numerical solutions of the Riemann problem for the Buckley–Leverett equation at t=1 by using the weighted scheme \eqref{eq:51} with different $\sigma$. Flux limiters are computed by linear programming with and without taking into account the discrete entropy inequality \eqref{eq:52}}
\label{fig:17}       % Give a unique label
\end{figure*}
%=====================================================%

Numerical results of the IVP \eqref{eq:516}-\eqref{eq:517} for different values of $\sigma$ at times $t=1$, $t=2$ and $t=3$ are presented in Fig.~\ref{fig:12}. The right part of Fig.~\ref{fig:12} illustrates the enlarged fragments of the numerical results at corner points. We note that the schemes RusanovLP and RusanovLE are less diffusive than the Godunov scheme. RusanovLP solution is in good agreement with the exact solution but it is not the entropy solution at all time steps (Fig.~\ref{fig:15}). 
Fig.~\ref{fig:16} shows a good agreement between numerical results for which the flux limiters are calculated from exact (RusanovLE) and approximate (RusanovAE) solutions of the linear programming problem \eqref{eq:317},\eqref{eq:53}-\eqref{eq:54}.

%=====================================================%
%figure 18
\begin{figure*}[!tb]
  \centering 
  \includegraphics[scale=0.9]{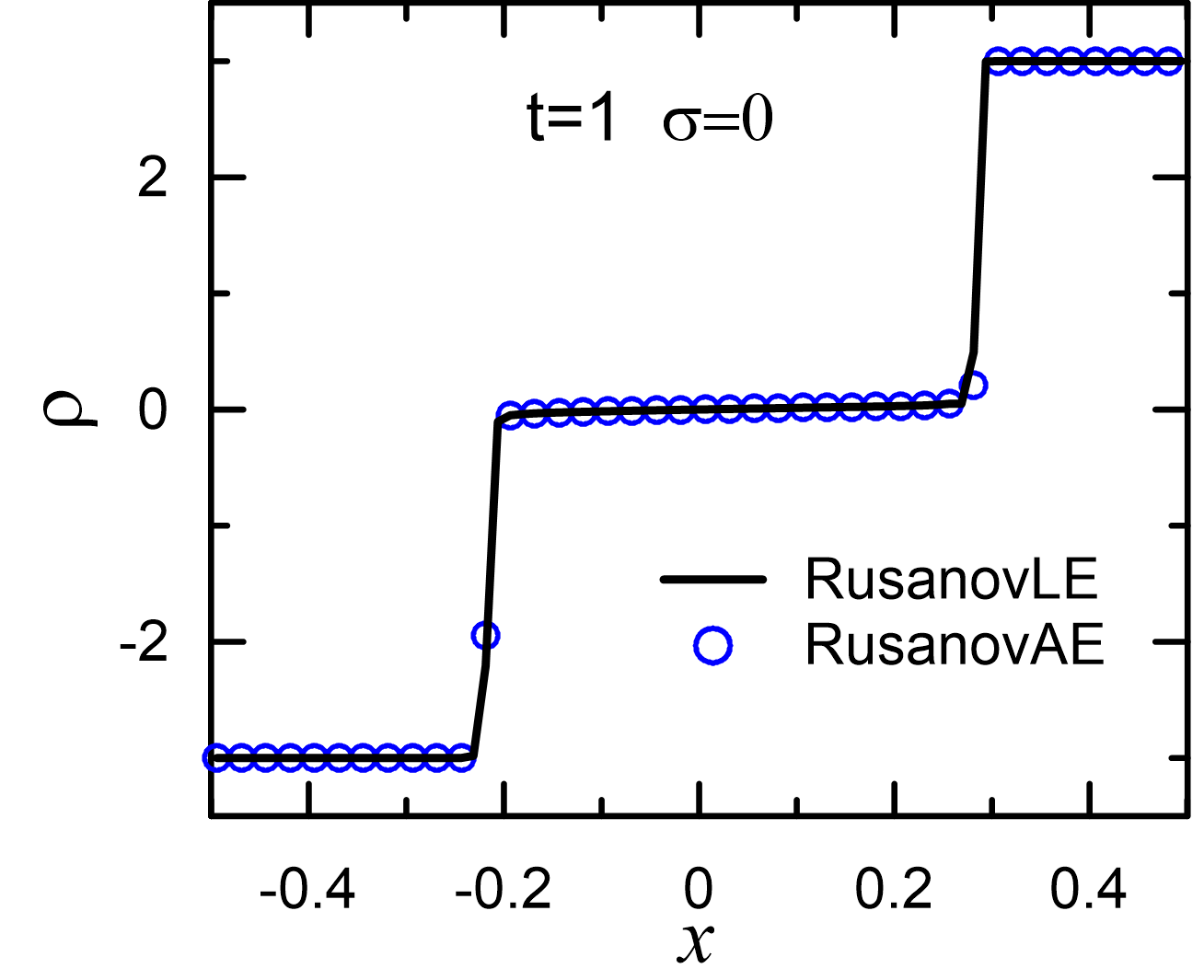}
  \includegraphics[scale=0.9]{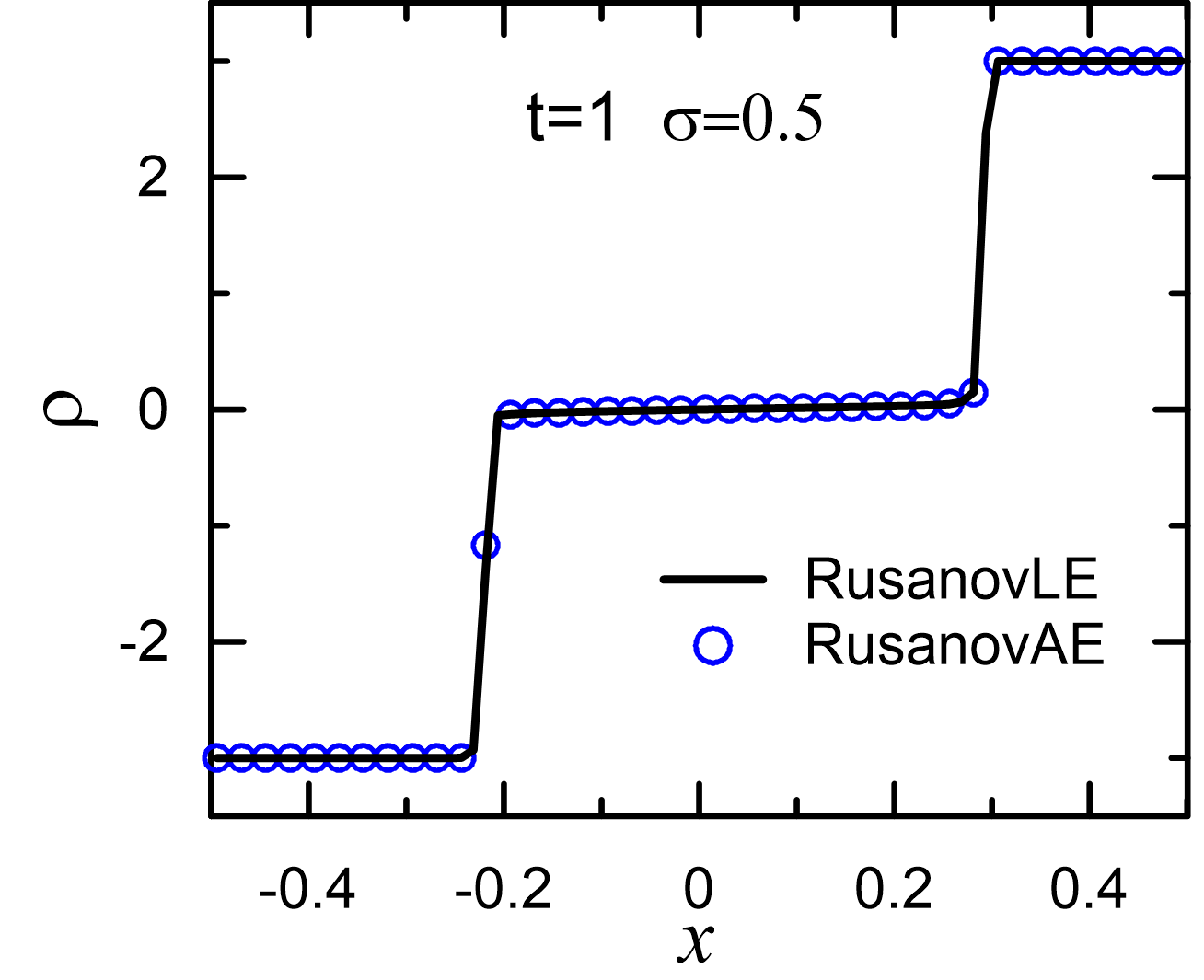}
  \includegraphics[scale=0.9]{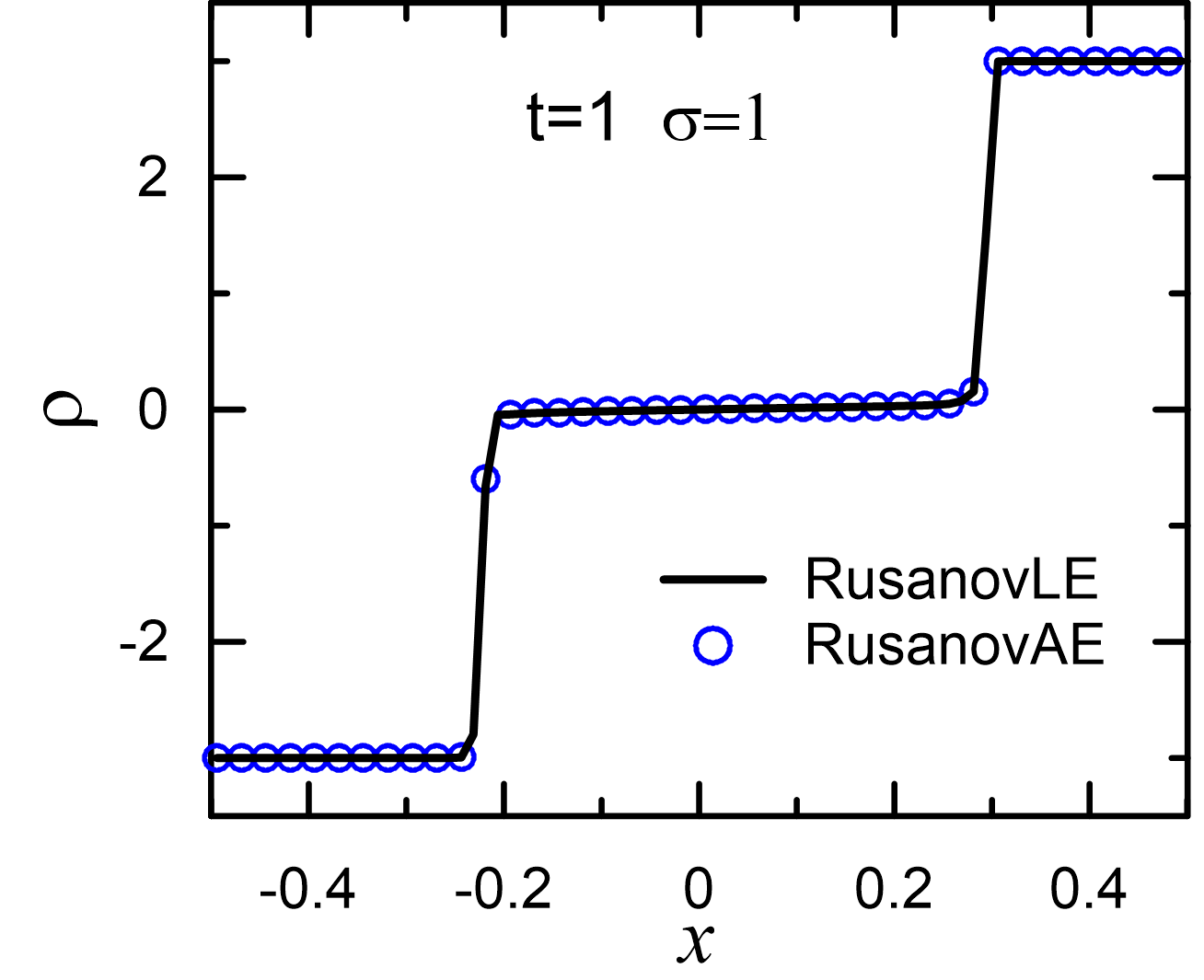}
\caption{Comparison of the numerical solutions of the Riemann problem for the Buckley-Leverett equation obtained with the weighted scheme \eqref{eq:51}. Flux limiters are computed by using exact (RusanovLE) and approximate (RusanovAE) solutions of linear programming taking into account the discrete entropy inequality \eqref{eq:52}}
\label{fig:18}       % Give a unique label
\end{figure*}
%=====================================================%

\subsection{Buckley-Leverett Equation}  \label{Sec55}

As in~\cite{b35}, we consider the following Riemann problem for the Buckley-Leverett equation
\begin{equation}
\label{eq:518a} 
  \frac{\partial \rho }{\partial t} + \frac{\partial }{\partial x}\left( {\frac{4\rho^2}{4\rho^2 + (1 - \rho )^2}} \right) = 0 
\end{equation}
subject to
\begin{equation}
\label{eq:519a} 
  \rho (x,0) = \begin{cases}
-3 \qquad & {\rm if} \;\; x < 0 \\
3  \qquad & {\rm if} \;\; x \ge 0
\end{cases}  
\end{equation}

The exact entropy solution consists of two shock waves and a rarefaction wave that is close to 0. In~\cite{b35} Chen and Shu for the square entropy function $U = \rho^2/2$ obtained non-physical solution by using the entropy stable high order discontinuous Galerkin scheme. They proved that their scheme satisfies the discrete entropy inequality with the Tadmor's numerical entropy flux. Therefore, our purpose is to repeat this test with the same entropy function for the discrete entropy inequality \eqref{eq:52} with the proper numerical entropy flux.

Like~\cite{b35}, the computational domain is [-0.5,0.5] and consists of 80 cells. The end time is $t=1$. Our numerical solutions obtained by using the conservative weighted scheme \eqref{eq:51} with different weights $\sigma$ are shown in Fig.~\ref{fig:17}. Flux limiters are calculated by using linear programming with and without taking into account the discrete entropy inequality \eqref{eq:52}.
Obviously, the numerical solution RusanovLP obtained without taking into account the discrete entropy inequality \eqref{eq:52} and corresponding to the classical FCT solution is not physically correct.

Comparison of the numerical solutions RusanovLE and RusanovAE is presented in Fig.~\ref{fig:18}. Flux limiters for RusanovLE and RusanovAE are calculated by using exact and approximate solutions of linear programming taking into account the discrete entropy inequality \eqref{eq:52}, respectively. We note a good agreement of these numerical solutions.

%=====================================================%
%figure 19
\begin{figure*}[!t]
  \centering 
  \includegraphics[scale=0.99]{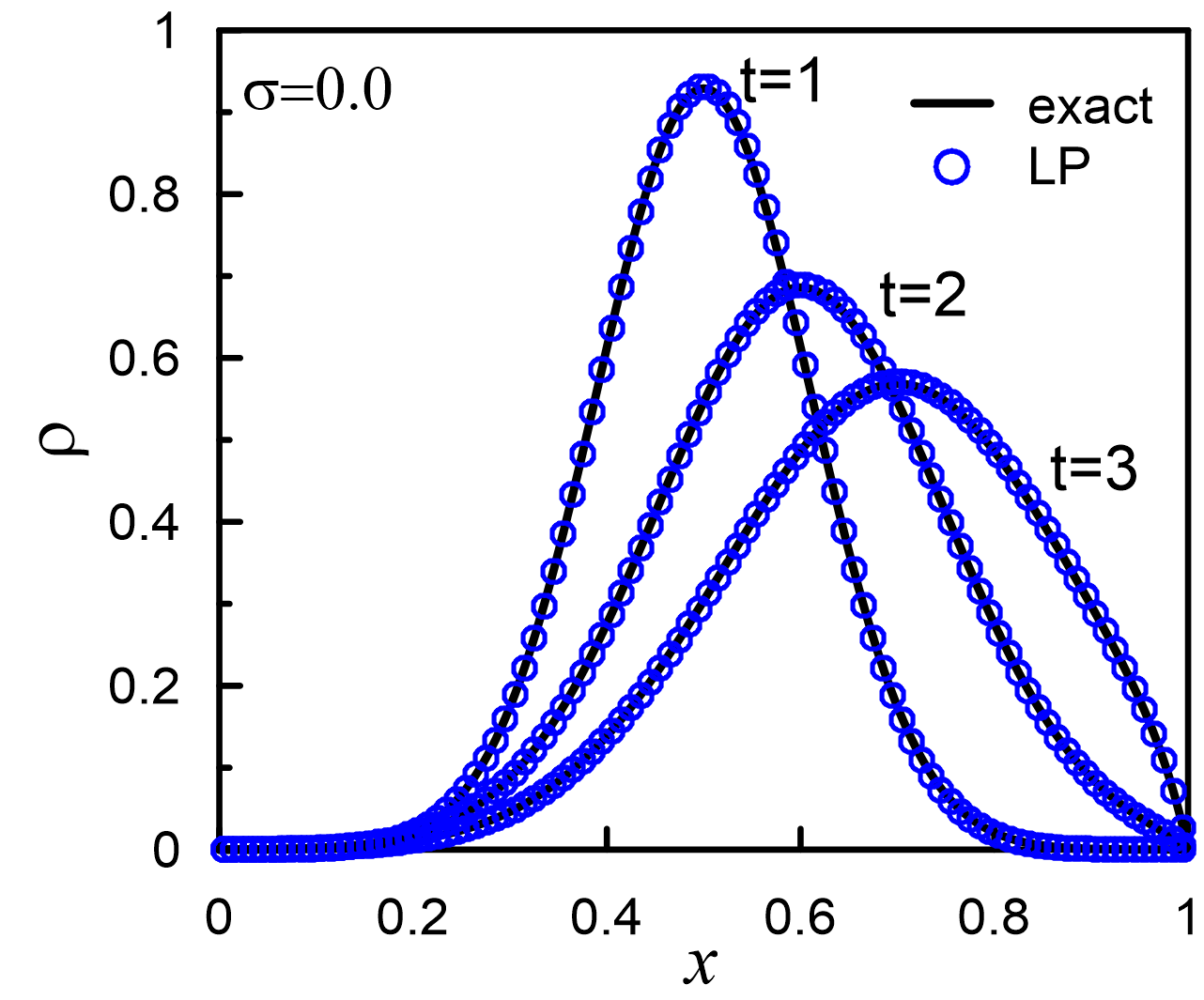}
  \includegraphics[scale=0.99]{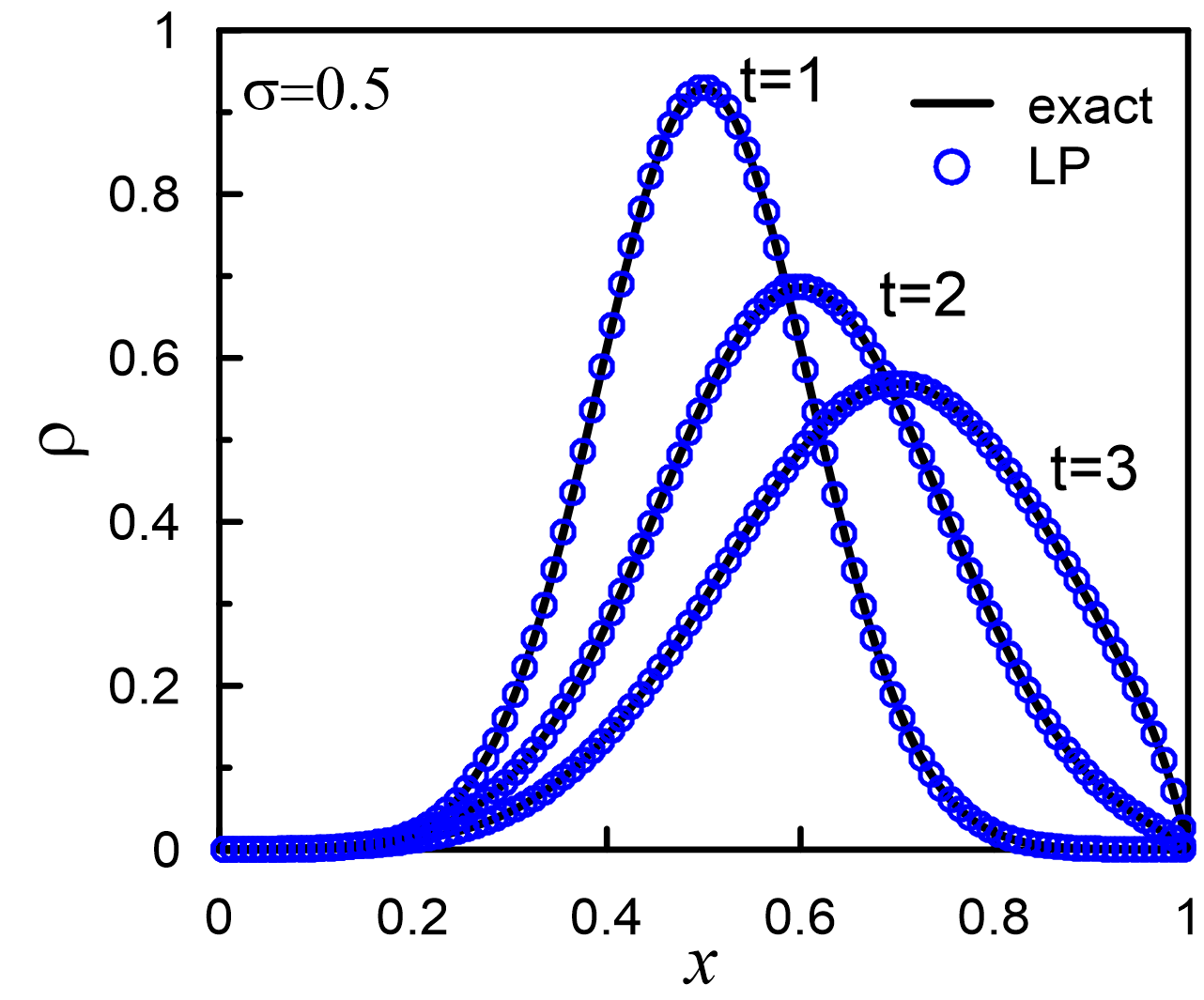}
  \includegraphics[scale=0.99]{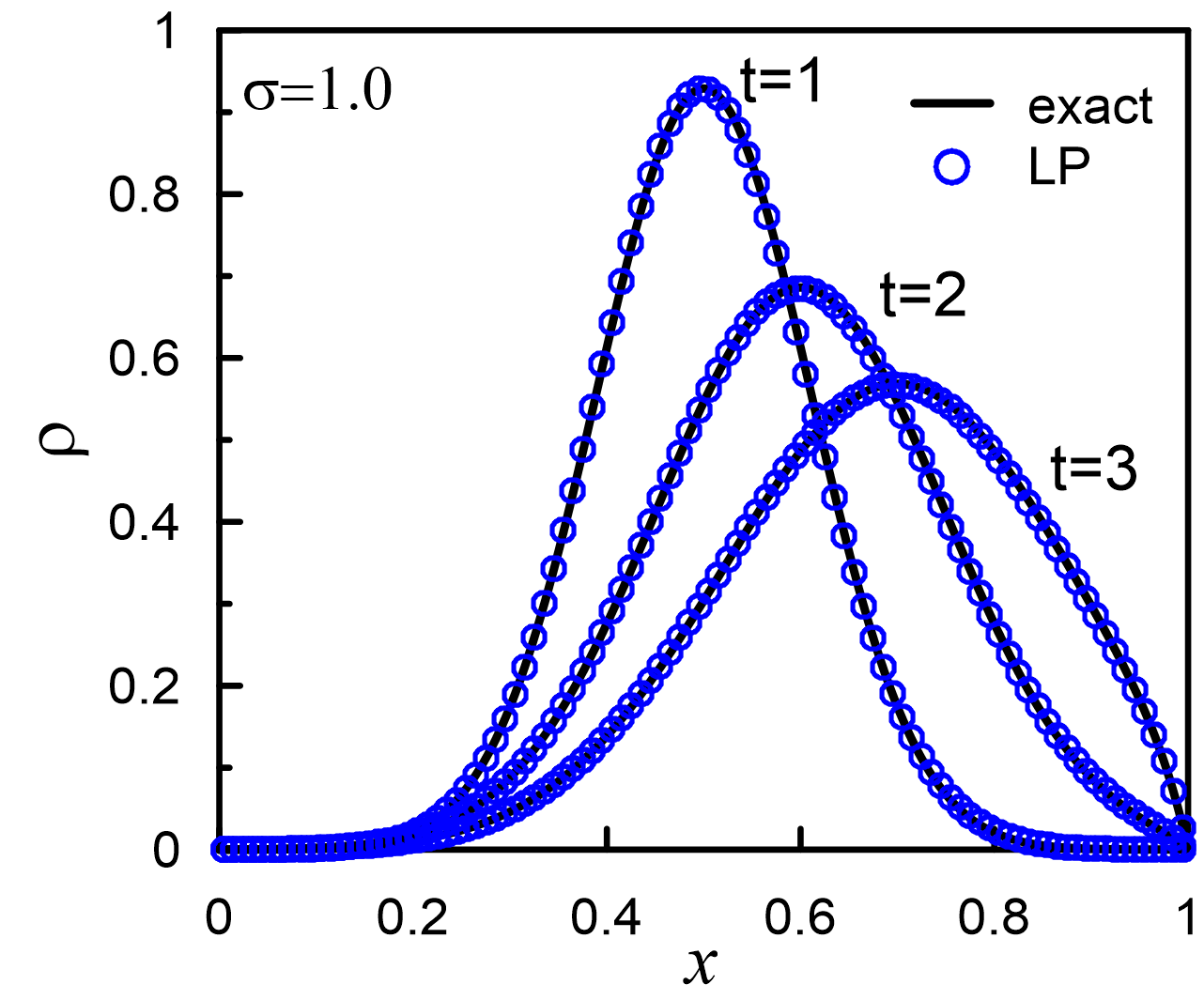}
\caption{Numerical solutions of the IBVP for the convection-diffusion equation \eqref{eq:518}-\eqref{eq:521} obtained by using the scheme \eqref{eq:45}-\eqref{eq:411} with various weight $\sigma$ at different times. Flux limiters are computed from exact solutions of the linear programming problems}
\label{fig:19}       % Give a unique label
\end{figure*}
%=====================================================%

\begin{table*}[!t]
\caption{\label{tab5} $L^1$-norm of errors and maximum values of numerical solutions of the IBVP for the convection-diffusion equation \eqref{eq:518}-\eqref{eq:521} obtained with the weighted scheme \eqref{eq:43}-\eqref{eq:44}. Flux limiters are calculated by using exact (LP) and approximate (AP) solutions of linear programming problems}
\centering
%\begin{tabular}{@{}p{2.45cm} p{1.2cm} *{2}{ p{2.55cm} p{1.5cm}@{}}}
\begin{tabular}{@{}p{1.0cm} p{1.3cm} p{1.3cm} p{2.55cm} p{1.5cm} p{2.55cm} p{1.5cm} @{}}
\hline 
 & exact &  & \multicolumn{2}{c}{LP} & \multicolumn{2}{c}{AP}  \\
 \cline{2-2} \cline{4-5} \cline{6-7}
 t & $y_{max}$ & $\sigma$  &  $L^1$ error & $y_{max}$ & $L^1$ error & $y_{max}$ \\[3pt] \hline
  \multirow{3}{*}{1}  
  &	& 0.0 &	1.1765$\times 10^{-3}$ & 0.93110 & 1.1765$\times 10^{-3}$ & 0.93110 \\
  & 0.92883 & 0.5 &	5.4850$\times 10^{-4}$ & 0.92946 & 5.4850$\times 10^{-4}$ & 0.92946 \\
  &	& 1.0 &	1.9557$\times 10^{-3}$ & 0.92837 & 1.9557$\times 10^{-3}$ & 0.92837 \\
\hline	
  \multirow{3}{*}{2}
  &	& 0.0 &	1.1313$\times 10^{-3}$ & 0.68877 & 1.1313$\times 10^{-3}$ & 0.68877 \\
  & 0.68602 & 0.5 &	4.2613$\times 10^{-4}$ & 0.68663 & 4.2613$\times 10^{-4}$ & 0.68663 \\
  &	& 1.0 &	1.5076$\times 10^{-3}$ & 0.68475 & 1.5076$\times 10^{-3}$ & 0.68475 \\
\hline	
  \multirow{3}{*}{3}
  &	& 0.0 &	1.0828$\times 10^{-3}$ & 0.57126 & 1.0828$\times 10^{-3}$ & 0.57126 \\
  & 0.56863 & 0.5 &	3.6479$\times 10^{-4}$ & 0.56917 & 3.6479$\times 10^{-4}$ & 0.56917 \\
  &	& 1.0 &	1.1901$\times 10^{-3}$ & 0.56723 & 1.1901$\times 10^{-3}$ & 0.56723 \\
\hline	
\end{tabular}
\end{table*}

\subsection{One-Dimensional Convection-Diffusion Equation}  \label{Sec56}

We consider the initial-boundary value problem (IBVP) for the convection-diffusion equation with constant coefficients
\begin{equation}
\label{eq:518} 
  \frac{\partial \rho}{\partial t} + u\frac{\partial \rho}{\partial x} = \varepsilon \frac{{\partial^2}\rho }{\partial {x^2}},	\qquad	x \in \left[ {0,1} \right],\quad t > 0						
\end{equation}
\begin{equation}
\label{eq:519} 
  \rho ( x,0 ) = \begin{cases}
2 \sin \left( {5 \pi (x - 0.3)} \right), \quad & {\rm if} \;\; 0.3 \le x \le 0.5 \\
0, \quad & \rm{otherwise}
\end{cases} 					
\end{equation}

\begin{equation}
\label{eq:520} 
  \rho (0,t) = 0											
\end{equation}
\begin{equation}
\label{eq:521} 
  \rho (1,t) = 0											
\end{equation}

In Fig.~\ref{fig:19} we present the numerical solutions of the IBVP \eqref{eq:518}-\eqref{eq:521} for $u = 0.1$ and $\varepsilon  = 0.005$  at different times. All simulations are performed on the uniform grid with step size $\Delta x_i=0.01$ and $\Delta t=0.01$. In the simulations, we apply the scheme \eqref{eq:43}-\eqref{eq:44} with various weights $\sigma$, the flux limiters of which are calculated from exact (LP) and approximate (AP) solutions  of linear programming problems. Comparison of $L^1$-norm of errors and maximum values of the numerical solutions LP and AP are given in Table~\ref{tab4}. Note that in this case, the exact and approximate solutions of the linear programming problems yield similar numerical results.

\section{Conclusions}  \label{Sec6}
In this paper, we present the design of the formulas for calculating flux limiters of FCT methods for scalar hyperbolic conservation laws and convection-diffusion equations. Following the FCT approach, we consider  a hybrid scheme which is a linear combination of monotone and high-order schemes. The difference between high-order flux and low-order flux is considered as an antidiffusive flux. The finding maximal flux limiters for the  antidiffusive fluxes is treated as an optimization problem with a linear objective function. Constraints for the optimization problem are inequalities that are valid for the monotone scheme and applied to the hybrid scheme. The discrete entropy inequality with the proper numerical entropy flux is used to single out physically correct solutions for scalar hyperbolic conservation laws. Approximate solutions of linear programming problems are applied for computing flux limiters. This approach allows us to reduce the classical two-step FCT to a single-step one for explicit difference schemes and to design flux limiters with desired properties.

The numerical experiments in subsection \ref{Sec53} show that the discrete entropy inequality with the Tadmor’s numerical entropy flux does not guarantee to obtain physically correct solutions for scalar hyperbolic conservation laws. We also note a good agreement between the numerical results for which the flux limiters are computed by using exact and approximate solutions of linear programming problems.

%%Vancouver style references.
\bibliographystyle{model1-num-names}
\bibliography{refs}

\end{document}